\documentclass[10pt,amssymb,amsfonts,psfig]{amsart}
\usepackage{mathrsfs,epsfig,amsfonts,amssymb,color,amsmath,amscd,url,bm}
\usepackage[curve,color,graph]{xypic}
\usepackage[letterpaper,asymmetric,left=1.48in,right=1.48in,top=1.48in,bottom=1.48in,bindingoffset=0.0in]{geometry}

\newtheorem{thm}{Theorem}[section]
\newtheorem{cor}[thm]{Corollary}
\newtheorem{lem}[thm]{Lemma}
\newtheorem{prop}[thm]{Proposition}

\newtheorem{rem}[thm]{Remark}
\newtheorem{problem}[thm]{Problem}
\theoremstyle{remark}

\theoremstyle{definition}

\theoremstyle{plain}

\newtheorem*{thmC}{Global Lee-Yang-Fisher Current Theorem}
\newtheorem*{thmD}{Equidistribution Theorem for the DHL}
\newtheorem*{thmE}{Equidistribution Theorem}
\newtheorem*{thmENDOS}{Equidistribution for Endomorphisms \cite{FAVRE_JONSSON,DS}}
\newtheorem*{thmBIRAT}{Equidistribution for Birational Mappings \cite{FAVRE_GUEDJ,FAVRE_THESIS,DILLER}}

\newtheorem*{LY Theorem}{Lee-Yang Theorem}
\newtheorem*{LY Theorem2}{General Lee-Yang Theorem}
\newtheorem*{LY TheoremBC}{Lee-Yang Theorem with Boundary Conditions}
\newtheorem*{LY Theorem2BC}{General Lee-Yang Theorem with Boundary Conditions}
\newtheorem*{conj}{Conjecture}

\newcommand{\FOLDEXPONENT}{r}
\newcommand{\BADSET}{D_{\geq d}}
\newcommand{\BADSETCORE}{\mathcal{N}_{\geq d}}
\newcommand{\VERYBADSET}{D_{> d}}
\newcommand{\BISET}{I^-(f)}

\numberwithin{equation}{section}
\numberwithin{figure}{section}

\newcommand{\Zpart}{\mathsf{Z}} 

\newcommand{\PLapl}{\frac i{\pi} \di\dibar} 
\newcommand{\Regular}{N} 
\newcommand{\LYF}{\mathcal{S}^c} 
\newcommand{\Green}{S} 
\newcommand{\current}{T} 
\newcommand{\FoverT}{F^\#} 
\newcommand{\hatFoverT}{\hat{F}^\#} 

\newcommand{\Rmig}{R} 
\newcommand{\Cmig}{C} 
\newcommand{\Cmigbl}{C_0}
\newcommand{\Cmigtl}{C_1}
\newcommand{\TOPmig}{\mathrm{T}} 
\newcommand{\BOTTOMmig}{\mathrm{B}} 
\newcommand{\FIXmig}{b} 
\newcommand{\CFIXmig}{e} 
\newcommand{\INDmig}{a} 

\newcommand{\CROSSING}{c} 

\newcommand{\Rphys}{\mathcal R} 
\newcommand{\Cphys}{\mathcal C} 
\newcommand{\Solidmig}{SC}

\newcommand{\Secmig}{P}

\newcommand{\BOTTOMphys}{\mathcal B} 
\newcommand{\FIXphys}{\beta} 
\newcommand{\CFIXphys}{\eta} 
\newcommand{\Imig}{G} 

\newcommand{\Line}{L}
\newcommand{\Lzero}{L_0}
\newcommand{\Lone}{L_1}
\newcommand{\Ltwo}{L_2}
\newcommand{\Lthree}{L_3}
\newcommand{\Lfour}{L_4}

\newcommand{\ex}{{\mathrm{exc}}}

\font\nt=cmr7

\def\note#1
{\marginpar
{\nt $\leftarrow$
\par
\hfuzz=20pt \hbadness=9000 \hyphenpenalty=-100 \exhyphenpenalty=-100
\pretolerance=-1 \tolerance=9999 \doublehyphendemerits=-100000
\finalhyphendemerits=-100000 \baselineskip=6pt
#1}\hfuzz=1pt}

\newcommand{\correspond}{\Psi}

\newcommand{\di}{\partial}
\newcommand{\dibar}{\bar\partial}

\newcommand{\ra}{\rightarrow}

\def\ssk{\smallskip}
\def\msk{\medskip}
\def\bsk{\bigskip}

\def\sm{\smallsetminus}

\newcommand{\tl}{\tilde}

\newcommand{\vol}{\operatorname{vol}}

\newcommand{\supp}{\operatorname{supp}}

\newcommand{\area}{\operatorname{area}}

\newcommand{\isom}{\approx}

\def\loc{{\mathrm{loc}}}

\newcommand{\eps}{{\epsilon}}

\newcommand{\De}{{\Delta}}
\newcommand{\de}{{\delta}}
\newcommand{\la}{{\lambda}}

\newcommand{\Om}{{\Omega}}
\newcommand{\om}{{\omega}}

\newcommand{\EE}{{\mathcal E}}

\newcommand{\FF}{{\mathcal F}}
\newcommand{\GG}{{\mathcal G}}

\newcommand{\HH}{{\mathcal H}}

\newcommand{\SSS}{{\mathcal S}}

\newcommand{\WW}{{\mathcal W}}

\newcommand{\C}{{\Bbb C}}

\newcommand{\D}{{\Bbb D}}

\newcommand{\N}{{\Bbb N}}

\newcommand{\R}{{\Bbb R}}
\newcommand{\T}{{\Bbb T}}

\newcommand{\Z}{{\Bbb Z}}

\newcommand{\LINV}{L_{\rm inv}}
\newcommand{\LLINV}{{\mathcal L}_{\rm inv}}

\def\B0{{\mathbf{0}}}

\newcommand{\Jac}{\operatorname{Jac}}

\newcommand{\Hol}{{\rm Hol}}

\newcommand{\CP}{ {\Bbb{CP}}   }

\catcode`\@=12

\def\Empty{}
\newcommand\oplabel[1]{
  \def\OpArg{#1} \ifx \OpArg\Empty {} \else
  	\label{#1}
  \fi}
		
\newcommand{\comm}[1]{}
\newcommand{\comment}[1]{}

\begin{document}

\bigskip\bigskip

\title[Lee-Yang-Fisher zeros ] {Lee-Yang-Fisher zeros for the DHL \\  
         and 2D rational dynamics, \\
        {\tiny II. Global Pluripotential Interpretation.}}

\author{Pavel Bleher}

\author{Mikhail Lyubich}

\author{Roland Roeder$^\dag$}
\thanks{$^\dag$Corresponding Author.  Email: \url{roederr@iupui.edu}}

\date{\today}

\address{Pavel Bleher \\ IUPUI Department of Mathematical Sciences \\ 402 N. Blackford St., LD270 \\ Indianapolis, IN 46202-3267.}
\email{pbleher@iupui.edu}

\address{Mikhail Lyubich \\ Mathematics Department and IMS, Stony Brook University, Stony Brook, NY 11794.}
\email{mlyubich@math.sunysb.edu}

\address{Roland Roeder \\ IUPUI Department of Mathematical Sciences \\ 402 N. Blackford St., LD270 \\ Indianapolis, IN 46202-3267.}
\email{roederr@iupui.edu}

 \begin{abstract} 
In a classical work of the 1950's, Lee and Yang proved that for fixed
nonnegative temperature, the zeros of the partition functions of a
ferromagnetic Ising model always lie on the unit circle in the complex magnetic
field.   Zeros of the partition function in the complex temperature were then
considered by Fisher, when the magnetic field is set to zero.  Limiting
distributions of Lee-Yang and of Fisher zeros are physically important as they
control phase transitions in the model.  One can also consider the zeros of the
partition function simultaneously in both complex magnetic field and complex
temperature.    They form an algebraic curve called the Lee-Yang-Fisher (LYF)
zeros.  In this paper we continue studying their limiting distribution for the
Diamond Hierarchical Lattice (DHL). In this case, it can be described in terms
of the dynamics of an explicit rational function $\Rmig$ in two variables (the
Migdal-Kadanoff renormalization transformation).  We study properties of the
Fatou and Julia sets of this transformation and then we prove that the
Lee-Yang-Fisher zeros are equidistributed with respect to a dynamical
$(1,1)$-current in the projective space.  The free energy of the lattice gets
interpreted as the pluripotential of this current.  We also prove a more
general equidistribution theorem which applies to rational mappings
having indeterminate points, including the Migdal-Kadanoff renormalization
transformation of various other hierarchical lattices.
\end{abstract}

\setcounter{tocdepth}{1}
 
\maketitle
\tableofcontents




\section{Introduction}

\subsection{Lee-Yang-Fisher zeros}

We will begin with providing a brief background on the Lee-Yang-Fisher zeros
that continues the discussion in Part I  \cite{BLR1}.

We consider the Ising model on a finite graph $\Gamma$ and its partition
function  $\Zpart_\Gamma$, which is a Laurent polynomial in two variables
$(z,t)$, where $z=e^{-h/T}$ is a ``field-like'' variable and $t=e^{-2J/T}$ is
``temperature-like'' one.   They are expressed in terms of the externally
applied magnetic field $h$, the temperature $T$, and the coupling constant $J >
0$; see \cite[Section~2.1]{BLR1} for more details.

For a fixed $t\in [0,1]$, the complex zeros of $\Zpart(z,t)$ in $z$ are called the
Lee-Yang zeros.  
The  Lee-Yang Theorem \cite{YL,LY} asserts that for the
ferromagnetic Ising model on any graph, {\it the
zeros of the partition function lie on the unit circle $\T$}  in the complex
plane.

If we have  a hierarchy of graphs $\Gamma_n$ of increasing size,
then under fairly general conditions, 
zeros of the partition functions $\Zpart_n= \Zpart_{\Gamma_n}$ will have a limiting
distribution $\mu_t$ on the unit circle.
This distribution captures phase transitions in the model.  

Instead of freezing temperature, one can freeze the external field, and
study  zeros of $\Zpart(z,t)$ in the $t$-variable.  They are called {\it Fisher zeros}
as they were first studied by Fisher  for the regular two-dimensional lattice,
see \cite{Fis0,Brascam_and_Kunz}.  Similarly to the Lee-Yang zeros,
asymptotic distribution of the Fisher zeros is supported on the singularities of the
magnetic observables, and  is thus related to phase transitions in the model.
However, Fisher zeros do not lie on the unit circle any more.
For instance, for the regular 2D lattice at zero field (corresponding
to $z=1$), the asymptotic distribution lies on the union of two
{\it Fisher circles} depicted on  Figure~\ref{FIG:FISHER_CIRCLES}.

\begin{figure}[htp]
\begin{center}

\begin{picture}(0,0)%
\includegraphics{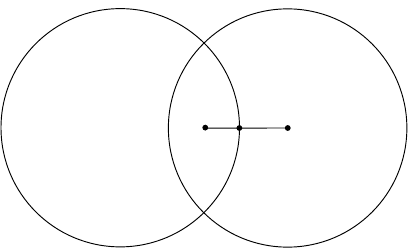}
\end{picture}%
\setlength{\unitlength}{3947sp}%
\begingroup\makeatletter\ifx\SetFigFont\undefined%
\gdef\SetFigFont#1#2#3#4#5{%
  \reset@font\fontsize{#1}{#2pt}%
  \fontfamily{#3}\fontseries{#4}\fontshape{#5}%
  \selectfont}%
\fi\endgroup%
\begin{picture}(3263,1922)(2136,-651)
\put(4324,173){\makebox(0,0)[lb]{\smash{{\SetFigFont{8}{9.6}{\familydefault}{\mddefault}{\updefault}{\color[rgb]{0,0,0}$t=1$}%
}}}}
\put(4088,376){\makebox(0,0)[lb]{\smash{{\SetFigFont{8}{9.6}{\familydefault}{\mddefault}{\updefault}{\color[rgb]{0,0,0}$t_c$}%
}}}}
\put(3646,177){\makebox(0,0)[lb]{\smash{{\SetFigFont{8}{9.6}{\familydefault}{\mddefault}{\updefault}{\color[rgb]{0,0,0}$t=0$}%
}}}}
\end{picture}%

\caption{\label{FIG:FISHER_CIRCLES}
The Fisher circles: $|t\pm 1| = \sqrt{2}$.}
\end{center}
\end{figure}

We can also consider the zeros of $\Zpart_n(z,t)$ as a single object in $\C^2$.
While $\{\Zpart_n(z,t)= 0\}$ is an algebraic curve in $\C^2$, we want to keep
track of the multiplicities to which $\Zpart_n(z,t)$ vanishes along each
irreducible component of this curve.  We will do this using the notion of {\em
divisor}, which is a sum of finitely many irreducible algebraic curves, each
with integer multiplicities (see \cite[Appendix A.3]{BLR1}).  Thus, the way
$Z_n(z,t)$ vanishes in $\C^2$ defines a divisor\footnote{We will see in Remark
\ref{REM:REDUCED_DIVISOR} that each of these multiplicities is one, and hence
there is no harm in thinking in terms of the algebraic curve.} $\SSS^c_n$ on
$\C^2$ which we call the {\em Lee-Yang-Fisher} (LYF) zeros.

In order to study the limiting distribution of the LYF zeros $\SSS^c_n$, as $n$ tends to
infinity, we will use the theory of currents; see \cite{DERHAM,LELONG}.  A
$(1,1)$-current  $\nu$ on $\C^2$ is a linear functional on the space of
$(1,1)$-forms that have compact support (see Appendix~\ref{APPENDIX:CURRENTS}).
A basic example is the current $[X]$ of integration over an irreducible algebraic curve
$X$.  Meanwhile, the current of integration $[D]$ over a divisor $D$ is the weighted sum of 
currents of integration over each of irreducible components, weighted 
according to the multiplicities.
  A plurisubharmonic function $G$ is called a  {\it pluripotential} of
$\nu$ if $\displaystyle{  \frac{i}{\pi} \di\dibar G } =\nu$, in the sense of
distributions.  (Informally, this means that $\displaystyle{ \frac 1 {2\pi}
\De(G|\, L)=\nu|\, L } $ for almost any  complex line $\Line$, so
$G|\, \Line$ is the electrostatic potential of the charge distribution $\nu|\,
\Line$.)

Let $d_n$ be the degree of divisor $\SSS^c_n$.  It is natural to ask whether there exists a
$(1,1)$-current $\LYF$ so that
\begin{eqnarray}\label{EQN:DESIRED_LIMIT_LYF}
\frac{1}{d_n} [\SSS^c_n] \rightarrow \LYF.
\end{eqnarray}
It would describe the limiting distribution of
Lee-Yang-Fisher zeros.  Within almost any complex line $L$, the limiting
distribution of zeros can be obtained as the restriction $\LYF | L$.

In order to justify existence of $\LYF$, one considers
the sequence of ``free energies''
\begin{eqnarray*}
\FoverT_n(z,t) := \log|\check \Zpart_n(z,t)|,
\end{eqnarray*}
\noindent
where $\check \Zpart_n(z,t)$ 
is the polynomial obtained by clearing the denominators of $\Zpart_n$.
We will say that the sequence of graphs $\Gamma_n$ has a {\em global thermodynamic limit} if 
\begin{eqnarray*}
\frac{1}{d_n} \FoverT_n(z,t) \ra \FoverT(z,t)
\end{eqnarray*}
in $L^1_\loc(\C^2)$.  In Proposition \ref{PROP:L1_GLOBAL_LIMIT} we will show
that this is sufficient for the limiting current $\LYF$ to exist and convergence (\ref{EQN:DESIRED_LIMIT_LYF}) to hold. 

The support of $\LYF$ consists of the singularities of the magnetic observables
of the model, thus describing ``global phase transitions'' in $\C^2$.  Connected
components of $\C^2 \sm \supp \LYF$ describe the distinct ``complex
phases'' of the system.



\subsection{Diamond hierarchical model}\label{DHL intro}


The {\it diamond hierarchical lattice} (DHL) is a sequence of graphs
$\Gamma_n$ illustrated on Figure~\ref{FIG:DIAMOND GRAPHS}. 
Part I \cite{BLR1} and much of the present paper are both devoted to study of this lattice.

\begin{figure}
\begin{center}

\begin{picture}(0,0)%
\includegraphics{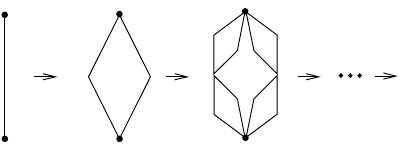}
\end{picture}%
\setlength{\unitlength}{3947sp}%
\begingroup\makeatletter\ifx\SetFigFont\undefined%
\gdef\SetFigFont#1#2#3#4#5{%
  \reset@font\fontsize{#1}{#2pt}%
  \fontfamily{#3}\fontseries{#4}\fontshape{#5}%
  \selectfont}%
\fi\endgroup%
\begin{picture}(3249,1601)(3649,-2964)
\put(4570,-2692){\makebox(0,0)[lb]{\smash{{\SetFigFont{8}{9.6}{\familydefault}{\mddefault}{\updefault}{\color[rgb]{0,0,0}$b$}%
}}}}
\put(5582,-2686){\makebox(0,0)[lb]{\smash{{\SetFigFont{8}{9.6}{\familydefault}{\mddefault}{\updefault}{\color[rgb]{0,0,0}$b$}%
}}}}
\put(3655,-2700){\makebox(0,0)[lb]{\smash{{\SetFigFont{8}{9.6}{\familydefault}{\mddefault}{\updefault}{\color[rgb]{0,0,0}$b$}%
}}}}
\put(4560,-1474){\makebox(0,0)[lb]{\smash{{\SetFigFont{8}{9.6}{\familydefault}{\mddefault}{\updefault}{\color[rgb]{0,0,0}$a$}%
}}}}
\put(3667,-1467){\makebox(0,0)[lb]{\smash{{\SetFigFont{8}{9.6}{\familydefault}{\mddefault}{\updefault}{\color[rgb]{0,0,0}$a$}%
}}}}
\put(5563,-1447){\makebox(0,0)[lb]{\smash{{\SetFigFont{8}{9.6}{\familydefault}{\mddefault}{\updefault}{\color[rgb]{0,0,0}$a$}%
}}}}
\put(6898,-2062){\makebox(0,0)[lb]{\smash{{\SetFigFont{8}{9.6}{\familydefault}{\mddefault}{\updefault}{\color[rgb]{0,0,0}$\Gamma_n$}%
}}}}
\put(3649,-2916){\makebox(0,0)[lb]{\smash{{\SetFigFont{8}{9.6}{\familydefault}{\mddefault}{\updefault}{\color[rgb]{0,0,0}$\Gamma_0$}%
}}}}
\put(5581,-2933){\makebox(0,0)[lb]{\smash{{\SetFigFont{8}{9.6}{\familydefault}{\mddefault}{\updefault}{\color[rgb]{0,0,0}$\Gamma_2$}%
}}}}
\put(4419,-2922){\makebox(0,0)[lb]{\smash{{\SetFigFont{8}{9.6}{\familydefault}{\mddefault}{\updefault}{\color[rgb]{0,0,0}$\Gamma = \Gamma_1$}%
}}}}
\end{picture}%

\end{center}
\caption{\label{FIG:DIAMOND GRAPHS} {Diamond hierarchical lattice.}}
\end{figure}

The Migdal-Kadanoff renorm-group RG equations 
for the DHL have the form:
\begin{equation}\label{R-intro}
    (z_{n+1},t_{n+1}) = \left( \frac{z_n^2+t_n^2}{z_n^{-2}+t_n^2}, \ \frac{z_n^2+z_n^{-2}+2}{z_n^2+z_n^{-2}+t_n^2+t_n^{-2}}\right):=
                       \Rphys(z_n,t_n),
\end{equation}
where $z_n$ and $t_n$ are the renormalized field-like and temperature-like
variables on $\Gamma_n$.  The map $\Rphys$ that relates these quantities is
also called the {\it renormalization transformation}. 

\msk
To study the Fisher zeros, we consider the line 
$\LLINV=\{z=1\}$ in $\C^2$.  This line is invariant under
$\Rphys$, and $\Rphys: \LLINV\ra \LLINV$ reduces to a fairly simple
one-dimensional rational map
\begin{eqnarray*}
        \Rphys:  t\mapsto \left(\frac {2t}{t^2+1}\right)^2.
\end{eqnarray*}
The Fisher zeros at level $n$ are obtained by pulling back the point $t=-1$
under $\Rphys^n$.  As shown in \cite{BL}, the limiting distribution of the
Fisher zeros in this case exists and it coincides with the measure  of maximal
entropy (see \cite{BROLIN,LYUBICH:NOTE,LYUBICH:MAX ENT,FLM}) of $\Rphys|\, \Line$.  The limiting support for this measure
is the Julia set for $\Rphys | \LLINV$, which is shown in Figure~\ref{FIG:INVARIANT_LINE_JULIA}.
It was studied by \cite{DDI,DIL,BL,Ish} and others. 

\begin{figure}[htp]
\begin{center}

\begin{picture}(0,0)%
\includegraphics{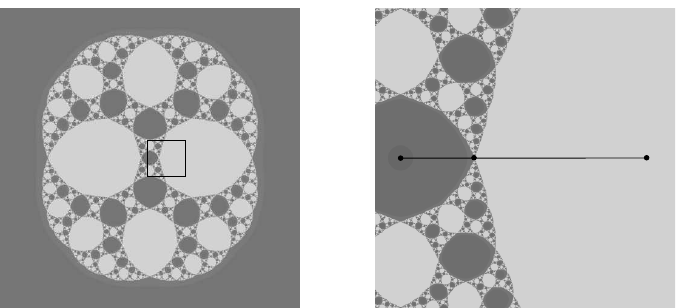}
\end{picture}%
\setlength{\unitlength}{3947sp}%
\begingroup\makeatletter\ifx\SetFigFont\undefined%
\gdef\SetFigFont#1#2#3#4#5{%
  \reset@font\fontsize{#1}{#2pt}%
  \fontfamily{#3}\fontseries{#4}\fontshape{#5}%
  \selectfont}%
\fi\endgroup%
\begin{picture}(5400,2400)(2401,-2761)
\put(7454,-1500){\makebox(0,0)[lb]{\smash{{\SetFigFont{8}{9.6}{\familydefault}{\mddefault}{\updefault}{\color[rgb]{0,0,0}$t=1$}%
}}}}
\put(6275,-1513){\makebox(0,0)[lb]{\smash{{\SetFigFont{8}{9.6}{\familydefault}{\mddefault}{\updefault}{\color[rgb]{0,0,0}$t_c$}%
}}}}
\put(5502,-1489){\makebox(0,0)[lb]{\smash{{\SetFigFont{8}{9.6}{\familydefault}{\mddefault}{\updefault}{\color[rgb]{0,0,0}$t=0$}%
}}}}
\end{picture}%

\caption{\label{FIG:INVARIANT_LINE_JULIA}
On the left is the Julia
set for $\Rphys|\, \LLINV$.  On the right is a zoomed-in view of a
boxed region around the critical point $t_c$.
The invariant interval $[0,1]$ corresponds to the states with
real temperatures $T\in [0,\infty]$ and vanishing field $h=0$.}
\end{center}
\end{figure}

In this paper, we will use the Migdal-Kadanoff RG equations to study the global
limiting distribution of Lee-Yang-Fisher zeros for the DHL in the complex
projective plane $\CP^2$.   (The divisors $\SSS^c_n$ are extended to $\CP^2$ in
the natural way.) The first main result of this paper is:

\begin{thmC}\label{Global Current Thm}
For the DHL, the currents $\frac{1}{2\cdot 4^n}[\SSS^c_n]$ converge distributionally to some $(1,1)$-current $\LYF$ on $\CP^2$ 
whose pluripotential coincides with the free energy $\FoverT$ of the system. 
\end{thmC}

It would seem natural to prove this theorem by extending $\Rphys$ as a rational
map $\Rphys:~\CP^2~\ra~\CP^2$ and then considering the normalized pullbacks
$\frac{1}{2\cdot 4^n} (\Rphys^n)^*\SSS^c_0$.  However, an important subtlety
arises because the degrees of $\Rphys$ do not behave properly under iteration: 
$$4^n < \deg(\Rphys^n) < (\deg(\Rphys))^n = 6^n.$$
This {\em algebraic instability\footnote{For the definition,
see \cite[\S 1.4]{S_PANORAME} or \cite[Appendix A.6]{BLR1}.}} of $\Rphys$ has the consequence that
$$\SSS^c_n \neq (\Rphys^n)^*\SSS^c_0.$$
The issue is resolved by working with another rational mapping $\Rmig : \CP^2 \ra \CP^2$ coming
directly from the Migdal-Kadanoff RG Equations, without passing to the
``physical'' $(z,t)$-coordinates.  
This map is semi-conjugate to $\Rphys$ by a degree two
rational map $\correspond:~\CP^2\ra\CP^2$.  Moreover, $\Rmig$ is algebraically
stable, satisfying $\deg(\Rmig^n) = (\deg(\Rmig))^n = 4^n$.  For each $n \geq 0$, we have:
$$\SSS^c_n = \correspond^{-1} (\Rmig^{-n} S^c_0),$$
where $S^c_0$ is an appropriate projective line.  

Note that even though $\Rmig$ is algebraically stable, it is still not
well-defined at two indeterminate points $\INDmig_\pm$
which strongly influence the global dynamics.

\subsection{Equidistribution of curves to the Green Current}\label{SUBSEC:EQUIDIST}
Associated to any (dominant, algebraically stable) rational mapping
$f:\CP^2\ra\CP^2$ is a canonically defined invariant current $\Green$, called
the Green current\footnote{It is common in the literature to denote the Green current by $T$, but we use $\Green$ to avoid
any confusion with the temperature.} of $f$.  It satisfies $f^* \Green$ = $d \cdot \Green$, where
$d = \deg f$.  
Such invariant currents are a powerful tool of higher-dimensional holomorphic
dynamics: see Bedford-Smillie \cite{BS}, Fornaess-Sibony \cite{FS}, Hubbard-Papadapol \cite{HUBBARD_PAPADAPOL}, and others
(see \cite{DS_SURVEY,S_PANORAME} for surveys of this subject). 

Let $A \subset \CP^2$ be an algebraic curve of degree $\deg(A)$.  Since the early 1990's there has been extensive research 
proving that
\begin{eqnarray}\label{EQN:DESIRED_EQUIDISTRIBUTION}
\frac{1}{d^n \deg(A)}(f^n)^* [A] \rightarrow  \Green
\end{eqnarray}
under certain hypotheses on $f$ and $A$.  See
\cite{BS,FS,RUSS_SHIFF,DILLER,FAVRE_GUEDJ,FAVRE_THESIS,FAVRE_JONSSON,GUEDJ_VOL1,GUEDJ_VOL2,DS,DDG,TAFLIN,PROTIN}
for a sample of papers on the subject.  Note also the recent survey \cite{DS3}.

If (\ref{EQN:DESIRED_EQUIDISTRIBUTION}) holds for $f = R$ (the Migdal-Kadanoff
RG mapping for the DHL) and $A = S^c_0$ (the principal LYF zeros),
then we obtain the Global Lee-Yang-Fisher Theorem by pulling everything back
under $\correspond$.  In this way, the classical Lee-Yang-Fisher theory gets
linked to the contemporary Dynamical Pluripotential Theory.

However, the majority of the papers studying
(\ref{EQN:DESIRED_EQUIDISTRIBUTION}) focus on the case that either
\begin{itemize}
\item  [1)]  
 $f$ is birational ($f$ has a rational ``inverse''), or
\item  [2)]        
$f$ is a holomorphic endomorphism (no indeterminate points), 
\end{itemize}
in order to obtain the sharpest possible results.  Otherwise, they either
assume $A$ is generic, or they work with a more ``diffused'' current in place
of $A$.  In any case, because $I(R) \neq \emptyset$ and $d_{\rm top}(R) > 1$, there does not seem to be an existing result that 
applies to our setting:

\begin{thmD}
Convergence {\rm (\ref{EQN:DESIRED_EQUIDISTRIBUTION})} holds for the Migdal-Kadanoff
Renormalization mapping $R: \CP^2 \rightarrow \CP^2$ and any algebraic curve $A
\subset \CP^2$.
\end{thmD}

\begin{rem}
The dynamical approach to studying the limiting distribution of Lee-Yang-Fisher 
zeros for hierarchical lattices has independently been considered in
\cite{DeSMa} and studied numerically in \cite{DeS}.
\end{rem}

The strategy  of 
the proof of the Equistribution Theorem for the DHL can be
adapted to prove a more general Equidistribution Theorem, also suitable for
rational maps 
whose indeterminacy locus satisfies certain properties.

Let $f: \CP^2 \rightarrow \CP^2$ be a dominant algebraically stable rational
mapping of algebraic degree $d$.  Denote the indeterminacy set of $f$ by
$I(f)$.  For any $Y \subset \CP^2$ we define $f(Y)$ and $f^{-1}(Y)$ using a
resolution of $I(f)$; see Appendix \ref{APPENDIX:STRUCTURE_RAT_MAP}.

We say that an algebraic curve $A$ is {\em backward invariant}\footnote{Note that such a curve is also
forward invariant unless it contains an indeterminate point that blows-up to a different curve.} if $f^{-1}(A) =
A$ and we say that $A$ is {\em collapsed by $f$} if $f(A \setminus I(f))$
is a single point.   Let 
\begin{align*}
\BISET:= \{f(A \setminus I(f))\, :\, \mbox{$A$ is a collapsed curve of $f$}\}.
\end{align*}
Since each collapsed curve is critical for $f$, $\BISET$ is finite.

Choose a volume form on  $\CP^2$ normalized so that $\vol(\CP^2)=1$.  For any
$z \in \CP^2$ we define the {\em volume
exponent}\label{DEF:VOL_EXP} $\sigma(z,f)$ to be the smallest positive number so that for any
$\gamma > \sigma(z,f)$ there is a constant $K > 0$ and a neighborhood $N$ of
$z$ such that for any measurable set $Y \subset \CP^2$ we have
\begin{align}\label{EQN:DEF_SIGMA}
\vol(f^{-1}(Y) \cap N ) \leq K \left(\vol Y\right)^{1/\gamma}.
\end{align}
In \S \ref{SEC:VOL_EST_ONE_ITERATE} we will give two estimates on $\sigma(z,f)$
in terms of how the 
 complex Jacobian $\Jac f:=\det Df$ vanishes at $z$.  One of them is simply in terms of the order of vanishing $\mu(z,f)$ of $\Jac f$ at $z$,
while the second is stronger, but requires a more detailed assumption on $\Jac f$ near $z$.

A sequence of points $\{z_n\} \subset \CP^2$ is an {\em orbit} of $f$
if $z_{n+1} \in f(\{z_n\})$ for each $n \geq 0$.  (If $z_n \in I(f)$, $z_{n+1}$
can be any point on the algebraic curve $f(\{z_n\})$.) If $z_n \not \in I(f)$
for every $n$, we will refer to the orbit as a {\em regular orbit}.  Otherwise,
we will refer to it as an {\em indeterminate orbit}.

If $f^n$ is holomorphic in a neighborhood of $z \in \CP^2$ let $c(z,f^n)$
denote \label{DEF:C} the order of vanishing of the power series expansion for $f^n$ expressed
in local coordinates centered at $z$ and $f^{n}(z)$, respectively.
If $z_0,\ldots,z_{k-1}$ is a regular periodic orbit of period $k$ for~$f$ then
\begin{align*}
c_\infty(z_0,f) := \lim_{n \rightarrow \infty} c(z_0,f^{nk})^{1/nk}
\end{align*}
exists and satisfies $c_\infty(z_0,f) \leq d$; see \S
\ref{SUBSEC:SUPERATTRACTING}.  We say that a regular periodic point $z_0$ is {\em
superattracting} if $c_\infty(z_0,f) > 1$ and is {\em maximally
superattracting}\label{DEF:MAX_SUPER_ATTR} if $c_\infty(z_0,f) = d$.  In these
cases, the orbit of $z_0$ is attracting at superexponential rate in {\em
all directions}.

Let $\mathcal{E}$ be the finite set containing all
\label{DEF_EXCEPTIONAL_SET}
\begin{itemize}
\item[(a)] maximally superattracting periodic points, and
\item[(b)] superattracting periodic points $z_0$ of period $k$ for which there is an algebraic curve $C$ 
that is backward invariant under $f^k$, collapsed to $z_0$ under some iterate of $f^k$,  and for which
$z_0$ is a singular point of $C$.
\end{itemize}
Denote the respective subsets of $\mathcal{E}$ where (a) or (b) holds as $\mathcal{E}(a)$ and $\mathcal{E}(b)$.
We will call $\mathcal{E}$ the {\em exceptional set} for $f$.

\begin{thmE}\label{THM:EQUIDISTRIBUTION}
Let $f: \CP^2 \rightarrow \CP^2$ be a dominant algebraically stable
rational map of degree $d \geq 2$  and let $\Green$ denote the Green Current of $f$.  Assume that
\begin{itemize}
\item[(i)]  $I(f) \neq \emptyset$, 
\item[(ii)] $\sigma(z,f) < d$ for every $z \in I(f)$, and
\item[(iii)] no periodic orbit passes through both the finite set
  $\VERYBADSET:= \{z \in \CP^2 \, : \, \sigma(z,f) > d\}$ and 
$I(f) \cup \BISET$.
\end{itemize}
Then for any algebraic curve $A$ that does not pass through the exceptional set $\mathcal{E}$ we have
\begin{eqnarray*}
\frac{1}{d^n \deg A}(f^n)^* [A] \rightarrow  \Green.
\end{eqnarray*}
\end{thmE}

Note that Hypotheses {\rm (}i{\rm )} and {\rm (}ii{\rm )} are 
verifiable algebraic conditions on the map $f$  itself.  
The last Hypothesis {\rm (}iii{\rm )}  is more problematic as it is dynamical;
still it is amenable to verification under favorable circumstances  as
it  requires that  a certain finite set of points (specified algebraically)  is aperiodic.    
We illustrate application of the Equidistribution Theorem to a few examples in~\S\ref{SEC:APPLICATIONS}.

\msk
Let us also compare\footnote{We will specialize the following two results to
the case of pulling back a curve, instead of pulling back an arbitrary closed positive
$(1,1)$ current.} our result to the cases of birational maps and endomorphisms:

\begin{thmBIRAT}
Let $f:\CP^2~\rightarrow~\CP^2$ be an algebraically stable  birational mapping of degree $d \geq 2$
and let $\Green$ denote the Green Current of $f$.  Let $A \subset \mathbb{CP}^2$ be an algebraic curve.
Then, there is a exceptional set $\mathcal{E}$ consisting of at most one point such that
\begin{eqnarray*}
\frac{1}{d^n \deg A}(f^n)^* [A] \rightarrow  \Green \qquad \mbox{if and only if} \qquad  \mbox{$A$ does not pass through $\mathcal{E}$.}
\end{eqnarray*}
\end{thmBIRAT}

The exceptional set $\mathcal{E}$ consists of a maximally superattracting fixed
point through which there passes a backward invariant curve.  In the special
case that $f$ is a H\'enon mapping,
\begin{align*}
f [x:y:z] = [x^2+ayz:xz:z^2],
\end{align*}
we have that $\mathcal{E} = [1:0:0]$ is the superattracting fixed point at infinity, with the totally invariant curve
corresponding to the line at infinity $\{z=0\}$ \cite{BS,FS2}.

\begin{thmENDOS}
Let $f: \CP^2 \rightarrow \CP^2$ be a holomorphic endomorphism of
degree $d \geq 2$ and let $\Green$ denote the Green Current of $f$.  
Then there is a totally invariant algebraic 
set $\mathcal{E}_1$ consisting of at most three projective complex lines and a finite totally invariant set $\mathcal{E}_2$ with the following property:
If $A$ is an algebraic curve such that
\begin{itemize}
\item[(i)] $A \not \subset \mathcal{E}_1$, and
\item[(ii)] $A \cap \mathcal{E}_2 = \emptyset$,
\end{itemize}
then 
\begin{eqnarray*}
\frac{1}{d^n \deg A}(f^n)^* [A] \rightarrow  \Green.
\end{eqnarray*}
\end{thmENDOS}
\noindent

The exceptional set $\mathcal{E}_1$ corresponds to curves on which the order of
vanishing of the Jacobian grows at rate $\geq d^n$
under iteration, 
and hence the volume exponent $\sigma(z,f^n)$ growing at
rate $\geq d^n$ as well.  Meanwhile the set $\mathcal{E}_2$ consists of maximally
superattracting periodic points.

Our general strategy is similar to that in the above mentioned works:
We prove the $L^1_\loc$-convergence of the potentials of the currents
under consideration, which requires estimates on the volume growth
under the iterated  pullbacks.   
(For the latter, we have especially profited from  the techniques developed  by Favre and Jonsson
\cite{FAVRE_JONSSON}).  
However,  in our setting there is a 
possibility that the orbit of a point $z$ recurs to $I(f)$, while also having
bad growth of the volume exponent. 
In the case of birational maps, this is eliminated since the only critical
points are on collapsed curves, whose orbits stay away from $I(f)$ 
(by the algebraic stability assumption).  In our case,
Hypotheses (i)--(iii) allow us to rule out the problematic scenario.
%

The final punch line of our argument is 
an application of the Borel-Cantelli  Lemma,
which makes it quite  elementary and general. 


\subsection{Other hierarchical lattices}

The Diamond Hierarchical Lattice has the merit of being one of the simplest
non-trivial hierarchical lattices.   Instead of using the diamond to generate
our sequence of graphs $\{\Gamma_n\}_{n=0}^\infty$ (as shown in Figure
\ref{FIG:DIAMOND GRAPHS}) we can use any finite graph $\Gamma$ with two marked
vertices $a$ and $b$ that is symmetric under interchange of $a$ and $b$.  One
obtains $\Gamma_{n+1}$ by replacing each edge of $\Gamma_n$ with a copy of
$\Gamma$, using the marked vertices $a$ and $b$ as ``endpoints''.  We will call
the sequence of graphs {\em the hierarchical lattice generated by $\Gamma$}.

\begin{figure}
\scalebox{1.1}{

\begin{picture}(0,0)%
\includegraphics{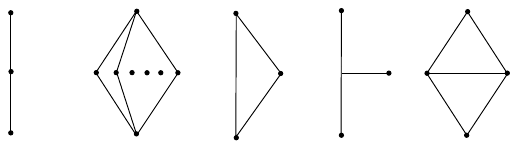}%
\end{picture}%
\setlength{\unitlength}{3947sp}%
\begingroup\makeatletter\ifx\SetFigFont\undefined%
\gdef\SetFigFont#1#2#3#4#5{%
  \reset@font\fontsize{#1}{#2pt}%
  \fontfamily{#3}\fontseries{#4}\fontshape{#5}%
  \selectfont}%
\fi\endgroup%
\begin{picture}(4365,1519)(4188,-2921)
\put(5527,-2204){\makebox(0,0)[lb]{\smash{{\SetFigFont{8}{9.6}{\familydefault}{\mddefault}{\updefault}{\color[rgb]{0,0,0}$k$}%
}}}}
\put(4459,-1501){\makebox(0,0)[lb]{\smash{{\SetFigFont{8}{9.6}{\familydefault}{\mddefault}{\updefault}{\color[rgb]{0,0,0}$a$}%
}}}}
\put(4472,-2671){\makebox(0,0)[lb]{\smash{{\SetFigFont{8}{9.6}{\familydefault}{\mddefault}{\updefault}{\color[rgb]{0,0,0}$b$}%
}}}}
\put(4203,-2866){\makebox(0,0)[lb]{\smash{{\SetFigFont{8}{9.6}{\familydefault}{\mddefault}{\updefault}{\color[rgb]{0,0,0}Linear Chain}%
}}}}
\put(7015,-2842){\makebox(0,0)[lb]{\smash{{\SetFigFont{8}{9.6}{\familydefault}{\mddefault}{\updefault}{\color[rgb]{0,0,0}Tripod}%
}}}}
\put(7093,-1505){\makebox(0,0)[lb]{\smash{{\SetFigFont{8}{9.6}{\familydefault}{\mddefault}{\updefault}{\color[rgb]{0,0,0}$a$}%
}}}}
\put(7104,-2701){\makebox(0,0)[lb]{\smash{{\SetFigFont{8}{9.6}{\familydefault}{\mddefault}{\updefault}{\color[rgb]{0,0,0}$b$}%
}}}}
\put(7800,-2829){\makebox(0,0)[lb]{\smash{{\SetFigFont{8}{9.6}{\familydefault}{\mddefault}{\updefault}{\color[rgb]{0,0,0}Split Diamond}%
}}}}
\put(8124,-1513){\makebox(0,0)[lb]{\smash{{\SetFigFont{8}{9.6}{\familydefault}{\mddefault}{\updefault}{\color[rgb]{0,0,0}$a$}%
}}}}
\put(8114,-2695){\makebox(0,0)[lb]{\smash{{\SetFigFont{8}{9.6}{\familydefault}{\mddefault}{\updefault}{\color[rgb]{0,0,0}$b$}%
}}}}
\put(5143,-2875){\makebox(0,0)[lb]{\smash{{\SetFigFont{8}{9.6}{\familydefault}{\mddefault}{\updefault}{\color[rgb]{0,0,0}$k$-fold DHL}%
}}}}
\put(6094,-2865){\makebox(0,0)[lb]{\smash{{\SetFigFont{8}{9.6}{\familydefault}{\mddefault}{\updefault}{\color[rgb]{0,0,0}Triangle}%
}}}}
\put(5477,-1508){\makebox(0,0)[lb]{\smash{{\SetFigFont{8}{9.6}{\familydefault}{\mddefault}{\updefault}{\color[rgb]{0,0,0}$a$}%
}}}}
\put(6251,-1518){\makebox(0,0)[lb]{\smash{{\SetFigFont{8}{9.6}{\familydefault}{\mddefault}{\updefault}{\color[rgb]{0,0,0}$a$}%
}}}}
\put(5462,-2689){\makebox(0,0)[lb]{\smash{{\SetFigFont{8}{9.6}{\familydefault}{\mddefault}{\updefault}{\color[rgb]{0,0,0}$b$}%
}}}}
\put(6269,-2714){\makebox(0,0)[lb]{\smash{{\SetFigFont{8}{9.6}{\familydefault}{\mddefault}{\updefault}{\color[rgb]{0,0,0}$b$}%
}}}}
\end{picture}%

}
\caption{\label{FIG:GEN_GRAPHS} Generating graphs for some other hierarchical lattices.}
\end{figure}

Associated to each generating graph $\Gamma$ is a Migdal-Kadanoff
renormalization mapping, which is a rational map $R_\Gamma: \CP^2 \rightarrow
\CP^2$.   In \S \ref{SEC:APPLICATIONS} we discuss the Migdal Kadanoff renormalization mappings
associated to the five different hierarchical lattices whose generating graphs are shown in Figure~\ref{FIG:GEN_GRAPHS}.
We will see that the Equidistribution
Theorem applies to the Migdal Kadanoff renormalization mappings associated to
the $k$-fold DHL ($k \geq 2$), the Triangle, and the Split Diamond thus proving the Global LYF Theorem for each of those lattices.
We find that the Equidistribution Theorem does not apply to the Migdal Kadanoff renormalization mappings for the Linear Chain or the Tripod, however
an easy argument directly shows that the Global LYF Theorem holds for the Linear Chain.  We do not know if it holds for the Tripod.

\begin{problem}\label{PROB_GENERAL_HL}
Does the Global LYF Theorem hold for every hierarchical lattice?
\end{problem}

\noindent
Several additional open problems are listed in Appendix \ref{APP:PROBLEMS}.

\subsection{Structure of  the paper}
We begin in \S \ref{SEC:MODEL} by recalling the definitions of free energy and
the classical notion of thermodynamic limit for the Ising model.  We then
discuss the notion of global thermodynamic limit, which is sufficient in order
to guarantee that some lattice have a $(1,1)$-current $\LYF$ describing its
limiting distribution of LYF zeros in $\C^2$.  
We also give an alternative interpretation of the
partition function as a section of (an appropriate tensor power of) the hyperplane 
bundle over $\CP^2$ that will be central to the proof of the Global LYF Current Theorem.
We conclude \S \ref{SEC:MODEL} by summarizing material on the Migdal-Kadanoff
RG equations, including the details for reducing the proof of the Global LYF
Theorem to proof of (\ref{EQN:DESIRED_EQUIDISTRIBUTION}) for $A = \SSS^c_n$ and
$f = R$.

In \S \ref{SEC:GLOBAL STRUCTURE} we summarize the global features of the
mappings $\Rphys$ and $\Rmig$ on the complex projective space $\CP^2$ that were
studied in \cite{BLR1}, including their critical and indeterminacy loci,
superattracting fixed points and their separatrices.

In the next section, \S \ref{SUBSEC:JULIA_AND_FATOU}, we define the Fatou and
Julia sets for $\Rphys$ and show that the Julia set coincides with the closure
of preimages of the invariant complex line $\{z=1\}$ (corresponding to the
vanishing external field).  It is based on M. Green's criteria for Kobayashi
hyperbolicity of the complements of several algebraic curves in  $\CP^2$
\cite{GREEN:PAMS,GREEN:AJM} that generalize the classical Montel Theorem.  We
then use this result to prove that points in the interior of the solid cylinder
$\D \times I$ are attracted to a superattracting fixed point $\eta=(0,1)$ of
$\Rphys$.  

The proofs of the equidistribution theorems require estimates on the volume of
a tubular neighborhood of an algebraic curve and estimates on how volume is
transformed under a rational map.  These estimates are presented in \S
\ref{SEC:VOL_EST_ONE_ITERATE}.

\S \ref{SEC:DHL_EQUIDISTRIBUTION} is devoted to proving the
Equidistribution Theorem for the DHL.  In \S \ref{SEC:EQUIDISTRIBUTION} we then show how to adapt its proof
to prove the Equidistribution Theorem.
In \S \ref{SEC:APPLICATIONS} we discuss applications of the Equidistribution Theorem to other hierarchical lattices.

Like Part I, this paper is written for readers from both complex dynamics and
statistical physics, so we provide background material in two appendices.  To
minimize overlap, we will refer the reader to appendices of Part I when
possible.  In Appendix \ref{APP:COMPLEX GEOM} we collect needed 
background in complex geometry (line bundles over $\mathbb{CP}^2$, currents and their pluri-potentials, Kobayashi
hyperbolicity, and normal families).  In Appendix \ref{APP:COMPLEX_DYNAMICS} we
provide information on the Green current.  In Appendix~\ref{APP:PROBLEMS} we
collect several open problems.

\subsection{Basic notation and terminology}
$\C^*=\C\sm \{0\}$, $\T=\{|z|=1\}$, $\D_r=\{|z|<r\}$, $\D\equiv \D_1$, and
 $\N=\{0,1,2\dots\}$. 
Given two variables $x$ and $y$, 
$x\asymp y$ means that $c \leq |x/y|\leq C$ for some constants $C>c>0$.  \\

\msk \noindent{\bf Acknowledgments:} \\
\noindent
We thank Jeffrey Diller, Mattias Jonsson, Han Peters, Robert Shrock, and Dror
Varolin for interesting discussions and comments.  We also thank the referee
for stimulating comments that motivated us to bring the original version of the
results (see Stony Brook Preprint ims11-3 from 2011) to a more general form.
The work of the first author has been supported in part by the NSF grants
DMS-0652005,  DMS-0969254, DMS-1265172, and DMS-1565602. The work of the second
author has been supported in part  by NSF, NSERC and CRC funds.  The work of
the third author has been supported in part by NSF grants DMS-1102597 and
DMS-1348589,  and by startup funds from the Department of Mathematics at IUPUI.


\section{Description of the model}
\label{SEC:MODEL}

\subsection{Free Energy and Thermodynamic Limit}\label{SUBSEC:THERMO}

The  partition function 
of the Ising model on a graph $\Gamma$ 
is a symmetric Laurent polynomial in $(z,t)$ of the form
\begin{equation}\label{symmetric Laurent}
   \Zpart_\Gamma = \sum_{n=0}^d a_n(t) (z^n+z^{-n})  
\end{equation}
of degree $d$ equal to the the number of edges in $\Gamma$.  (See \cite[Section 2.1]{BLR1} for the definition.)

\begin{rem}\label{REM:REDUCED_DIVISOR} 
Setting  $\check \Zpart(z,t): = z^{d} t^{d/2} \Zpart(z,t)$ clears the denominators of $ \Zpart(z,t)$
and results in a polynomial in $z$ and $t$ of degree $2d$
whose divisor of zeros is the same as that of $\Zpart(z,t)$.  It follows from the definition
of $\Zpart(z,t)$ that $\check \Zpart(z,0) = z^{2d}+1$, each of whose zeros is simple.  In particular
the divisor of zeros assigns multiplicity one to each irreducible component of $\{\Zpart(z,t) = 0\}$.
\end{rem} 

The {\it  free energy} of the system is defined as  
\begin{equation}
      F_\Gamma =   - T\log |\Zpart_\Gamma|,
\end{equation} 
where $T$ is the temperature (related to the temperature-like variable
by  $t= e^{-J/T}$, where $J$  is the coupling constant of the model). 

 
It will be more convenient  to consider the following variant of the free energy:
\begin{eqnarray*}
\FoverT_\Gamma(z,t) :=  -\frac{1}{T} F_\Gamma(z,t) + d(\log|z|+\frac{1}{2} \log|t|) = \log|\check \Zpart(z,t)|.
\end{eqnarray*}
The advantage of using $\FoverT_\Gamma$, instead of
$F_\Gamma$, is that it extends as a plurisubharmonic function on all of $\C^2$.
We will also refer to $\FoverT_\Gamma$ as the ``free energy''.

 
Assume that we have a lattice  given by a hierarchy  of graphs
$\Gamma_n$ with $d_n\to \infty$ edges.
Let us consider its partition functions $\Zpart_n$
and free energies $\FoverT_n$.   
To pass to the thermodynamic limit we normalize the free energy  {\it per bond}.
One says that the hierarchy of graphs has a (pointwise)   {\it thermodynamic limit} if
\begin{equation}\label{L1 assumption}
\frac{1} { 2 d_n} \FoverT_n(z ,t)\to \FoverT(z,t) \quad \mbox {for any} \ z\in \R_+, \ t\in (0,1).
\end{equation}
In this case, the function $\FoverT$ is called  the (modified)  {\it free energy} of the lattice.
For many\footnote{Note that the DHL is not in this class---instead, dynamical techniques are used to justify its classical thermodynamic
limit.} lattices (e.g. $\Z^d$), existence of the thermodynamic limit can be
justified by van Hove's Theorem \cite{VANHOVE,Ruelle_book}.  If the classical thermodynamic limit
exists, then one can justify existence of the limiting distribution of Lee-Yang
zeros and relate it to the limiting free energy; see \cite[Prop. 2.2]{BLR1}.

 In order to prove existence of
a limiting distribution for the Fisher zeros, one needs to prove
existence of the thermodynamic limit in the $L^1_\loc(\C)$-sense:
 \begin{eqnarray}\label{EQN:C1_L1_LIMIT}
 \frac {1} {2 d_n} \FoverT_n (1,t) \ra \FoverT(1,t) \quad \mbox {in} \,\, L^1_\loc(\C).
 \end{eqnarray}

\noindent
For the $\Z^2$ lattice this is achieved by the Onsager solution, which provides an
explicit formula for the limiting free energy;  see, for example, \cite{Baxter}.
Similar techniques apply to the triangular, hexagonal, and
various homopolygonal lattices (see \cite{Matveev_Shrock1,Matveev_Shrock2} for
suitable references and an investigation of the distribution of Fisher zeros 
for these lattices).  For various hierarchical lattices, (\ref{EQN:C1_L1_LIMIT}) can be proved by
dynamical means.  

The situation is similar for the Lee-Yang-Fisher zeros:

\begin{prop}\label{PROP:L1_GLOBAL_LIMIT}
Let $ (\Gamma_n) $ be a lattice for which the thermodynamic limit
exists in the $L^1_\loc(\C^2) $-sense:   
 \begin{eqnarray}\label{EQN:C2_L1_LIMIT}
 \frac{1}{2 d_n}\FoverT_n(z,t) \ra \FoverT(z,t) \quad \mbox {in} \,\, L^1_\loc(\C^2).
 \end{eqnarray}
Then there is a closed positive $(1,1)$-current $\LYF$ on $\C^2$ describing
the limiting distribution of Lee-Yang-Fisher zeros.  Its pluripotential coincides with
the free energy $\FoverT(z,t)$.
\end{prop}

For the DHL, we will prove existence of the limit (\ref{EQN:C2_L1_LIMIT}) in
the Global LYF Current Theorem.

\begin{rem}  It is an open question whether the limit {\rm (\ref{EQN:C2_L1_LIMIT})} exists for any classical lattice, including the $\mathbb{Z}^2$ lattice.
Moreover, it seems to also be an open question whether the limit {\rm (\ref{EQN:C1_L1_LIMIT})} exists for the $\Z^d$ lattice, when $d \geq 3$, and other classical three dimensional lattices.  See Problem
\ref{PROP:GLOBAL_LIMIT}.
\end{rem}

\begin{proof}[Proof of Prop. \ref{PROP:L1_GLOBAL_LIMIT}:]
The locus of Lee-Yang-Fisher zeros $\SSS^c_n$ are the zero set (counted with
multiplicities) of the degree $2 d_n$ polynomial $\check \Zpart_n(z,t)$.   The Poincar\'e-Lelong
Formula describes its current of integration:
\begin{eqnarray*}
[\SSS^c_n] = \PLapl \log |\check \Zpart_n(z,t)| = \PLapl \FoverT_n(z,t).
\end{eqnarray*}
Hypothesis (\ref{EQN:C2_L1_LIMIT}) implies
\begin{eqnarray*}
\frac{1}{2 d_n  } [\SSS^c_n] = \PLapl \frac{1}{2 d_n} \FoverT_n(z,t) \rightarrow \PLapl \FoverT(z,t) =: \LYF.
\end{eqnarray*}
\end{proof}

\subsection{Global consideration of partition functions and free energy on $\CP^2$}
\label{SUBSEC:THERMO2}

It will be convenient for us to extend the partition functions $\Zpart_n$ and their
associated free energies $\FoverT_n$ from $\C^2$ to $\CP^2$.  We will use the
homogeneous coordinates  $[Z:T:Y]$ on $\CP^2$, with the  copy of
$\C^2$ given by the affine coordinates  $(z,t)~\mapsto~[z:t:1]$.

For each $n$, we clear the denominators of $\Zpart_n(z,t)$, obtaining a
polynomial $\check \Zpart_n(z,t)$ of degree $d_n:=2|\EE_n|$.  It lifts to a
unique homogeneous polynomial $\hat \Zpart_n(Z,T,Y)$ of the same degree that
satisfies $\check \Zpart_n(z,t) = \hat \Zpart_n(z,t,1).$
The associated free energy becomes a plurisubharmonic function $$\hatFoverT_n(Z,T,Y) := \log|\hat \Zpart_n(Z,T,Y)|$$ on $\C^3$.
It is related by the Poincar\'e-Lelong Formula to the current of integration over the Lee-Yang-Fisher zeros:
$\pi^* [\SSS^c_n] := \PLapl \ \hatFoverT_n(Z,T,Y)$.

Both of these extensions are defined on $\C^3$, rather than $\CP^2$. 
In the proof of the Global LYF Current Theorem, it will be useful for us to interpret
the partition function as an object defined on $\CP^2$.  Instead of being a
function on $\CP^2$, it gets interpreted as a section
$s_{\Zpart_n}$ of an appropriate
tensor power of the hyperplane bundle; See Appendix \ref{hyp bundle}.  The Lee-Yang-Fisher zeros $\SSS^c_n$ are described as the zero locus
of this section.

\comment{*******************************
\msk
Appropriate generalizations of Proposition \ref{PROP:L1_GLOBAL_LIMIT} can be
used to determine whether the limiting distribution of Lee-Yang-Fisher zeros
exists in $\CP^2$.  For example, it suffices that
the sequence of normalized free energies
\begin{eqnarray}\label{EQN:NORMALIZED_FREE_ENERGY_C3}
\frac{1}{d_n} \hatFoverT(z,t,y)
\end{eqnarray}
converges in $L^1_\loc(\C^3)$.  Each of the corresponding sections
$\sigma_{\frac{1}{d_n}\FoverT_n}$ is of the same real line bundle over $\CP^2$, independent of the degree $d_n$, so one
can also justify existence of the limiting distribution of Lee-Yang-Fisher
zeros by showing that these sections converge in $L^1(\CP^2)$.
************************}

\subsection{Migdal-Kadanoff renormalization for the DHL}

The renormalized field-like   and temperature-like variables $z_n$ and  $ t_n$ 
that appear in the Migdal- Kadanoff RG equations (\ref{R-intro})
are defined  through certain 
``conditional partition functions  of level $n$''  in the following way: 
\begin{equation}\label{zt-uv}
   z_n^2= W_n/U_n, \quad t_n^ 2 = \frac {V_n^2} {U_n W_n}.
\end{equation}
%
%
In the $(U,V,W)$-coordinates the Migdal-Kadanoff RG  equation
assumes the  homogeneous form
\begin{eqnarray*}
  U_{n+1}= (U_n^2+V_n^2)^2, \quad V_{n+1}= V_n^2 (U_n+W_n)^2, \quad W_{n+1}= (V_n^2+W_n^2)^2 ,
\end{eqnarray*}
and the  total  partition function  becomes a linear form
$$
   \Zpart_n \equiv \Zpart_{\Gamma_n}= U_n+2V_n+W_n. 
$$
(See Part I \cite{BLR1} for the derivation
of these equations.) 
This leads us to   a homogeneous degree $4$ polynomial map
\begin{eqnarray}\label{EQN:RMIG_HAT}
\hat \Rmig: (U,V,W) \mapsto \left((U^2+V^2)^2 , V^2(U+W)^2, (W^2+V^2)^2\right),
\end{eqnarray}
called the {\em Migdal Kadanoff Renormalization}, such  that $(U_n,V_n,W_n) = \hat \Rmig^n(U_0,V_0,W_0)$.  
(The corresponding map $R: \CP^2\ra \CP^2$ will be  referred to in the same way.)
Moreover,  letting $Y_0:= U+2V + W$, we obtain: 
\begin{equation}\label{EQN:MIGDAL_FREE_ENERGIES}
   \hat \Zpart_n = Y_0 \circ \hat \Rmig^n,      
\end{equation}
so the partition functions $\hat \Zpart_n$ are obtained by pulling the linear
form $Y_0 \equiv \hat \Zpart_0$ back by $\hat\Rmig^n$. 

We will often write $\Rmig$ in the system of local coordinates $u = U/V$ and $w=W/V$, in which it has the form
\begin{equation}\label{uv coord}
   \Rmig: (u,w)\mapsto \left( \frac{u^2+1}{ u+w},\ \frac{w^2+1}{u+w} \right)^2.
\end{equation}

Notice  that the form $Y_0$ is {\it not a function} on $\CP^2$ but rather a
{\it section} $s_{Y_0}$ of  the  hyperplane line bundle over $\CP^2$,  see
Appendix \ref{hyp bundle}.  Respectively, the partition functions $\hat \Zpart_n$ 
are sections  of the tensor powers of this  line bundle.
Accordingly, the Lee-Yang-Fisher loci $S_n^c$ are the  {\it
 zero divisors} of these sections.  

The free energy is also no longer a function on $\CP^2$, rather it is lifted to become a function on $\C^3$, given by
\begin{eqnarray}
\hatFoverT_n := \log|\hat{\Zpart}_n|.
\end{eqnarray}

\msk 
The above formulae express the partition functions and free energies in
terms of the $U,V,W$ coordinates.  To re-express them in terms of the physical
coordinates, we pull each of them back by
\begin{equation}
  \correspond: \CP^2\ra \CP^2, \quad (U:V:W) = \correspond(z,t) = (z^{-1} t^{-1/2} : t^{1/2} : z t^{-1/2}).
\end{equation}
This change of variables also 
semi-conjugates the map
\begin{equation}\label{R}
    \Rphys: (z,t)\mapsto  \left( \frac{z^2+t^2}{z^{-2}+t^2}, \ \frac{z^2+z^{-2}+2}{z^2+z^{-2}+t^2+t^{-2}}\right),
\end{equation} 
corresponding to RG equation (\ref{R-intro}), to $\Rmig$.



\section{Global properties of the RG transformation in $\CP^2$}\label{SEC:GLOBAL STRUCTURE}
 
We will now summarize (typically without proofs) results from \cite{BLR1} about the global properties of
the RG mappings.

\subsection{Preliminaries}\label{SUBSEC:SEMICONJUGACY}

The renormalization mappings $\Rphys$ and $\Rmig$ are semi-conjugate by the
degree two rational mapping $\correspond: \CP^2 \ra \CP^2$ given by
(\ref{zt-uv}).  

Both mappings have topological degree $8$ (see Proposition 4.3 from Part~I).
However, as noted in the Introduction,  their algebraic degrees behave
differently: $\Rmig$ is algebraically stable, while $\Rphys$ is not.
Since $\deg(\Rmig^n) = 4^n$, for any algebraic curve $D$ of
degree $d$,  the pullback $(\Rmig^n)^* D$ is a divisor of degree $d \cdot
4^n$. (For background on divisors, see \cite[Appendix A.3]{BLR1}.)  For
this reason, we will focus most of our attention on the dynamics of $\Rmig$.

The semiconjugacy $\correspond$ sends the Lee-Yang cylinder $\Cphys:=\mathbb{T} \times [0,1]$ to a Mobius
band $\Cmig$ that is invariant under $\Rmig$.  It is obtained as the closure 
in $\CP^2$ of the topological annulus
\begin{equation}\label{C-D}
     \Cmigbl= \{(u,w)\in \C^2:\ w=\bar u, \ |u|\geq 1\}.
\end{equation}
Let $\TOPmig= \{(u,\bar u): \ |u|=1\} $ be the  ``top'' circle of $\Cmig$,
and let  $\BOTTOMmig$ be the slice of $\Cmig$ at infinity.  In fact,
$\correspond: \Cphys \ra \Cmig$ is a conjugacy, except that it maps the bottom
$\BOTTOMphys$ of $\Cphys$ by a $2$-to-$1$ mapping to $\BOTTOMmig$ (see Proposition 3.1
from Part~I).

\subsection{Indeterminacy points for $\Rmig$}
\label{SUBSEC:INDETERMINANT_PTS}

In homogeneous coordinates on $\CP^2$, the map $\Rmig$ has the form:
\begin{equation} \label{EQN:MK_HOMOG}
  \Rmig : [U: V: W] \mapsto [(U^2+ V^2)^2: \  V^2(U + W)^2: \ (V^2+W^2)^2)].
\end{equation}
One can see that $\Rmig$ has precisely two points of indeterminacy $\INDmig_+
:= [i:1:-i]$ and $\INDmig_-:=[-i:1:i]$.  Resolving all of the indeterminacies
of $\Rmig$ by blowing-up the two points $\INDmig_\pm$ (see 
\cite[Appendix A.2]{BLR1}),  one obtains a holomorphic mapping $\widetilde \Rmig: \widetilde \CP^2
\ra \CP^2$.  

In  coordinates
$\xi = u-i$ and $\chi =  (w+i)/(u-i)$ near $\INDmig_+ = (i,-i)$, we obtain the following expression for
the map $\widetilde \Rmig: \widetilde \CP^2\ra \CP^2$ near  $L_\ex(\INDmig_+)$:
\begin{equation}\label{blow up near a+}
    u = \left(\frac{\xi+2i}{1+\chi}\right)^2,\quad  w = \left(\frac{\chi^2 \xi-2i \chi}{1+\chi}\right)^2.
\end{equation}
(Similar formulas hold near $\INDmig_- = (-i,i)$.)
The exceptional divisor $L_\ex(\INDmig_+)$ is mapped by $\tl \Rmig$ to the conic
\begin{eqnarray*}
G:= \{(u-w)^2+8(u+w)+16 = 0\}.
\end{eqnarray*}

\subsection{Superattracting fixed points and their separatrices}
\label{SUBSEC:FIXED POINTS}

We will often refer to $\Lzero:= \{V=0\}\subset \CP^2$ as the  {\it line at infinity}.
It contains two symmetric superattracting fixed  points, $e=(1:0:0)$ and $e'=(0:0:1)$.  
Let $\WW^s(e)$ and $\WW^s(e')$ stand for the attracting basins of these points.
It will be useful to consider local coordinates $(\xi= W/U, \, \eta= V/U)$ near $e$.

The line at infinity  $\Lzero=\{\eta=0\}$ is $R$-invariant, and the restriction $R|\Lzero$ is the
power map $\xi\mapsto \xi^4$. Thus,  points in the disk $\{ |\xi|<1\}$ in $\Lzero$
are attracted to $e$, points in the disk $\{ |\xi|> 1\} $ are attracted to  $e'$,
and these two basins are separated by the unit  circle $\BOTTOMmig$.
We will also call $\Lzero$ the {\it fast separatrix} of $e$ and $e'$.

Let us also consider the conic
\begin{equation}\label{Lone}
   \Lone = \{ \xi=\eta^2\}= \{V^2= UW\}
\end{equation}
passing through points $e$ and $e'$.  It is an embedded copy of $\CP^1$ that is
invariant under $\Rmig$, with $R|\, \Lone(w) = w^2$, where $w=W/V= \xi/\eta$.
Thus,  points in the disk $\{ |w|<1\}$ in $\Lone$ are attracted to $e$, points
in the disk $\{ |w|> 1\} $ are attracted to  $e'$, and these two basins are
separated by the unit circle $\TOPmig$ (see \S 3.1 from Part~I).  We will call
$\Lone$ the {\it slow separatrix} of $e$ and $e'$.

If a point $x$ near $e$ (resp. $e'$) does not belong to the fast separatrix
$\Lzero$, then its orbit is ``pulled'' towards the slow separatrix $\Lone$ at
rate $\rho^{4^n}$, with some $\rho<1$, and converges to $e$ (resp. $e'$) along
$\Lone$ at rate $r^{2^n}$, with some $r<1$.

The strong separatrix $\Lzero$ is transversally superattracting:
all nearby points are pulled towards $\Lzero$ uniformly at rate $r^{2^n}$.
It follows that these points either converge  to one of the fixed points, $e$ or $e'$,
or converge to the circle $\BOTTOMmig$.

Given a neighborhood $\Om$ of $\BOTTOMmig$, let
\begin{equation}\label{complex basin of B}
  \WW^s_{\C,\loc} (\BOTTOMmig)= \{x\in \CP^2:\ R^n x\in \Om\ (n\in \N) \ {\mathrm {and}}\
        \R^n x\to \BOTTOMmig\ \mathrm{as}\ n\to \infty \}
\end{equation}
(where $\Om$ is implicit in the notation, and an assertion involving
$\WW^s_{\C,\loc}(\BOTTOMmig)$ means that it holds for arbitrary small suitable
neighborhoods of $\BOTTOMmig$).  
It is shown in Part~I (\S 9.2) that $\WW^s_{\C,\loc} (\BOTTOMmig)$ has the topology of a
$3$-manifold that is laminated by the union of holomorphic local stable manifolds
$W^s_{\C,\loc}(x)$ of points $x \in \BOTTOMmig$.

We conclude:
\begin{lem}\label{nbd of B}
 $\WW^s(e)\cup \WW^s(e')\cup \WW^s_{\C,\loc}(\BOTTOMmig)$ fills in some neighborhood of $\Lzero$.
\end{lem}

\subsection{Regularity of $\WW^s_{\C,\loc}(x)$}
\label{SUBSEC:REGULARITY}
For a diffeomorphism the existence and regularity of the local stable
manifold for a hyperbolic invariant manifold $N$ has been studied extensively
in \cite{HPS}.  In order to guarantee a $C^1$ local stable manifold
$\WW^s_{\loc}(N)$, a strong form of hyperbolicity known as {\em normal
hyperbolicity} is assumed.  Essentially, $N$ is normally hyperbolic for $f$ if the
expansion of $Df$ in the transverse unstable direction dominates the maximal tangent expansion of
$D(f|_N)$ and the contraction of $Df$ in the transverse stable direction
dominates the maximal tangent contraction of $D(f|_N)$.  See \cite[Thm.
1.1]{HPS}.  If, furthermore, the expansion of $Df$ in the transverse unstable direction dominates
the $r$-th power of the maximal tangent expansion of $D(f|_N)$ and the contraction
of $Df$ in the transverse stable direction dominates the $r$-th power of the maximal tangent contraction
of $D(f|_N)$, this guarantees that the stable manifold is of class $C^r$.
The corresponding theory for endomorphisms is less developed, although note that
some aspects of \cite{HPS}, related to persistence of normally hyperbolic invariant laminations,
have been generalized to endomorphisms in \cite{BERGER}.

In our situation, $\BOTTOMmig$ is not normally hyperbolic because it lies
within the invariant line $L_0$ and $\Rmig$ is holomorphic.  This forces the
expansion rates tangent to $\BOTTOMmig$ and transverse to $\BOTTOMmig$
(within this line) to coincide.  Therefore, the following result does not seem
to be part of the standard hyperbolic theory:

\begin{lem}\label{LEM:REGULARITY}
$\WW^s_{\C,\loc}(\BOTTOMmig)$ is a $C^\infty$ manifold and the stable foliation is a $C^\infty$ foliation by complex analytic discs.
\end{lem}

\begin{proof}
In Proposition 9.11 from Part~I, we showed that within the cylinder $\Cphys$
the stable foliation of $\BOTTOMphys$ has $C^\infty$ regularity and that the
stable curve of each point is real analytic.  Mapping forward under
$\correspond$, we obtain the same properties for the stable foliation of
$\BOTTOMmig$ within $\Cmig$.

Let us work in the local coordinates $\xi = W/U$ and $\eta = V/U$.  In these coordinates, $\BOTTOMmig = \{\eta = 0, |\xi|=1\}$.
The stable curve $W^s_{\C,\loc}(\xi_0)$ of any $\xi_0 \in \BOTTOMmig$ can be given by expressing $\xi$ as a holomorphic function of 
$\eta$:

\begin{eqnarray}
\xi = h(\eta,\xi_0) = \sum_{i=0}^\infty a_i(\xi_0) \eta^i.
\end{eqnarray}
The right hand side is a convergent power series with coefficients depending on $\xi_0$, having a uniform radius of convergence
over every $\xi_0 \in \BOTTOMphys$.
The series is uniquely determined by its values on the real slice $\Cmig$, in which the leaves depend with $C^\infty$
regularity on $\xi_0$.  Therefore, each of the coefficients $a_i(\xi_0)$ is $C^\infty$ in $\xi_0$.  This gives
that each $W^s_{\C,\loc}(\xi_0)$  depends with $C^\infty$ regularity on $\xi_0$, implying the stated result.
\end{proof}

\begin{rem}The technique from the proof of Lemma \ref{LEM:REGULARITY} applies
to a 
more general situation:  Suppose that $M$ is a
real analytic manifold and $f: M \ra M$ is a real analytic map.  Let $N \subset
M$ be a compact real analytic invariant submanifold for $f$, with $f|N$
expanding and with $N$ transversally attracting under $f$.  Then $N$ will have
a stable foliation $\WW^s_\loc(N)$ of regularity $C^r$, for some $r >  0$ (see the beginning of this subsection), with the stable
manifold of each point being real-analytic.  The stable manifold
$\WW^s_{\C,\loc}(N)$ for the extension of $f$ to the complexification $M_\C$ of
$M$ will then also have $C^r$ regularity.
\end{rem}

\begin{rem} It has been shown by Kaschner and the third author that $\WW^s_{\C,\loc}(\BOTTOMmig)$ is not real analytic at any point \cite[Thm. B]{KR}.
\end{rem}

\subsection{Critical locus}\label{SUBSEC:CRITICAL_LOCUS}
The complex Jacobian of $\hat R: \C^3\ra \C^3$ (\ref{EQN:RMIG_HAT}) is equal to
\begin{align}\label{JAC_IN_HOMOG}
\Jac \hat R \equiv   \det D\hat R = 32\, V\, (UW-V^2)\, (U+W)^2\, (U^2+V^2)\, (W^2+V^2),
\end{align}
and therefore the critical locus of $\Rmig$ consists of               
6 complex lines and one conic:
\begin{eqnarray*}
\Lzero &:=& \{V=0\}      = \mbox{line at infinity},\\
\Lone &:=& \{UW= V^2\} = \mbox{conic}\ \{ uw=1\}, \\
\Ltwo &:=& \{U = -W\}  =  \{ u=-w\} = \mbox{ the collapsing line},
                \\
\Lthree^\pm &:=& \{U = \pm i V \} = \{ u=\pm i\}, \\
\Lfour^\pm &:=& \{W = \pm i V\} =  \{ w=\pm i\}.
\end{eqnarray*}
(Here the  curves are written in the homogeneous coordinates $(U:V:W)$
and in the affine ones, $(u=U/V, w=W/V)$.)
The critical locus is schematically depicted on Figure~\ref{FIG:CRITICAL_CURVES}, while its image, the critical value locus,  is
depicted on Figure~\ref{FIG:MK_CONFIGURATION}.  

\begin{figure}
\begin{center}

\begin{picture}(0,0)%
\includegraphics{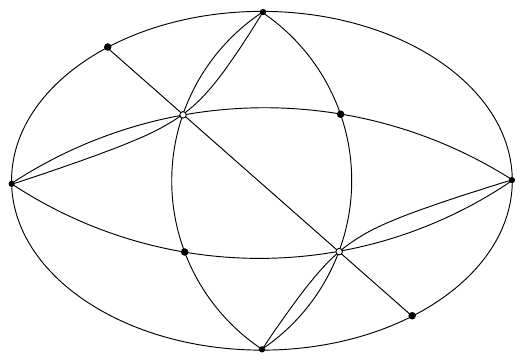}%
\end{picture}%
\setlength{\unitlength}{3947sp}%
\begingroup\makeatletter\ifx\SetFigFont\undefined%
\gdef\SetFigFont#1#2#3#4#5{%
  \reset@font\fontsize{#1}{#2pt}%
  \fontfamily{#3}\fontseries{#4}\fontshape{#5}%
  \selectfont}%
\fi\endgroup%
\begin{picture}(4247,3107)(4032,-6056)
\put(4126,-3286){\makebox(0,0)[lb]{\smash{{\SetFigFont{8}{9.6}{\familydefault}{\mddefault}{\updefault}{\color[rgb]{0,0,0}$\CROSSING:=[1:0:-1]$}%
}}}}
\put(4788,-4429){\makebox(0,0)[lb]{\smash{{\SetFigFont{8}{9.6}{\familydefault}{\mddefault}{\updefault}{\color[rgb]{0,0,0}Separatrix}%
}}}}
\put(6158,-6010){\makebox(0,0)[lb]{\smash{{\SetFigFont{8}{9.6}{\familydefault}{\mddefault}{\updefault}{\color[rgb]{0,0,0}$\CFIXmig'$}%
}}}}
\put(7444,-5705){\makebox(0,0)[lb]{\smash{{\SetFigFont{8}{9.6}{\familydefault}{\mddefault}{\updefault}{\color[rgb]{0,0,0}$[1:0:-1]$}%
}}}}
\put(7406,-3375){\makebox(0,0)[lb]{\smash{{\SetFigFont{8}{9.6}{\familydefault}{\mddefault}{\updefault}{\color[rgb]{0,0,0}$\Lzero$}%
}}}}
\put(7133,-3261){\makebox(0,0)[lb]{\smash{{\SetFigFont{8}{9.6}{\familydefault}{\mddefault}{\updefault}{\color[rgb]{0,0,0}Separatrix}%
}}}}
\put(8264,-4510){\makebox(0,0)[lb]{\smash{{\SetFigFont{8}{9.6}{\familydefault}{\mddefault}{\updefault}{\color[rgb]{0,0,0}Fixed point $\CFIXmig$}%
}}}}
\put(5597,-3510){\makebox(0,0)[lb]{\smash{{\SetFigFont{8}{9.6}{\familydefault}{\mddefault}{\updefault}{\color[rgb]{0,0,0}$\Lfour^-$}%
}}}}
\put(7311,-4023){\makebox(0,0)[lb]{\smash{{\SetFigFont{8}{9.6}{\familydefault}{\mddefault}{\updefault}{\color[rgb]{0,0,0}$\Lthree^+$}%
}}}}
\put(6051,-4329){\makebox(0,0)[lb]{\smash{{\SetFigFont{8}{9.6}{\familydefault}{\mddefault}{\updefault}{\color[rgb]{0,0,0}$\Ltwo$}%
}}}}
\put(6158,-4450){\makebox(0,0)[lb]{\smash{{\SetFigFont{8}{9.6}{\familydefault}{\mddefault}{\updefault}{\color[rgb]{0,0,0}collapsing}%
}}}}
\put(6318,-4576){\makebox(0,0)[lb]{\smash{{\SetFigFont{8}{9.6}{\familydefault}{\mddefault}{\updefault}{\color[rgb]{0,0,0}line}%
}}}}
\put(5624,-5023){\makebox(0,0)[lb]{\smash{{\SetFigFont{8}{9.6}{\familydefault}{\mddefault}{\updefault}{\color[rgb]{0,0,0}$-(i,i)$}%
}}}}
\put(6891,-3903){\makebox(0,0)[lb]{\smash{{\SetFigFont{8}{9.6}{\familydefault}{\mddefault}{\updefault}{\color[rgb]{0,0,0}$(i,i)$}%
}}}}
\put(7071,-4764){\makebox(0,0)[lb]{\smash{{\SetFigFont{8}{9.6}{\familydefault}{\mddefault}{\updefault}{\color[rgb]{0,0,0}$\Lone$}%
}}}}
\put(6590,-3483){\makebox(0,0)[lb]{\smash{{\SetFigFont{8}{9.6}{\familydefault}{\mddefault}{\updefault}{\color[rgb]{0,0,0}$\Lfour^+$}%
}}}}
\put(5551,-4150){\makebox(0,0)[lb]{\smash{{\SetFigFont{8}{9.6}{\familydefault}{\mddefault}{\updefault}{\color[rgb]{0,0,0}$\INDmig_-$}%
}}}}
\put(6798,-5244){\makebox(0,0)[lb]{\smash{{\SetFigFont{8}{9.6}{\familydefault}{\mddefault}{\updefault}{\color[rgb]{0,0,0}$\INDmig_+$}%
}}}}
\put(7357,-5040){\makebox(0,0)[lb]{\smash{{\SetFigFont{8}{9.6}{\familydefault}{\mddefault}{\updefault}{\color[rgb]{0,0,0}$\Lthree^-$}%
}}}}
\put(5396,-3060){\makebox(0,0)[lb]{\smash{{\SetFigFont{8}{9.6}{\familydefault}{\mddefault}{\updefault}{\color[rgb]{0,0,0}Fixed point $\CFIXmig'$}%
}}}}
\put(4047,-4550){\makebox(0,0)[lb]{\smash{{\SetFigFont{8}{9.6}{\familydefault}{\mddefault}{\updefault}{\color[rgb]{0,0,0}$\CFIXmig$}%
}}}}
\put(5061,-4543){\makebox(0,0)[lb]{\smash{{\SetFigFont{8}{9.6}{\familydefault}{\mddefault}{\updefault}{\color[rgb]{0,0,0}$\Lone$}%
}}}}
\end{picture}%

\end{center}
\caption{\label{FIG:CRITICAL_CURVES}Critical locus for $\Rmig$ shown with the separatrix $\Lzero$ at infinity.}
\end{figure}

\begin{figure}
\begin{center}

\begin{picture}(0,0)%
\includegraphics{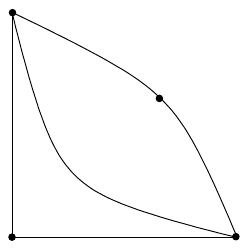}%
\end{picture}%
\setlength{\unitlength}{3947sp}%
\begingroup\makeatletter\ifx\SetFigFont\undefined%
\gdef\SetFigFont#1#2#3#4#5{%
  \reset@font\fontsize{#1}{#2pt}%
  \fontfamily{#3}\fontseries{#4}\fontshape{#5}%
  \selectfont}%
\fi\endgroup%
\begin{picture}(2355,2170)(4336,-4157)
\put(6045,-2841){\makebox(0,0)[lb]{\smash{{\SetFigFont{8}{9.6}{\familydefault}{\mddefault}{\updefault}{\color[rgb]{0,0,0}$\FIXmig_0 = \tilde{R}(\tl \Ltwo) = [1:0:1]$}%
}}}}
\put(6676,-3961){\makebox(0,0)[lb]{\smash{{\SetFigFont{8}{9.6}{\familydefault}{\mddefault}{\updefault}{\color[rgb]{0,0,0}$\CFIXmig = [1:0:0]$}%
}}}}
\put(4801,-2086){\makebox(0,0)[lb]{\smash{{\SetFigFont{8}{9.6}{\familydefault}{\mddefault}{\updefault}{\color[rgb]{0,0,0}$\CFIXmig'=[0:0:1]$}%
}}}}
\put(5198,-3269){\makebox(0,0)[lb]{\smash{{\SetFigFont{8}{9.6}{\familydefault}{\mddefault}{\updefault}{\color[rgb]{0,0,0}$\Lone$}%
}}}}
\put(4726,-4111){\makebox(0,0)[lb]{\smash{{\SetFigFont{8}{9.6}{\familydefault}{\mddefault}{\updefault}{\color[rgb]{0,0,0}$\B0 = [0:1:0]$}%
}}}}
\put(4351,-3136){\makebox(0,0)[lb]{\smash{{\SetFigFont{8}{9.6}{\familydefault}{\mddefault}{\updefault}{\color[rgb]{0,0,0}$\Rmig(\Lthree^\pm)$}%
}}}}
\put(5428,-3878){\makebox(0,0)[lb]{\smash{{\SetFigFont{8}{9.6}{\familydefault}{\mddefault}{\updefault}{\color[rgb]{0,0,0}$\Rmig(\Lfour^\pm)$}%
}}}}
\put(5551,-2461){\makebox(0,0)[lb]{\smash{{\SetFigFont{8}{9.6}{\familydefault}{\mddefault}{\updefault}{\color[rgb]{0,0,0}$\Lzero$}%
}}}}
\end{picture}%

\end{center}
\caption{Critical values locus of $R$.\label{FIG:MK_CONFIGURATION}}
\end{figure}

It will be helpful to also consider the critical locus for the lift $\tl \Rmig:
\tl \CP^2 \ra \CP^2$.  Each of the critical curves $L_i$ lifts by
proper transform (see \cite[Appendix A.2]{BLR1}) to a critical curve $\tl L_i
\subset \tl \CP^2$ for $\tl \Rmig$.
Moreover,
any critical point for $\tl \Rmig$ belongs to either one of these proper transforms or to one of
the exceptional divisors $L_\ex(\INDmig_\pm)$.

By symmetry, it is enough to consider the blow-up of $a_+$.  We saw in Part~I
that there are four critical points on the exceptional divisor
$L_\ex(\INDmig_+)$ occurring at $\chi = -1, 1, \infty,$ and $0$, where $\chi = (w+i)/(u-i)$.
They correspond to intersections of $L_\ex(\INDmig_+)$ with the
collapsing line $\tl L_2$, the  $\tl L_1$, and the critical lines
$\tl L_3^+$ and $\tl L_4^-$, respectively.

\begin{rem}\label{REM:WHITNEY}
In Part I we showed that all of the critical points of $\tl R$ except the fixed points $e,e'$,
the collapsing line $\tl \Ltwo$, and two points $\{\pm (i,i)\} = {\tl \Lthree}^\pm \cap {\tl \Lfour}^\pm$,
are degree two {\em Whitney folds}, i.e.\ they can be brought into the normal form $(x,y) \mapsto (x,y^2)$ in holomorphic coordinates.  See \cite[Appendix D.2]{BLR1} and also Lemma \ref{LEM:NORMALFORM}.
\end{rem}

\subsection{Local study of the critical locus of $R$ and $\tilde{R}$} \label{SUBSEC:VANISHIN_ORDER}

In the proof of the Equidistribution Theorem for the DHL we will need details
about how the Jacobian of $R$ vanishes, when $R$ is expressed in local
coordinates; See Lemmas \ref{LEM:FJ_ESTIMATE} and~\ref{LEM:SPECIAL_JACOBIAN}.

These details can be recovered from Formula (\ref{JAC_IN_HOMOG}) for $\Jac \hat
R$ as follows: Suppose $z \in N \subset  \CP^2 \setminus I(R)$, with $N$ an
open set admitting a local section $s: N \rightarrow \C^3 \setminus \{0\}$ of
the canonical projection $\pi: \C^3 \setminus \{0\} \rightarrow \CP^2$.  If we
express $R$ in local coordinates on $N$ (in the domain) and on $R(N)$ (in the
codomain), then the complex Jacobian of this local expression for $R$ differs
from $\Jac \hat R \circ s$ by a non-vanishing analytic function.

For any $z \in \CP^2 \setminus I(R)$ define $\mu(z,R)$\label{DEF_MU} to be the order of
vanishing at $z$ for the complex Jacobian of any local coordinate expression
for $R$.  By the chain rule, the result is independent of the choice of
charts.  Moreover, by the discussion in the previous paragraph, $\mu(z,R)$
equals the order of vanishing of $\Jac \hat R \circ s$ at~$z$.

\begin{lem}\label{LEM:MULTIPLICITIES_C2}  Let $c$ be the point of intersection between $L_0$ and $L_2$.  We have
\begin{itemize}
\item[(a)] $\mu(z,R) \leq 2$ for any $z \in \CP^2 \setminus \{\CROSSING,\CFIXmig,\CFIXmig',\INDmig_\pm\}$,
\item[(b)] $\mu(\CROSSING,R) = 3$, moreover, there are local coordinates $(x,y)$ centered at $\CROSSING$ in which $\Jac R  \asymp xy^2$, and
\item[(c)] $\mu(\CFIXmig,R) = \mu(\CFIXmig',R) = 4$.
\end{itemize}
\end{lem}

\begin{proof}
Any $z \in \CP^2 \setminus \{\pm(i,i),\CROSSING,\INDmig_\pm,\CFIXmig,\CFIXmig'\}$ is either regular
or is a smooth point of the critical locus; See
Figure~\ref{FIG:CRITICAL_CURVES}.  Since each of the irreducible factors of
$\Jac \hat R$ occurs to the first or second power, at any smooth point $z$ of the
critical locus we have $\mu(z,R) \leq 2$.

Since the lines $L_3^\pm$ and $L_4^\pm$ intersect transversally at $\pm(i,i)$,
we can choose local coordinates $(x,y)$ centered at $\pm(i,i)$ so that
$L_3^\pm$ is given by $\{x=0\}$ and $L_4^\pm$ is given by $\{y=0\}$.  As the
irreducible factors of $\Jac \hat R$ corresponding to $L_3^\pm$ and $L_4^\pm$
occur to the first power, near the origin in these coordinates we have $\Jac R \asymp xy$, giving that $\mu(\pm(i,i),R) = 2$.

Similarly, we can choose local coordinates $(x,y)$ centered at 
$\CROSSING$ with $L_0$ corresponding to $\{x=0\}$
and $L_2$ corresponding to $\{y=0\}$.  As the irreducible factors of $\Jac \hat
R$ corresponding to these two lines have have first and second powers,
respectively, in the expression for $\Jac \hat R$ we have $\Jac R \asymp xy^2$,
implying $\mu(\CROSSING,R) = 3$.

The four separate critical lines $L_0,L_1,L_3^+$ and $L_3^-$ meet at $\CFIXmig$ and $\Jac \hat R$ vanishes to order one along
each of them, so that $\mu(\CFIXmig,R) = 4$.  The same result holds at~$\CFIXmig'$, by symmetry.
\end{proof}

To deal with the indeterminate points $\INDmig_\pm$ we need the following:

\begin{lem}  \label{LEM:MULTIPLICITIES_EXCEPTIONAL}
For every $\tilde{z} \in L_\ex(\INDmig_+) \cup L_\ex(\INDmig_-)$ we have $\mu(\tilde{z},\tilde{R}) \leq 2$.
\end{lem}

\begin{proof}
By symmetry, we can focus on $\INDmig_+$.  There is a neighborhood $N$ of the
exceptional divisor $L_\ex(\INDmig_+)$ in which the critical locus of
$\tilde{R}$ consists of the proper transforms $\tl L_1$, $\tl L_2$, $\tl L_3^+$
and $\tl L_4^-$ and in which these curves are disjoint.  As these four critical
curves of $\tilde{R}$ are smooth and disjoint in $N$, the order on vanishing of
$\Jac \tilde{R}$ is locally constant on each of them.  Since each of
the defining equations for $L_1, L_3^+,$ and $L_4^-$ occur to the first power
in $\Jac \hat R$, we have $\mu(\tilde{z},\tilde{R}) = 1$ for any $\tilde{z} \in N$ that is
on  $\tl L_1$, $\tl L_3^+$ or $\tl L_4^-$.   Meanwhile, the defining equation for
$L_2$ occurs to the second power in $\Jac \hat R$, so that $\mu(\tilde{z},\tilde{R}) = 2$ for any $\tilde{z} \in N \cap \tilde L_2$.
\end{proof}


\section{Fatou and Julia sets and the measure of maximal entropy}
\label{SUBSEC:JULIA_AND_FATOU}

\subsection{Julia set}

For a rational map $R:\CP^n \ra \CP^n$, the {\em Fatou set} $F_R$ is defined to
be the maximal open set on which the iterates $\{R^m\}$ are well-defined and
form a normal family.  The complement of the Fatou set is the {\it Julia set}
$J_R$.  

If $R$ is dominant and has no collapsing varieties, 
Lemma~A.1 from \cite{BLR1} gives that $R$ is locally surjective (except at
indeterminate points), so that the Fatou set is forward invariant and
consequently, the Julia set is backward invariant.

If $R$ has indeterminate points, then according to this definition they are in
$J_R$.  In this case, $F_R$ and $J_R$ are not typically totally invariant.  One
can see this by blowing up an indeterminate point and observing that the image
of the resulting exceptional divisor typically intersects $F_R$.  Note also that 
any algebraic curve $A$ that is mapped by some iterate of $R$ to an
indeterminate point (such a curve exists iff $R$ is not algebraically stable) is in $J_R$.  

\msk 
The Migdal-Kadinoff renormalization $\Rmig$ is not locally surjective at any 
$x~\in~\Ltwo~\sm~\Lzero$.  More specifically, if $N$ is a small neighborhood of
$x$ that avoids $\Lzero$, then 
\begin{align*}
\Rmig(N) \cap \Lzero = \FIXmig_0 = \Rmig(x),
\end{align*}
where $\FIXmig_0 = [1:0:1]$,
since any point of $\Lzero \sm \{\FIXmig_0\}$ has preimages only in $\Lzero$.
However, we still have the desired invariance:

\begin{lem} The Migdal-Kadinoff renormalization $\Rmig$ has forward invariant Fatou set and, consequently, backward invariant Julia set.
\end{lem}

\begin{proof}
It suffices to show that $\Ltwo \subset J_\Rmig$, since $\Rmig$ is locally 
surjective at any other point, by Lemma A.1 from \cite{BLR1}.  By definition,
$\{\INDmig_\pm\} \subset J_\Rmig$, so we consider
$x~\in~\Ltwo~\sm~\{\INDmig_\pm\}$.  Let $N$ be any small neighborhood  of $x$.
Note that $\Rmig(x) = \FIXmig_0$ is a fixed point of saddle-type, with one-dimensional
stable and unstable manifolds.   Therefore, in order for the iterates to form a normal
family on $N$, we must have $\Rmig(N) \subset \WW^s(\FIXmig_0)$.  However, this
is impossible, since there are plenty of regular points for $\Rmig$ in $N$.
\end{proof}

Lemmas \ref{nbd of B}
and \ref{LEM:REGULARITY} give a clear picture of $J_\Rmig$ in a neighborhood of
the line at infinity $\Lzero$.  

\begin{prop}\label{PROP:JULIA_NEAR_LZERO}
Within some neighborhood $N$ of $\Lzero$ we have that $J_\Rmig \cap N =
\WW^s_{\C,\loc}(\BOTTOMmig) \cap N$.  Within this neighborhood, $J_\Rmig$ is a
$C^\infty$ 
$3$-dimensional manifold.
\end{prop}

Let us consider the locus  $\{h=0\}$ of vanishing magnetic field in $\CP^2$ for the DHL. 
In the affine coordinates, it is an $R$-invariant line $\LINV = \{u = w\}$;
in the physical coordinates, it is an $\Rphys$-invariant line $\LLINV=\{z=1\}$.
The two maps are conjugate by the M\"obius transformation $\LLINV\ra \LINV$, $u=1/t$. 
Dynamics of $\displaystyle{\Rphys: t\mapsto  \left( \frac {2t}{t^2+1} \right)^2}$ on $\LLINV$
was studied in \cite{BL}.
In particular, it is shown in that paper that the Fatou set for $\Rphys|\, \LLINV$ consists entirely of the
basins of attraction of the fixed points $\FIXphys_0 := \{t=0\}$ and $\FIXphys_1:=\{t=1\}$ 
which are superattracting within this line: see~Figure~\ref{FIG:INVARIANT_LINE_JULIA}.  
Under the conjugacy, the Fatou set for
$\Rmig| \LINV$ consists of the basins of attraction for the two fixed points
\begin{align*}
b_0 := \{u=\infty\} = [1:0:1] \qquad \mbox{and} \qquad b_1 := \{u=1\} = [1:1:1].
\end{align*}

\begin{prop}\label{PROP:JULIA}
$J_\Rmig = \overline{\bigcup_n \Rmig^{-n}(\LINV)}$. 
\end{prop}

\begin{proof}
Since every point in the Fatou set of $\Rmig |\LINV$ is in the basin of attraction of either $\FIXmig_0$ or $\FIXmig_1$ and
since these points are of saddle-type in $\CP^2$, the family of iterates $\Rmig^n$ cannot be normal
in an open neighborhood of any point on $\LINV$.  Thus $\LINV \subset J_\Rmig$.
It follows that $\overline{\bigcup_n \Rmig^{-n}(\LINV)} \subset J_\Rmig$ since $J_\Rmig$ is
closed and backward invariant.

We will now show that $\bigcup_n \Rmig^{-n}(\LINV)$ is dense in $J_\Rmig$.
Consider a 
configuration of five algebraic curves
\begin{eqnarray*}
X_0 &:=& \{V = 0\}  = \mbox{the separatrix}\ \Lzero, \\
X_1 &:=& \{U = W\}  =\mbox{the invariant line}\ \LINV,  \\
X_2 &:=& \{U = -W\}  =\mbox{the collapsed line}\ \Ltwo\subset \Rmig^{-1}(\LINV), \\
X_3 &:=& \{U^2+2V^2+W^2 = 0\} =\mbox{a component of}\ \Rmig^{-1}(\LINV),\\ 
X_4 &:=& \{U^4+2U^2V^2+2V^4+2W^2V^2+W^4=0\} =\mbox{a component of}\  \Rmig^{-1}(X_3).
\end{eqnarray*}
\noindent

We will 
 use the results of M.~Green  to check that the complement of these curves,
$M:=\CP^2 \sm \bigcup_{i} X_i$, 
is a complete Kobayashi hyperbolic manifold hyperbolically embedded in $\CP^2$
(see Appendix \ref{APP:KOB}).
We will first check that $M$ is Brody hyperbolic,
i.e.,  there are no non-constant holomorphic maps $f:\C \ra M$. 
To this end we will apply Green's Theorem \ref{conic and 3 lines}. 
It implies  that the image of $f$  must lie in a line $L$ that is tangent to the conic $X_3$
at an intersection point  with $X_i$, for one of the lines $X_i$, $i=0,1,2$, and
contains the intersection point  $X_j \cap X_l$ of the other two lines.  
It is a highly degenerate situation which does not occur for a generic configuration.
However, this is exactly what happens in our case,  as
the lines $X_0,X_1,X_2$ form a self-dual triangle with respect to the conic $X_3$
(see \S \ref{self-dual triangles}).  
However, one can check by direct calculation
that the last curve, $X_4$, must intersect each of
these tangent lines $L$ in at least one point away from $X_0,\ldots,X_3$.  Since any 
holomorphic map from $\C$ to $L \sm \bigcup_{i} X_i$ must then omit $3$ points in
$L$, it must be constant.

So, $M$ is Brody hyperbolic.  Moreover, for each $i=0,\ldots,4$ the remaining
curves $\bigcup_{j \neq i} X_j$ intersect $X_i$ in at least three points 
so that there is no non-constant holomorphic map from $\C$ to
$X_i \sm \bigcup_{j \neq i} X_j$.  Therefore, another of Green's results (Thm.~\ref{THM:GREEN1}) applies  showing that $M$ is complete hyperbolic and
hyperbolically embedded.  It then follows from Proposition~\ref{COR:NORMAL_FAMILIES} that the family $\{\Rmig^n\}$ is normal on any open
set $N \subset \CP^2$ for which $\Rmig^n: N \ra M$ for all $n$.
\msk

Given any $\zeta \in J_\Rmig$ and any neighborhood $N$ of
$\zeta$, we'll show that some preimage $\Rmig^{-n}(\LINV)$ intersects $N$.
Since $\zeta \in J_\Rmig$, the family of iterates $\Rmig^n$ are not normal on $N$,
hence $\Rmig^n(N)$ must intersect $\bigcup_{i} X_i$ for some $n$.  If the
intersection is with $X_i$ for $i>0$ then  $R^{n+2}(N)$  intersects $\LINV$.

So, some iterate $\Rmig^n(N)$ must intersect $X_0=\Lzero$. 
Suppose first that $\zeta \in \Lzero$.  
Then,  by Lemma \ref{nbd of B},
$\zeta\subset \WW^s(\CFIXmig)\cup  \WW^s(\CFIXmig')\cup \T$.
Since the first two basins are contained in the Fatou set,
$\zeta \in \T$,  where preimages of the fixed point $b_0\in \LINV$ are dense.  

Finally, assume $\zeta\not\in \Lzero$. By shrinking $N$ if needed,
we can make it disjoint from $\Lzero$. 
Hence there is $n>0$  such that   $R^n(N)$ intersects $\Lzero$,
while $R^{n-1}(N)$ is disjoint from $\Lzero$. 
But since $R^{-1}(\Lzero)=\Lzero\cup \Ltwo$, 
we conclude that $R^{n-1}(N)$  must intersect $\Ltwo$.
But $\Ltwo$ collapses under $R$ to the fixed point $b_0 \in \LINV$.  
Hence $R^n(N)$  intersects $\LINV$.
\end{proof}

\begin{rem}
We thank the referee for pointing out that
the above proof is similar to the one that Bonifant and Dabija use to show that
if an endomorphism $f: \mathbb{CP}^2 \rightarrow \mathbb{CP}^2$ of positive degree
has an invariant elliptic curve $Q$ then any point of $Q$ has backward orbit
under $f$ that is dense in the Julia set $J_f$; see \cite[Thm. 5.4]{BD}.
\end{rem}

We will now relate $J_R$ to the Green current $\Green$.  (See Appendix
\ref{APP:COMPLEX_DYNAMICS} for the definition and basic properties of
$\Green$.) 

\begin{prop}\label{PROP:JULIA2} $J_R = \supp \Green$.
\end{prop}

\begin{proof}

The inclusion $\supp \Green \subset J_R$ follows immediately from Theorem~\ref{Green current vs Fatou}.  We will use Proposition~\ref{PROP:JULIA} to
show that $J_R \subset \supp \Green$.  Since $\supp \Green$ is a backward
invariant closed set, it is sufficient for us to show that $\LINV \subset \supp
\Green$.

Note that $\LINV = \WW^s(\FIXmig_0) \cup \WW^s(\FIXmig_1) \cup
J_{\Rmig|\LINV}$.  The basin $\WW^s(\FIXmig_0)$ is open and contained within the
normal set for $\Rmig$ (see Appendix \ref{APP:COMPLEX_DYNAMICS} for the
definition of normal set).  Therefore, $\WW^s(\FIXmig_0) \subset J_R \cap
\Regular \subset \supp \Green$, by Theorem~\ref{Green current vs Fatou}.  Since
$\supp \Green$ is closed, we also have that $J_{\Rmig|\LINV} \subset \supp
\Green$.

Consider the ``top'' unit circle $\TOPmig= \{(u,\bar u): \ |u|=1\}$
and note that $R|\TOPmig$ is the squaring map $u \mapsto u^2$.  Since
the fixed point $\FIXmig_1$ and indeterminate
points $\INDmig_\pm$ are on $\TOPmig$, we therefore have
\begin{eqnarray*}
\FIXmig_1 \in \overline{\bigcup_{n \geq 0} \Rmig^{-n} \{\INDmig_\pm\}},
\end{eqnarray*}
implying that none of the points of $\WW^s(\FIXmig_1)$ are normal.
Hence we cannot directly use Theorem~\ref{Green current vs Fatou} to
conclude that $\WW^s(\FIXmig_1) \subset \supp \Green.$ 

Notice that the points of $\Ltwo \sm  \{\INDmig_\pm\}$ are normal, since they
are in $\WW^s(\Lzero)$.  Theorem~\ref{Green current vs Fatou} gives that $\Ltwo
\sm  \{\INDmig_\pm\} \subset \supp \Green$, since $\Ltwo \subset J_\Rmig$.
Because $\supp \Green$ is closed, $\Ltwo \subset \supp \Green$.  Let $D_2
\subset \Ltwo$ be a small disc centered around $\INDmig_+$.  Preimages of $D_2$
under appropriate branches of $\Rmig^{n}$ will give discs intersecting $\Lone$
transversally at a sequence of points converging to $\FIXmig_1$.
By the Dynamical $\lambda$-Lemma (see \cite[pp.  80-84]{PDM}), this
sequence of discs will converge to $\WW^s_0(\FIXmig_1) \subset \LINV$, where
$\WW^s_0(\FIXmig_1)$ is the immediate basin of $\FIXmig_1$.  Since each of the
discs is in $\supp \Green$, and the latter is closed, we find that
$\WW^s_0(\FIXmig_1) \subset \supp \Green$.  Further preimages show that all of
$\WW^s(\FIXmig_1) \subset \supp \Green$.

\comment{*****************************************************
An alternative approach is to show that 
\begin{eqnarray*}
J_R = \overline{\bigcup_{n \geq 0} \Rmig^{-n} \Ltwo}.
\end{eqnarray*}
However, computations using M. Green's theorems seem tricky.
************************}

\end{proof}

\subsection{Fatou Set}
\label{SUBSEC:SOLID_CYLINDERS}

Because $J_\Rmig = \supp \Green$, we immediately have:

\begin{cor}\label{COR:PSEUDO}The Fatou set of $\Rmig$ is pseudoconvex.
\end{cor}

For the definition of pseudoconvexity, see \cite{KRANTZ}.

\begin{proof}
It is well-known that in the complement in $\CP^2$ of the support of a closed positive $(1,1)$-current 
is pseudoconvex.  See \cite[Thm. 6.2]{CEGRELL} or 
\cite[Lem. 2.4]{Ueda}.
\end{proof}

\begin{rem}We thank the referee for pointing out that
Corollary \ref{COR:PSEUDO} can also be obtained directly from Proposition~\ref{PROP:JULIA}.  Suppose that $\LINV = \{l_{\rm inv}~=~0\}$ and  $\Lzero =
\{l_0~=~0\}$.  For any $N \geq 0$ one can define  a holomorphic function which
does not extend to $\bigcup_{n=0}^N R^{-n}\LINV$ by
\begin{eqnarray*}
z \mapsto \frac{\left(l_0(\hat z)\right)^k}{\left(\prod_{n=0}^N l_{\rm inv} \circ \hat \Rmig^n(\hat z)\right)}, \text{   where  } \pi(\hat z) = z \text{  and  } k = \sum_{n=0}^N 4^n.
\end{eqnarray*}
Therefore, 
the Fatou set of $\Rmig$ is the interior of the intersection of
domains of holomorphy, so it is also a domain of holomorphy.  Hence, it is
pseudoconvex.  \end{rem}

Computer experiments indicate that the Fatou set of $\Rmig$ may be precisely the
union of the basins of attraction  $\WW^s(\CFIXmig)$ and $\WW^s(\CFIXmig')$ for the two
superattracting fixed points $\CFIXmig$ and $\CFIXmig'$.  See Problem \ref{PROB:GLOBAL_BASINS}.

We can prove the following more modest statement.
Consider the solid cylinders\footnote{They correspond to actual solid cylinders in the $(z,t)$ coords; see Cor. \ref{THM:SOLID_CYLINDERS}.}
\begin{eqnarray*}
\Solidmig &:=& \left\{[U:V:W] \,\,:\,\, \frac{V^2}{UW} \in [0,1]\,\, \mbox{and}\,\,\left|\frac{W}{U} \right| < 1\right\} \mbox{  and  }\\
\Solidmig' &:=& \left\{[U:V:W] \,\,:\,\, \frac{V^2}{UW} \in [0,1]\,\,\mbox{and}\,\, \left| \frac{W}{U} \right| > 1\right\}.  
\end{eqnarray*}
\noindent

\begin{thm}\label{PROP:SOLID_CYLINDERS2}
For the mapping $\Rmig$ we have $\Solidmig \subset \WW^s(\CFIXmig)$ and $\Solidmig' \subset \WW^s(\CFIXmig')$.
\end{thm}

\noindent
In the proof, we will need to use an important property of $\Rmig: \Cmig \ra
\Cmig$ that was proved in Part I.  Recall from \S \ref{SUBSEC:SEMICONJUGACY} that $\Cmig = \correspond(\Cphys)$ is the invariant real M\"obius
band and that $\Cmigbl = \Cmig \sm \BOTTOMmig$ is the topological annulus obtained by removing the ``core curve'' $\BOTTOMmig$.

The key property is:
\begin{itemize}
\item[(P9$'$)] Every proper path $\gamma$ in $\Cmigbl$ lifts under $\Rmig$
               to at least $4$ proper paths in $\Cmigbl$.
               If $\gamma$ crosses $\Imig$ at a single point, then $\Rmig^{-1}\gamma=\cup\, \de_i$.
\end{itemize}

\begin{proof}[Proof of Theorem~\ref{PROP:SOLID_CYLINDERS2}:]
It suffices to prove the proposition for $\Solidmig$, since the statement for $\Solidmig'$ follows from the symmetry $\rho$.

We will decompose $\Solidmig$ as a union of complex discs 
and show that each disc is in $\WW^s(\CFIXmig)$. 
Let 
$$\Secmig_c := \left\{[U:V:W] \,\,:\,\, \frac{V^2}{UW} = c \in [0,1]\right\},$$ and 
$$\Secmig_c^* := \left\{[U:V:W] \,\,:\,\, \frac{V^2}{UW} = c \in [0,1] \,\, \mbox{and}\,\,\left| \frac{W}{U} \right| < 1\right\},$$ 
\noindent
so that $\Solidmig = \bigcup_{c \in [0,1]} \Secmig_c^*$.  

The discs $\Secmig_0^*$ and $\Secmig_1^*$ are in $\WW^s(\CFIXmig)$ because they are each within
the forward invariant critical curves $L_0$ and $L_1$, respectively, on which
the dynamics is given by $(W/U) \rightarrow (W/U)^4$ and $(W/U) \rightarrow (W/U)^2$, respectively.

\ssk 
We now show that for any $c \in (0,1)$ we also have $\Secmig_c^* \subset
\WW^s(\CFIXmig)$.  In fact $\CFIXmig \in \Secmig_c^*$, so it suffices to
show that $\Rmig^n$ forms a normal family on $\Secmig_c^*$.  Consider any $x
\in \Secmig_c^*$.  If $x = \CFIXmig$, then 
$x \in \WW^s(\CFIXmig)$ so that $\Rmig^n$ is normal on some neighborhood of $x$ in $\Secmig_c^*$.

Now consider any $x \in \Secmig_c^* \sm \{\CFIXmig\}$.  There is a neighborhood of  
$N \subset \Secmig_c^*$ of $x$ with $\CFIXmig \not \in N$, on which we will show that $\Rmig^n$
forms a normal family.
Recall the family of curves $X_0,\ldots,X_4$ from the proof of
Proposition~\ref{PROP:JULIA}, where we showed that $\CP^2 \sm \bigcup_{i} X_i$
is complete hyperbolic and hyperbolically embedded.  We
will show for every $n$ that $\Rmig^n(N)$ is in $\CP^2 \sm \bigcup_{i} X_i$, so that $\Rmig^n$
is normal on $N$. 

Since $\Secmig_c^* \cap X_0 = \{\CFIXmig\}$, and $\CFIXmig \not \in N$, we have that $N
\cap X_0 = \emptyset.$   Therefore, by reasoning identical to that in the proof of Proposition~\ref{PROP:JULIA}, if $\Rmig^n(N)$ intersects $X_i$ for any $i=0,\ldots,4$ we must have
that some iterate $\Rmig^m(N)$ intersects $X_1 = \LINV$.

We will check that forward iterates of $\Rmig^n(\Secmig_c^*)$ are disjoint from 
$\LINV$, which is sufficient since $N \subset \Secmig_c^*$.
The line $\LINV$ intersects the invariant annulus $\Cmigbl$ in two properly
embedded radial curves, so Property (P9') gives that $(\Rmig^n)^* \LINV$
intersects $\Cmig$ in at least $2\cdot 4^n$ properly embedded radial curves.  

One can check that $\Secmig_c$ intersects the invariant annulus
$\Cmig$ in the horizontal curve 
$$\left\{[U:V:W] \,\,:\,\, \frac{V^2}{UW} = c \in [0,1] \,\, \mbox{and}\,\,\left| \frac{W}{U} \right| = 1\right\},$$
 which corresponds to $|u| = 1/\sqrt{c} > 1$ in the $u$ coordinate for $\Cmig$.
Therefore, the $2\cdot 4^n$ radial curves in $\Cmig$
from $(\Rmig^n)^* \LINV$ intersect $\Secmig_c$ in at least $2\cdot4^n$ distinct points
within $\Cmig$.

We will now show that these are the only intersection points between
$(\Rmig^n)^* \LINV$ and $\Secmig_c$ in all of $\mathbb{CP}^2$.  Since $\Rmig$
is algebraically stable, Bezout's Theorem gives \hbox{${\rm deg}(\Secmig_c)
\cdot {\rm deg}((\Rmig^n)^* \LINV)=2\cdot4^n$} intersection points, counted
with multiplicities, in all of $\mathbb{CP}^2$.  Therefore $\Secmig_c \cap
(\Rmig^n)^* \LINV~\subset~\Cmig$. 

Since $\Secmig_c^* \subset \Secmig_c$ with $\Secmig_c^* \cap \Cmig = \emptyset$, we conclude
that $\Secmig_c^* \cap (\Rmig^n)^* \LINV = \emptyset$ for ever $n$.
In other words, $\Rmig^n(\Secmig_c^*) \cap \LINV = \emptyset$ for ever $n$.
 Thus, the same holds for $N \subset \Secmig_c^*$, implying that
$\Rmig^n$ is a normal family on $N$.  
\end{proof}

Theorem~\ref{PROP:SOLID_CYLINDERS2} has an interesting consequence for $\Rphys$.
The fixed point $\CFIXmig'$ for $\Rmig$ has a single preimage $\CFIXphys'=\correspond^{-1}(\CFIXmig')$,
which is a superattracting fixed point for $\Rphys$.  However, $\CFIXmig$ has the entire collapsing line $Z=0$
as preimage under $\correspond$.  Within this line is another superattracting fixed point $\CFIXphys = [0:1:1]$ for $\Rphys$ and every point
in $\{Z=0\} \sm \{\B0,\gamma\}$ is collapsed by $\Rphys$ to $\CFIXphys$.

We obtain:

\begin{cor}\label{THM:SOLID_CYLINDERS}For the mapping $\Rphys$, the solid cylinder
$\{(z,t) : \vert z \vert < 1, t \in (0,1]\}$ is in $\WW^s(\CFIXphys)$ and,
symmetrically, the solid cylinder $\{(z,t) : \vert z \vert >1, t \in [0,1] \}$
is in $\WW^s(\CFIXphys')$.
\end{cor}

Notice that we had to omit the ``bottom'', $t=0$, of the solid cylinder in $\WW^s(\CFIXphys)$ because points on it are forward asymptotic
to the indeterminate point $\B0$.

\subsection{Measure of Maximal Entropy}
\label{SUBSEC:MEAS_MAX_ENT}

There is a conjecture specifying the expected ergodic properties of a dominant
rational map of a projective manifold\footnote{It is stated more generally in \cite{GUEDJ_ERGODIC},
for meromorphic maps of compact K\"ahler manifolds.} in terms
of the relationship between various {\em dynamical degrees} of the map; see
\cite{GUEDJ_ERGODIC}.

Since the Migdal-Kadanoff renormalization $\Rmig$ is an algebraically stable map of
$\CP^2$, there are only two relevant dynamical degrees, the topological degree
$\deg_{top} \Rmig$ and the algebraic degree $\deg \Rmig$, which satisfy
\begin{eqnarray*}
\deg_{top} R = 8 > 4 = \deg R.
\end{eqnarray*}
This case of {\em high topological degree} was studied by Guedj
\cite{GUEDJ_LARGE_DEGREE}, who made use of a bound on topological entropy obtained by
Dinh and Sibony \cite{DS_ENTROPY}.  In our situation, his results give

\begin{prop}
$\Rmig$ has a unique measure $\nu$ of maximal entropy $\log 8$ with the following properties
\begin{itemize}
\item[(i)] $\nu$ is mixing;
\item[(ii)] The Lyapunov exponents of $\nu$ are bounded below by $\log \sqrt 2$;
\item[(iii)] If $\theta$ is any probability measure that does not charge the postcritical set\footnote{Here, the postcritical set is defined without taking the closure.} of $\Rmig$, then
$8^{-n} (\Rmig^n)^* \theta \ra \nu$;
\item[(iv)] If $P_n$ is the set of repelling periodic points of $\Rmig$ of period $n$ 
then
$$8^{-n} \sum_{a \in P_n} \delta_a \ra \nu.$$ 
 {\rm (}In fact, it suffices to take just the repelling periodic points in  $\supp \nu$.{\rm )}
\end{itemize}
\end{prop}

The measure $\nu$ satisfies the backwards invariance $\Rmig^* \nu = 8
\nu$, hence its support is totally invariant.  In our situation, $\supp
\nu \subsetneq J_R$ because (for example) the points in
$\WW^s(\BOTTOMphys)$ are not in $\supp \nu$.  It can be thought of as the
``little Julia set'' within $J_\Rmig$ on which the ``most chaotic'' dynamics
occurs.  

\noindent
\begin{rem}
The statement of (iv) in \cite[Thm. 3.1]{GUEDJ_LARGE_DEGREE} does not
emphasize that one can restrict his or her attention to the periodic points within $\supp \nu$, 
but it follows from the proof in \cite{GUEDJ_LARGE_DEGREE} and the fact that
$\supp \nu$ is totally invariant.  See \cite[Thm. 1.4.13]{DS_SURVEY} for
the analogous argument for holomorphic $f: \CP^2 \ra \CP^2$.
\end{rem}

\begin{rem}We know very little about the support of $\nu$.  See Problem \ref{PROB:MEASURE_CP2}.
\end{rem}


\section{Volume Estimates} \label{SEC:VOL_EST_ONE_ITERATE}

This section is devoted to estimating the volume exponent $\sigma(z,f)$ in terms of the Jacobian $Jf$ and to studying the sets
\begin{align}\label{EQN:DEF_VOLUME_EXPONENT_SETS}
D_{\geq a} := \{z \in \CP^2 \, : \, \sigma(z,f) \geq a\} \quad \mbox{and} \quad D_{> a} := \{z \in \CP^2 \, : \, \sigma(z,f) > a\}
\end{align}
for various values of $a > 0$.

\begin{lem}\label{LEM:TUBE_VOL} Let $h(x,y)$ be a non-constant holomorphic function vanishing with multiplicity $\ell$ at $(0,0)$.  Then there is a neighborhood $N$ of $(0,0)$ in $\C^2$ and a constant $K > 0$ such that for any $s > 0$ we have
\begin{align*}
\vol(\{|h(x,y)| < s\} \cap N) \leq K s^{2/\ell}.
\end{align*}
\end{lem}

\begin{proof}
We can suppose that the coordinates $(x,y)$ satisfy $h(0,y) \not \equiv 0$.
The Weierstrass Preparation Theorem then gives that there is a sufficiently small bidisc $\D_\epsilon^2$ centered at $(0,0)$ in which
\begin{align*}
h(x,y) = \alpha(x,y)\left(y^\ell + \beta_{\ell-1}(x) y^{\ell-1} + \cdots + \beta_0(x)\right),
\end{align*}
where $\alpha(x,y)$ is a non-vanishing holomorphic function and  the coefficients $\beta_j(x)$ are holomorphic in $\D_\epsilon$.

Up to a multiplicative constant, we can suppose $\alpha(x,y) = 1$.  Then in each vertical slice we have
\begin{align*}
\{x = x_0\} \cap \{|h(x_0,y)| < s\}  \subset \{x = x_0\} \times \bigcup_{i=1}^\ell  \D_{s^{1/\ell}}(r_i)
\end{align*}
where $r_1,\ldots,r_\ell$ are the roots of $h(x_0,y)$, listed with multiplicities.  In particular,
\begin{align*}
{\rm area}(\{x = x_0\} \cap \{|h(x_0,y)| < s\}) \leq K_1 s^{2/\ell}
\end{align*}
for some constant $K_1$.  The result then follows by Fubini's Theorem, after integrating over all $x_0 \in \D_\epsilon$.
\end{proof}

We now estimate the volume exponent $\sigma(z,f)$ in terms of the order of vanishing $\mu(z,f)$ of the complex Jacobian
$\Jac f$ at $z$ (see p. \pageref{DEF:VOL_EXP}).

\begin{lem}[Favre-Jonsson {\cite[Prop.  6.3]{FAVRE_JONSSON}}]\label{LEM:FJ_ESTIMATE}
Let $U, V \subset \C^2$ and $f: U \rightarrow V$ be holomorphic and at most $d_{\rm top}$-to-one off of a measure zero subset of $U$.
Then for any $z \in U$ we have
\begin{eqnarray}
\sigma(z,f) \leq \mu(z,f)+1.
\end{eqnarray}
\end{lem}

\begin{proof}
Lemma~\ref{LEM:TUBE_VOL} gives a neighborhood $N$ of $z$ and constant $K_0 > 0$ for which
\begin{align*}
\vol(\{|\Jac f(x,y)|^2 < s\} \cap N) \leq K_0 s^{1/\mu},
\end{align*}
where $\mu \equiv \mu(z,f)$.  Let $X$ be a measurable subset of $N$
and choose $s$ so that $K_0 s^{1/\mu}~=~\frac{1}{2}\vol(X)$.  
The Chebyshev Inequality gives
\begin{align*}
\vol f(X) &\geq \frac{1}{d_{\rm top}} \int_X |\Jac f|^2 d \vol \geq \frac{s}{d_{\rm top}} \Big(\vol X - \vol(\{|\Jac f|^2 < s\})  \Big) \\
& \geq \frac{s}{2 d_{\rm top}} \vol X \geq K \left(\vol X\right)^{1+\mu}.
\end{align*}
for an appropriate constant $K > 0$.

\end{proof}

\begin{lem}\label{LEM:SHARP}
Suppose $f: U \rightarrow V$ satisfies the hypotheses of Lemma~\ref{LEM:FJ_ESTIMATE}.
If $z$ is a smooth point of the critical locus of $f$, then $\sigma(z,f) = \mu(z,f)+1$, 
i.e.\ the estimate from Lemma \ref{LEM:FJ_ESTIMATE} is sharp at such points.
\end{lem}

\begin{proof}
It suffices to prove that $\sigma(z,f) \geq \mu(z,f)+1$.
We will do this by showing that in any neighborhood $N$ of $z$ there is a constant $C > 0$ such that there are measurable sets $X\subset N$ of arbitrarily small
measure for which $\vol f(X) \leq C (\vol X)^{\mu+1}$.

Since $z$ is a smooth point of the critical locus, 
one can choose local coordinates $(x,y)$ centered at $z$ so that 
$\Jac f(x,y) \asymp y^\mu$.  Given any neighborhood $N$ of $z$ we can apply
a linear rescaling to our coordinates in order to assume that the unit bidisc $\mathbb{D}^2$ 
is contained in $N$.  
For any $0 < \epsilon < 1$ let 
$X := \mathbb{D} \times \mathbb{D}_\epsilon \subset N$.
We have
\begin{align*}
\vol f(X) & \leq \int_X |y|^{2\mu} d \vol  = \pi \int_{0}^{2\pi}
\int_{0}^{\epsilon} r^{2\mu} r dr d\theta = 2\pi^2
\frac{\epsilon^{2\mu+2}}{2\mu+2} \asymp (\vol X)^{\mu+1}.
\end{align*}
\end{proof}

\begin{rem}
Unlike in one-dimensional dynamics, one can have points with  $\sigma(z,f)>d$.
Consider $f: \C^2 \ra \C^2$ given by $f(x,y) = (xy^2,y^3)$.  One has $\Jac = 3
y^4$ so that $\sigma({\bm 0},f) = 5$, by Lemma \ref{LEM:SHARP}.
\end{rem}

\msk

If the zero set of $\Jac f$ has a normal crossing singularity at $z$, then one has the following stronger estimate:

\begin{lem}\label{LEM:SPECIAL_JACOBIAN}Let $U, V \subset \C^2$ and $f: U \rightarrow V$
be holomorphic and at most $d_{\rm top}$-to-one off of a measure zero subset of $U$.  If $\Jac f~\asymp~x^a y^b$ in suitable local coordinates $(x,y)$ centered
at $z \in U$. Then
\begin{align}
\sigma(z,f) \leq \max(a,b)+1.
\end{align}
\end{lem}

The proof of Lemma~\ref{LEM:SPECIAL_JACOBIAN} will use:

\begin{lem}\label{area pullback under Q}
 Let $Q: w\mapsto w^d$. For any measurable set $Y\subset \C$,
$$ \area (Q^{-1} Y) \leq (\area Y)^{1/d}.$$
\end{lem}

\begin{proof}
We can assume that $\area Y>0$.
  Let us take the radius $r>0$ such that $\pi r^2= \area Y$.
Let $Y_- = Y \cap \D_r$, $Y_+ = Y\sm Y_-$, $Y_c = \D_r\sm Y_-$.
 Then
\begin{align*}
     (\area Y)^{1/d} &= (\area \D_r)^{1/d} = \area (Q^{-1} \D_r) = \area (Q^{-1} Y_-)+ \area (Q^{-1} Y_c) \geq \\
   \area (Q^{-1} Y_-) &+ d\Jac Q^{-1}(r) \area Y_c \geq \area (Q^{-1} Y_-)+ \area (Q^{-1} Y_+) = \area (Q^{-1} Y).
\end{align*}
\end{proof}

\begin{proof}[Proof of Lemma \ref{LEM:SPECIAL_JACOBIAN}:]
Let $\gamma > \max(a,b)+1$ and let $N = \{|x| < \eps\} \times \{|y| < \eps\}$.
For any measurable set $X \subset N$ we have:
\begin{eqnarray*}
\vol f(X) &\geq \frac{1}{d_{\rm top}} \int_X |\Jac f|^2 d \vol \asymp \int_{X^v} |y|^{2b}  \int_{X_y^h} |x|^{2a} \, d\area (x)d\area (y),
\end{eqnarray*}
 where $X^v$ is the projection of $X$ onto the $y$-axis and $X^h_y$ are the slices of $X$ by horizontal lines.
The  inner integral above is exactly $\area (Q_{a+1}(X^h_y))/(a+1)^2$, where $Q_{a+1}(x)=x^{a+1}$.
By Lemma~\ref{area pullback under Q}, it is bounded below by $(\area X^h_y)^{a+1}/(a+1)^2$.
Using the
H\"older inequality, with $p = \gamma/(\gamma-1)$ and $q = \gamma$, we find
\begin{align*}
 \left(\int_{X_v} 1/|y|^{2b/(\gamma-1)} \, d\area (y) \right)^{\gamma-1} \int_{X^v} |y|^{2b} (\area X_y^h )^{a+1} \, d\area (y)  \\
 \geq \left(\int_{X^v} (\area X_y^h)^{(a+1)/\gamma} \, d\area(y)\right)^\gamma \geq  (\vol X)^\gamma.
\end{align*}
The conclusion follows since $b < \gamma -1$ implies that
$1/|y|^{2b/(\gamma-1)}$ is locally integrable.
\end{proof}

\begin{rem}
Favre and Jonsson prove a similar volume estimate within their study of the the exceptional set $\mathcal{E}_1$; see \cite[Prop. 7.1 and Lem. 7.2]{FAVRE_JONSSON}.  
\end{rem}

Let us return to the case that $f: \CP^2 \ra \CP^2$ is a rational map.
For any $z \in I(f)$ we can estimate $\sigma(z,f)$ by applying either
Lemma~\ref{LEM:FJ_ESTIMATE} or Lemma~\ref{LEM:SPECIAL_JACOBIAN}  at each point of the exceptional divisor $\pi^{-1}(z)$:

\begin{prop}\label{PROP:ESTIMATE_FOR_RATIONAL_MAP} Consider the resolution of indeterminacy {\rm (\ref{RESOLUTION})} for $f$.  For any $z \in \CP^2$ the exponent $\sigma(z,f)$ exists and satisfies
\begin{align}
\sigma(z,f) \leq \max_{\tilde{z} \in \pi^{-1}(\{z\})} \sigma(\tilde{z},\tilde{f}).
\end{align}
\end{prop}

\begin{proof}
Note that $\pi: \widetilde{\CP^2} \rightarrow \CP^2$ decreases volumes.  Therefore, for any $z
\in \CP^2$  and any $\gamma~>~\max_{\tilde{z} \in \pi^{-1}(\{z\})}\sigma(\tilde{z},\tilde{f})$ it suffices for us to find a neighborhood
$\tilde{N}$ of $\pi^{-1}(\{z\})$ and a constant $K > 0$ so that for any
measurable set $Y \subset \CP^2$ we have
\begin{align}
\vol_{\widetilde{\CP^2}} \left(\tilde{N} \cap \tilde{f}^{-1}(Y) \right) \leq K \left(\vol Y\right)^{1/\gamma}.
\end{align}
This follows by applying the definition of $\sigma(\tilde{z},\tilde{f})$ at
each point of $\pi^{-1}(\{z\})$ and using that $\pi^{-1}(\{z\})$ is compact.
\end{proof}

\begin{lem}\label{LEM:FINITELY_MANY_SIGMA_VALUES}
The volume exponent $\sigma$ satisfies:
\begin{itemize}
\item[(i)] $\sigma(z,f)$ is an integer away from $I(f)$ and singular points of ${\rm Crit}(f)$,
\item[(ii)] $\sigma(z,f)$ assumes finitely many values as $z$ varies over $\CP^2$, and
\item[(iii)] for any $a > 1$ the set $D_{\geq a}$ and $D_{> a}$ {\rm (}see
Equation \ref{EQN:DEF_VOLUME_EXPONENT_SETS}{\rm )} is algebraic, consisting of
finitely many algebraic curves together with finitely many isolated points.
\end{itemize}
\end{lem}

\begin{proof}
Property (i) is a consequence of Lemma \ref{LEM:SHARP} and Property (ii) follows 
because $I(f)$ and the set of singular points of ${\rm Crit}(f)$ are finite.  
Property (iii) follows from Lemma \ref{LEM:SHARP} and upper semicontinuity of $\sigma$:
\begin{align*}
\limsup_{z \rightarrow z_0} \sigma(z,f) \leq \sigma(z_0,f),
\end{align*}
which is a consequence of the definition of $\sigma$.
\end{proof}

\begin{prop}\label{PROP:CURVES_OF_VERY_HIGH_DEG_COLLAPSED}
Suppose that $C$ is an irreducible algebraic curve contained in~$D_{> d}$.  Then $C$ is collapsed by $f$.
\end{prop}

\begin{proof}
Suppose for contradiction that $C$ is not collapsed by $f$.  Then Lemma
\ref{LEM:NORMALFORM} gives that $C$ is a Whitney Fold of $f$, i.e.\
there exists $\FOLDEXPONENT \in \mathbb{Z}_+$ and a finite $S \subset C$ such that for any $p \in C \setminus S$ there
are systems of holomorphic coordinates $(x,y)$ centered at $p$ and $(z,w)$
centered at $f(p)$ in which \begin{align}\label{EQN:NORMAL_FORM}
(z,w) = f(x,y) = (x,y^{\FOLDEXPONENT}).
\end{align}
Moreover, for all $p \in C \setminus S$ we have $\sigma(p,f) = \FOLDEXPONENT > d$, since we suppose $C \subset D_{> d}$.

In these coordinates, for any $w_0 \neq 0$ we have that $f^{-1}(\{w=w_0\})$
intersects the $y$-axis transversally in $\FOLDEXPONENT > d$ points.
Let $L \subset \CP^2$ be a projective line through
$p$ that is tangent to the $y$-axis in these local coordinates.  If we let
$\Lambda$ be a complex projective line in $\CP^2$ that is tangent to
$\{w=w_0\}$ at $(0,w_0)$ and take $|w_0|$ sufficiently small, then $f^{-1}(\Lambda)$ will be an algebraic curve of degree $d$
that intersects $L$ transversally in $\FOLDEXPONENT > d$ points, by the stability of transverse intersections between analytic curves under
small perturbations.  This contradicts the Bezout Theorem.
\end{proof}


\section{Proof of the Equidistribution Theorem for the DHL} \label{SEC:DHL_EQUIDISTRIBUTION}

\subsection{Transformation of volume by $R$} \label{SEC:DHL_VOLUME_ESTS}

\begin{prop}\label{PROP:2ND_ITERATE_IS_GOOD}
There exist constants $K > 0$ and $1~<~\tau~<~\deg(R^2) = 4^2$ such that
\begin{align*}
\vol\left(R^{-2} (Y) \right) \leq K \left(\vol Y\right)^{1/\tau}
\end{align*}
for any
measurable $Y \subset \CP^2$.
\end{prop}

\begin{proof}
Let $\Omega$ be any forward invariant neighborhood of the two superattracting fixed points $\CFIXmig$ and $\CFIXmig'$. 
We will first prove that there exists $K' > 0$ such that
\begin{align}\label{EQN:DESIRED_ESTIMATE_FIRST_ITERATE}
\vol\left(R^{-1} (Y) \right) \leq K' \left(\vol Y\right)^{1/3}
\end{align}
for any
measurable $Y \subset \CP^2 \setminus \Omega$.  By compactness, it suffices to prove that every $z \in \CP^2 \setminus \Omega$
has volume exponent $\sigma(z,R) \leq 3$. 

Lemma \ref{LEM:MULTIPLICITIES_C2}(a) gives that for any $z \in \CP^2 \setminus \{\CROSSING,\CFIXmig,\CFIXmig',\INDmig_\pm\}$ we have $\mu(z,R) \leq 2$
and hence $\sigma(z,R) \leq 3$, by Lemma \ref{LEM:FJ_ESTIMATE}.
Meanwhile, Lemma \ref{LEM:MULTIPLICITIES_C2}(b) gives local coordinates $(x,y)$ centered at $\{\CROSSING\} = L_0 \cap L_2$ in which $\Jac R  \asymp xy^2$,
so that Lemma \ref{LEM:SPECIAL_JACOBIAN} gives $\sigma(\CROSSING,R) \leq 3$.

We now use Proposition~\ref{PROP:ESTIMATE_FOR_RATIONAL_MAP} to check that the
indeterminate points $\INDmig_\pm$ satisfy  $\sigma(\INDmig_\pm,R)~\leq~3$.  By
Lemmas \ref{LEM:MULTIPLICITIES_EXCEPTIONAL} and \ref{LEM:SPECIAL_JACOBIAN} we
have $\sigma(\tilde{z},\tilde{R}) \leq 3$ for every $\tilde{z} \in
L_\ex(\INDmig_+) \cup  L_\ex(\INDmig_-)$.  
This completes the proof of  (\ref{EQN:DESIRED_ESTIMATE_FIRST_ITERATE}).

\vspace{0.05in}

A calculation\footnote{We omit the calculation, but the reader can readily
check it using Maple.} shows that $\mu(\CFIXmig,DR^2) = 14$  and hence
Lemma~\ref{LEM:FJ_ESTIMATE} gives
$\sigma(\CFIXmig,R^2)~\leq~15$.  By symmetry, the same holds at $\CFIXmig'$.
Combined with (\ref{EQN:DESIRED_ESTIMATE_FIRST_ITERATE}), this completes the
proof of Proposition~\ref{PROP:2ND_ITERATE_IS_GOOD}.
\end{proof}

\begin{rem}
In Lemma \ref{LEM:MULTIPLICITIES_C2}{\rm (c)} we saw that $\mu(\CFIXmig,R) = 4$,
so that Lemma~\ref{LEM:FJ_ESTIMATE} gives $\sigma(\CFIXmig,R)~\leq~5$, which is
insufficient for our purposes.  Meanwhile, the four separate critical curves
$L_0,L_1,L_3^+$ and $L_3^-$  meeting at $\CFIXmig$ imply that $\det D R$ does
not have the form needed to apply Lemma~\ref{LEM:SPECIAL_JACOBIAN}.  {\rm (}The same
holds at the symmetric fixed point $\CFIXmig'$.{\rm )}  This is why we needed to pass
to the second iterate of $R$ in the proof of
Proposition~\ref{PROP:2ND_ITERATE_IS_GOOD}.   
\end{rem}

\subsection{Completing the proof of the Equidistribution Theorem for the DHL}\label{SUBSEC:DHL_PROOF}

Let $g:=R^2$ and $d:=16 = {\rm deg}(g)$.
It suffices to prove that
\begin{eqnarray}\label{EQN:DESIRED_EQUIDISTRIBUTION_G}
\frac{1}{d^n \deg(A)}(g^n)^* [A] \rightarrow  \Green
\end{eqnarray}
for any algebraic curve $A \subset \CP^2$ because
the Green current $S$ for $R$ is also the Green current for $g = R^{2}$
and because the normalized pullback $\frac{1}{4}R^*$ acts continuously on the  
space of closed positive $(1,1)$ currents and has $S$ as a fixed point.  

Let $\pi: \C^3 \setminus \{0\} \rightarrow \CP^2$ denote the canonical projection.
For any $z \in \CP^2$ we will denote by $\hat{z} \in \C^3 \setminus \{0\}$ any choice of a point
of $\pi^{-1}(z)$.
Let
\begin{align*}
A = \{z \in \CP^2 \,:\, P(\hat z) = 0\},
\end{align*}
with $P$ a homogeneous polynomial of degree $a = \deg(A)$.
We must show that the limit
\begin{equation}\label{F is Green}
\lim_{n\to \infty} \frac 1{a d^{n}}  \log |P \circ \hat g^n (\hat{z}) |
\end{equation}
exists in $L^1_\loc(\C^3)$ and is equal to the Green Potential
\begin{equation}\label{DEF_GREEN}
G(\hat z) := \lim_{n \to \infty} \frac 1{d^{n}} \log \|\hat g^n \hat z\|
\end{equation}
of $g$. 

\vspace{0.05in}

The homogeneous polynomial $P$ determines a section $s_P$ of the $a$-th
tensor power of the hyperplane bundle; see Appendix \ref{hyp bundle}.
For each $n \geq 0$ we will consider the function
\begin{align*}
\phi_n: \mathbb{CP}^2 \rightarrow [-\infty,\infty) \quad \mbox{where} \quad \phi_n(z) := \frac{1}{a d^{n}} \log \|s_P(g^n(z))\|,
\end{align*}
with $\|\cdot \|$ denoting the Hermitian norm on this bundle.  By definition of the norm,
\begin{align*}
\phi_n(z) = \frac 1{a d^{n}}  \log \frac{|P \circ \hat g^n (\hat{z}) |}{\|\hat g^n \hat z\|^a} = \frac 1{a d^{n}}  \log |P \circ \hat g^n (\hat{z}) | - \frac 1{d^{ n}} \log \|\hat g^n \hat z\|.
\end{align*}
The limit in (\ref{DEF_GREEN})
exists by the hypothesis that $R$ (and hence~$g$) is algebraically stable; see
Theorem~\ref{Green potential}.  Therefore, the desired convergence of currents
(\ref{EQN:DESIRED_EQUIDISTRIBUTION_G}) will follow from:

\begin{thm}\label{Herm norm}
$
\phi_n \to 0  \quad \mathrm{in}\ L^1_\loc(\CP^2)\quad \mathrm{as}\quad n\to \infty.
$
\end{thm}

\begin{proof}

We will use the following general convergence criterion:

\begin{lem}\label{generality}
  Let $\phi_n$ be a sequence of $L^2$ functions on a finite measure space $(X,m)$
with bounded $L^2$-norms. If $\phi_n\to 0 $ a.e. then $\phi_n\to 0 $ in $L^1$.
\end{lem}

\begin{proof}
Take any $\eps> 0$ and $\de>0$.
  By Egorov's Lemma,  there exists a set $X'\subset X$ with $m(X\sm X')<\eps$
such that $\phi_n\to 0$ uniformly on $X'$. So, eventually the sup-norms of the $\phi_n$ on $X'$ are
bounded by $\de$. Hence
\begin{align*}
  \int |\phi_n|\, dm = \int_{X'} |\phi_n|\,  dm + \int_{X\sm X'} |\phi_n|\, dm\leq \de \cdot m(X) + B\sqrt{\eps},
\end{align*}
where the last estimate follows from the Cauchy-Schwarz Inequality (with $B$ the $L^2$-bound on the $\phi_n$).
The conclusion follows.
\end{proof}


\begin{lem}\label{vol neighborhood A}
There exists $K > 0$ such that for any $s < 0$ we have
\begin{align*}
\vol ( \{\phi_0 < s\}) \leq K e^{2s}.
\end{align*}
\end{lem}

\begin{proof}
Since $A$ is compact, it suffices to work in a neighborhood of any point $z \in
A$.  Without loss of generality, we can suppose $z= [0:0:1]$ so that it is the origin in the affine coordinates
$(x,y) \mapsto [x:y:1]$.  In these coordinates, 
\begin{align*}
\frac{|P(x,y,1)|}{\|(x,y,1)\|^a}  \asymp |P(x,y,1)|.
\end{align*}
The local multiplicity of $P(x,y,1)$ at $(0,0)$ is less than or equal to $a = \deg(P)$.
The estimate then follows in a neighborhood of $z$ from Lemma~\ref{LEM:TUBE_VOL}.
\end{proof}

\begin{lem}\label{vol pullback}
For any measurable set $Y \subset \CP^2$ there exists $1 < \tau < d$ and  $K > 0$ such that
for any $n \geq 0$ we have
\begin{align}\label{EQN:DYNAMICAL_VOLUME_ESTIMATE}
 \vol g^{-n} Y \leq K \, (\vol Y)^{1/\tau^n}.
\end{align}
\end{lem}

\begin{proof}
The estimate for a single iterate of $g = R^2$ is Proposition~\ref{PROP:2ND_ITERATE_IS_GOOD}.  The result then
follows inductively if we let $K := K_0^s$, where $s = \sum_{n=0}^\infty \frac{1}{\tau^n}$ and $K_0$ is the constant given by Proposition~\ref{PROP:2ND_ITERATE_IS_GOOD}.
\end{proof}

We will estimate the distribution of the tails of the random variables $\phi_n$:

\begin{lem}\label{tales}
Let $M=\sup \phi_0(z)$. Then there exists $K > 0$ so that
\begin{align*}
   \vol \{ |\phi_n|>r\}\leq K \exp\left(-  2 r \left(\frac{d}{\tau}\right)^n\right) \quad \mathrm{for\ any}\ r> M d^{-n}.
\end{align*}
\end{lem}

\begin{proof}
We have:
\begin{align*}
X_n(r) &:= \{|\phi_n |>r\}  = \{\phi_0 \circ g^n > r d^{n} \}\cup  \{\phi_0 \circ g^n < - r  d^{n}\} \\
&=  \{ \phi_0 \circ g^n < -r d^{n}\}  = g^{-n} \{  \phi_0 < -r d^{n}\}.
\end{align*}
We have used that $\phi_0 < r d^{n}$ to see that the first term in the union is empty.
According to Lemma~\ref{vol neighborhood A} there is a constant $K_1 > 0$ such that
\begin{align*}
\vol \{ \phi_0 < -r d^{n}\} \leq K_1 \exp(- 2 r d^{n})
\end{align*}
The result then follows from Lemma~\ref{vol pullback}.
\end{proof}

We can now show that the functions $\phi_n$ satisfy the conditions of Lemma~\ref{generality}.

\begin{lem}
   Assuming {\rm (\ref{EQN:DYNAMICAL_VOLUME_ESTIMATE})} holds for some $1 \leq \tau \leq d$, the sequence $\phi_n$ is $L^2$-bounded.
\end{lem}

\begin{proof}
   We have:
$$
  \| \phi_n \|^2 \leq \sum_{\ell=0}^\infty (\ell+1)^2 \vol \{|\phi_n|\geq \ell\}.
$$
By Lemma~\ref{tales}, this sum  is bounded by
\begin{align*}
  &\sum_{\ell=0}^{M+1}(\ell+1)^2  +   K \sum_{\ell>M} (\ell+1)^2 \exp\left(- 2 \ell \left(\frac{d}{\tau}\right)^n\right) \\
&\leq K_0 +  K \sum_{\ell=0}^\infty (\ell+1)^2   \exp (-2 \ell)<\infty.
\end{align*}
\end{proof}

\begin{lem}
  The sequence $\phi_n$ exponentially converges to 0  almost everywhere.
\end{lem}

\begin{proof}
  Fix any $\la\in (1,d/\tau)$.
For sufficiently large $n$, we have $\lambda^{-n} > M d^{-n}$, hence
Lemma~\ref{tales}  gives
\begin{align*}
  \vol \{|\phi_n|>\la^{-n}\} \leq K \exp \left(- 2 \left(\frac{d}{\tau \la}\right)^n\right).
\end{align*}
Since the sum of these volumes converges,
the Borel-Cantelli Lemma gives that for a.e. $x\in \CP^2$, we eventually have $|\phi_n(x)|\leq \la^{-n}$.
\end{proof}

\noindent
This completes the proof of Theorem~\ref{Herm norm} and hence of the Equidistribution Theorem for the DHL.
\end{proof}


\section{Proof of the Equidistribution Theorem} \label{SEC:EQUIDISTRIBUTION}

This section is devoted to proving:

\begin{prop}[\bf Finding a good iterate]\label{PROP:GOOD_ITERATE}
Suppose that $f: \CP^2 \rightarrow \CP^2$ is a dominant algebraically stable
rational map with $\deg f = d$ that satisfies the hypotheses the
Equidistribution Theorem {\rm (}see p. \pageref{THM:EQUIDISTRIBUTION} \hspace{-0.08in} {\rm )}.  For any forward invariant neighborhood $\Omega$ of
exceptional set $\mathcal{E}$ there exists an iterate
$n_0$ and constants $K > 0$ and $0~<~\tau~<~d^{n_0}$ such that
\begin{align}\label{EQN:DESIRED_ITERATE_ESTIMATE}
\vol\left(f^{-n_0} (Y) \right) \leq K \left(\vol Y\right)^{1/\tau}.
\end{align}
for any
measurable $Y \subset \CP^2 \setminus \Omega$.
\end{prop}

Once this proposition is proved, the remainder of the proof of the Equidistribution Theorem follows in exactly the same
way as the proof of the Equidistribution Theorem for the DHL (\S \ref{SUBSEC:DHL_PROOF}).

Throughout this section, we will supposed that $f: \CP^2 \rightarrow \CP^2$ is a dominant algebraically stable
rational map with $\deg f = d$.  However, we will keep track of which specific Hypotheses (i)-(iii) of the Equidistribution Theorem are used.

\subsection{No curves of maximal degeneracy.}\label{SUBSEC:NO_CURVES_MAX_DEG}

\begin{prop}\label{PROP:NO_CURVES_VERY_HIGH_SIGMA}
Suppose $f$ satisfies Hypothesis {\rm (ii)}.  Then $D_{>d}$ is a finite set.
\end{prop}

\begin{proof}
By Corollary \ref{LEM:FINITELY_MANY_SIGMA_VALUES}, $D_{>d}$ is algebraic, so if it were an infinite set it
would contain some irreducible algebraic curve $C$.  Then Proposition \ref{PROP:CURVES_OF_VERY_HIGH_DEG_COLLAPSED}
implies that $C$ is collapsed by $f$, which results in $C$ containing an indeterminate point for $f$, by Lemma \ref{LEM:COLLAPSED_CURVES_MEET_IF}.
This violates Hypothesis (ii).
\end{proof}

\begin{prop}\label{PROP:NO_CURVES_MAX_DEGEN}
Suppose $f$ satisfies Hypotheses {\rm (i)} and {\rm (ii)} of the Equidistribution Theorem.  Then there
exists $N > 0$ and a finite forward invariant ``non-escaping set'' $\BADSETCORE \subset \BADSET$
such that if $z \in \BADSET \setminus \BADSETCORE$ then there exists $0 < n <
N$ such that $f^n(z)~\not \in~\BADSET$.
\end{prop}

The proof of Proposition \ref{PROP:NO_CURVES_MAX_DEGEN} relies on several basic lemmas
about the structure of one iterate of a rational map that are presented in 
Appendix \ref{APPENDIX:STRUCTURE_RAT_MAP}.
We will also need the following definition.
Suppose $p \in \CP^2 \setminus I(f)$ is not on a collapsed curve.  Then $f$
induces a germ of an open mapping at $p$ and the {\em local topological degree}
$e(p,f)$ is the topological degree of that germ.  It satisfies the chain rule
\begin{align}\label{EQN_LOCAL_TOP_DEG_MULT}
e(p,f^2) = e(p,f) \cdot e(f(p),f),
\end{align}
so long as $e(p,f)$ and $e(f(p),f)$ are defined.

The key step in the proof of Proposition \ref{PROP:NO_CURVES_MAX_DEGEN} is:
\begin{lem}\label{LEM:MAXIMAL_INVARIANT_WHITNEYFOLD_CONTAINS_INDET_PT}
Suppose that $f$ satisfies Hypothesis {\rm (i)}.  Then there is no irreducible algebraic curve $C$ such that for some $k \geq 1$
we have $f^k(C) \subset C$, $C \cap I(f^k) = \emptyset$, and $\sigma(z,f^k) \geq d^k$ for all $z \in C$.
\end{lem}

\begin{proof}
Suppose for contradiction that such a curve $C$ exists.  Then Lemma
\ref{LEM:COLLAPSED_CURVES_MEET_IF} implies that $C$ is not collapsed by $f^k$.
Lemma \ref{LEM:NORMALFORM} then gives that $C$ is a Whitney Fold and that there
is a finite $S \subset C$ such that for every $p \in C \setminus S$ we have
\begin{align}\label{EQN:LOCAL_TOP_DEG_C_LEM}
e(p,f^k) = \sigma(p,f^k) \geq d^k.
\end{align}
Let us now consider the irreducible component $C$ as a divisor $(C)$, assigning it multiplicity one.
(See \cite[Appendix A]{BLR1} for basic background on divisors.)
Since $C$ is disjoint from $I(f^k)$, the pushfoward of $(C)$ by $f^k$ is defined by
\begin{align*}
(f^k)_* (C) := {\rm deg}_{\rm top}(f^k |_C: C \rightarrow f^k(C)) (f^k(C)) = {\rm deg}_{\rm top}(f^k |_C: C \rightarrow C) (C).
\end{align*}
Meanwhile, we have that
\begin{align*}
{\rm deg}((f^k)_* (C)) = d^k {\rm deg}((C));
\end{align*}
see \cite[Lem. A.5]{BLR1}.  Therefore, we conclude that
\begin{align}\label{EQN:DEGTOP_RESTRICTED_TO_C}
{\rm deg}_{\rm top}(f^k |_C: C \rightarrow C) = d^k.
\end{align}
Equation (\ref{EQN:DEGTOP_RESTRICTED_TO_C}) implies that a generic $z \in C$
will have $d^k$ preimages under $f^k |_C$ and Equation
(\ref{EQN:LOCAL_TOP_DEG_C_LEM}) then implies that a generic point $z'$ near $z$,
but not on $C$, will have $d^{2k}$ preimages under $f^k$.  (There are
$d^k$ preimages of $z'$ near each preimage of $z$ under $f^k |_C$.)  By Hypothesis (i)
$I(f) \neq \emptyset$, so this violates the fact that ${\rm deg}_{\rm top}(f) <
d^{2k}$; see Lemma~\ref{LEM:DTOP}.
\end{proof}

\begin{proof}[Proof of Proposition \ref{PROP:NO_CURVES_MAX_DEGEN}]
By Lemma \ref{LEM:FINITELY_MANY_SIGMA_VALUES}, the set $\BADSET$ is algebraic.
Moreover, none of the irreducible components of $\BADSET$ is collapsed by~$f$, using
Hypothesis~(ii) that $\BADSET$ is disjoint from $I(f)$ and Lemma
\ref{LEM:COLLAPSED_CURVES_MEET_IF}.

It suffices to show that for each irreducible component $C \subset \BADSET$ there is
an iterate $n$ such that $f^{n}(C) \not \subset \BADSET$.  In this case,
the Bezout Theorem implies $f^{n}(C) \cap \BADSET$ is finite.  This implies that
all but finitely many points of $C$ are mapped out of $\BADSET$ by $f^n$, since
$f$ does not collapse any irreducible component of $\BADSET$.

Suppose for the purpose of obtaining a contradiction that $C \subset \BADSET$
is some irreducible component with $f^n(C) \subset \BADSET$ for every $n \geq 1$.  As
$\BADSET$ contains finitely many irreducible components,
we conclude that some iterate $f^{n_0}(C)$ is periodic under $f$.
Therefore, without loss of generality, we can suppose that $C$ itself is
periodic under $f$ with some period $k \geq 1$.

Let $C_n := f^n(C)$ for each $0 \leq n < k$.  Lemma \ref{LEM:NORMALFORM} gives 
that there is a finite set $S_n \subset C_n$ and an exponent
$\FOLDEXPONENT_n$ so that for each $p \in C_n \setminus S_n$ we have
\begin{align*}
\sigma(p,f) = e(p,f) = \FOLDEXPONENT_n \geq d,
\end{align*}
with the last inequality coming from the fact that $C_n \subset \BADSET$.
Meanwhile, applying Lemma~\ref{LEM:NORMALFORM} to $f^k$ and $C$, together with
the chain rule 
(\ref{EQN_LOCAL_TOP_DEG_MULT}), we see that there is a finite $S \subset C$ so
that for all $p \in C \setminus S$ we have
\begin{align}\label{EQN:LOCAL_TOP_DEG_C}
\sigma(p,f^k) = e(p,f^k) = e(p,f) \cdot e(f(p),f) \cdots e(f^{n-1}(p),f) \geq d^k.
\end{align}
Using Hypothesis (i), we can now apply Lemma \ref{LEM:MAXIMAL_INVARIANT_WHITNEYFOLD_CONTAINS_INDET_PT} to conclude that $C$ contains a point of $I(f^k)$.
However, this contradicts that for each $0 \leq n < k$ the curve $C_n \subset \BADSET$ and is therefore disjoint from $I(f)$ by Hypothesis (ii).
\end{proof}

\subsection{Superattracting Periodic Points}\label{SUBSEC:SUPERATTRACTING} 

An important aspect of the work of Favre and Jonsson \cite{FAVRE_JONSSON} is to
consider asymptotic versions of the multiplicities $\mu$ and~$c$ (defined in \S \ref{SUBSEC:VANISHIN_ORDER} and \S \ref{SUBSEC:EQUIDIST},
respectively) along the orbit of any $z
\in \CP^2$.   In our paper, $I(f)$ is not assumed to be empty,
so that a point for which an iterate lands on $I(f)$ will have uncountably many
different forward orbits, causing such an asymptotic multiplicity not to exist.
However, we can apply the results of \cite{FAVRE_JONSSON} at regular periodic
points, as we will now summarize.

Let $g: (\C^2,{\bm 0}) \rightarrow (\C^2,{\bm 0})$ be a dominant holomorphic germ.  
Then $\hat{\mu}(z,g^n) := 3+2\mu(z,g^n)$ is submultiplicative:
\begin{align*}
\hat{\mu}(z,g^{n+m}) \leq \hat{\mu}(z,g^n) \hat{\mu}(g^n(z),g^m).
\end{align*}
In particular, 
\begin{align*}
\mu_\infty(z,g) := \lim \hat{\mu}(z,g^n)^{1/n}
\end{align*}
exists. 

For any $n \in \mathbb{N}$, let $c(z,g^n)$ be the order of the lowest order term
in the power series for $g^n$ centered at the origin.
It follows directly from the definition that  $c(z,g^{n+m}) \geq c(z,g^n)
c(g^n(z),g^m)$.  Moreover, Favre and Jonsson proved that  $c(z,g^n) \leq \frac{1}{2}
\mu(z,g^n)+1$.  Therefore, the limit 
\begin{align*}
c_\infty(z,g):=\lim c(z,g^n)^{1/n}
\end{align*}
exists and satisfies
\begin{align}\label{EQN:RELATION_CINFTY_MUINFTY}
c_\infty(z,g) \leq \mu_\infty(z,g).
\end{align}

Suppose $z$ is a regular periodic point of period $k$ for the dominant rational
map $f: \CP^2 \rightarrow \CP^2$.  Then, in suitable local coordinates, $f^k$
defines a dominant holomorphic germ $f^k: (\C^2,{\bm 0}) \rightarrow (\C^2,{\bm 0})$.  By
the discussion of the previous two paragraphs, we can define 
\begin{align*}
\mu_\infty(z,f) := \lim_{n \rightarrow \infty} \hat{\mu}(z,f^{nk})^{1/nk} \qquad \mbox{and} \qquad c_\infty(z,f) := \lim_{n \rightarrow \infty} c(z,f^{nk})^{1/nk}.
\end{align*}
The total degree of the Jacobian divisor $\Jac f^k$ on $\CP^2$ is $3(d^k-1)$, so $\Jac f^k$
cannot vanish to more than this order at ${\bm 0}$; see Remark \ref{REM:DEG_JAC}..  This implies that
\begin{align}\label{EQN:MUINFTY_BOUND_BY_D}
c_\infty(z,f) \leq \mu_\infty(z,f) \leq d = \deg(f).
\end{align}
In particular, this justifies the definition of {\em maximally superattracting
regular periodic point} from the Introduction.   

\begin{thm}{\bf (Favre-Jonsson \cite[Thm. 4.2]{FAVRE_JONSSON})}\label{THM:FJ_GERM}
Let $g: (\C^2,{\bm 0}) \rightarrow (\C^2,{\bm 0})$ be a holomorphic germ.  Let $V_1,\ldots,V_k$ be the irreducible components 
of the critical set of $g$.  Assume that $c_\infty({\bm 0},g) < \mu_\infty({\bm 0},g)$.  Then there exists $a_1,\ldots,a_k \geq 0$ (not all zero) such that
\begin{align}\label{GOOD_PULLBACK}
g^*\left(\sum_i a_i [V_i]\right) \geq \mu_\infty({\bm 0},g) \left(\sum_i a_i [V_i]\right).
\end{align}
\end{thm}

\noindent
Inequality (\ref{GOOD_PULLBACK}) means that if one subtracts the
current on the right from the current on the left, then the result 
is a weighted sum of currents of
integration over finitely many analytic curves, each assigned a non-negative
weight.

\begin{cor} \label{COR:BOUND_RELATING_MULTIPLICITIES}Suppose that $f: \CP^2 \rightarrow \CP^2$ is a dominant algebraically stable rational map
of degree $d$.  If $z$ is a regular periodic point of period $k$ for $f$ such that 
\begin{align}
c_\infty(z,f) < \mu_\infty(z,f) = d,
\end{align}  
then $f^k$ has a backward invariant curve $C$ through $z$.  Each of the irreducible components of~$C$ passes
through $z$.
\end{cor}

\begin{proof}
Taking the $k$-th iterate, we can suppose $z$ is a regular fixed point for~$f$.
Theorem~\ref{THM:FJ_GERM} implies that there exist weights $a_1,\ldots,a_k \geq
0$ (not all zero) such that (\ref{GOOD_PULLBACK}) holds with $\mu_\infty({\bm 0},f) =
d$.  Suppose two irreducible branches $V_i$ and $V_j$ of the critical locus of
the germ $f:(\C^2,{\bm 0}) \rightarrow (\C^2,{\bm 0})$ are obtained as the restriction of the same algebraic curve from
the critical locus of $f: \CP^2 \rightarrow \CP^2$.  In this case, it is straightforward to check the
proof of Theorem \ref{THM:FJ_GERM} from \cite{FAVRE_JONSSON} the weights are equal: $a_i =
a_j$.  Therefore, the local current $\sum_i a_i [V_i]$ extends to a global
closed positive $(1,1)$ current
\begin{align*}
T = \sum \alpha_k [C_k]
\end{align*}
on $\CP^2$ which also satisfies $f^* T \geq d \ T$.  Since the pullback under
$f$ multiplies the total mass of a closed positive $(1,1)$ current on $\CP^2$
by exactly $d$, we conclude that $f^* T = d \ T$ (see Appendix
\ref{APPENDIX:CURRENTS}).  In particular, the support of $T$ is a backward 
invariant curve for $f$ passing through~$z$.
\end{proof}


\subsection{Backward invariant curves}\label{SUBSEC:BACKWARD_INV_CURVES}
Because of Corollary \ref{COR:BOUND_RELATING_MULTIPLICITIES} we need a better understanding of algebraic curves
that are backward invariant under some iterate of $f$.  This subsection is devoted to proving the following:

\begin{prop}\label{PROP:CHARACTERIZATION_BACKWARD_INVARIANT_CURVES}
Suppose that $f$ satisfies Hypotheses {\rm (i)}-{\rm (iii)} of the Equidistribution Theorem and
that $z_0$ is a regular periodic point of period $k$ for $f$.  If $C$ is a {\rm
(}possibly reducible{\rm )} algebraic curve that is backward invariant under
$f^k$, each of whose irreducible components  passes through $z_0$, then $z_0 \in \mathcal{E}(b)$ {\rm (}see p. \pageref{DEF_EXCEPTIONAL_SET} for the
definition of $\mathcal{E}(b)${\rm )}.
\end{prop}

\begin{proof}
We will show that $z_0$ is superattracting and that there is some algebraic curve $C' \subset C$ (possibly also reducible) satisfying
the conditions necessary for $z_0$ to be in $\mathcal{E}(b)$.

Let $C_1,\ldots,C_m$ denote the irreducible components of $C$ and suppose for
contradiction that none of them is collapsed under $f^k$.  In this case, we
claim that $f^{-k}$ induces a permutation on $\{C_1,\ldots,C_m\}$.
Consider an arbitrary $0 \leq i \leq m$ and notice that $f^{-k}(C_i)$ is not reduced to an indeterminate point.  Indeed,
 by considering $C_i$ as a divisor $(C_i)$ of multiplicity one, the fact
that $(f^k)^* (C_i)$ is a divisor of degree $d^k \deg(C_i) > 0$ (see \cite[Lem. A.5]{BLR1}) implies that 
$f^{-k}(C_i)$ is a non-trivial algebraic curve.  Since 
$C$ is backward invariant under $f^k$, this implies that there is at least one $1 \leq j \leq m$
such that $C_j \subset f^{-k}(C_i)$.
Meanwhile, for each $1 \leq n \leq m$ the irreducible component $C_n$ is not
collapsed by $f^k$ so that it can occur as a component of $f^{-k}(C_p)$ for at
most one value of $1 \leq p \leq m$.  We conclude that, for each $1 \leq i \leq m$
there exists a unique $j$ such that $f^{-k}(C_i) = C_j$.

In particular, there exists $\ell \geq 1$ such that $B:=C_1$ is backward
invariant under $f^{\ell k}$.  If we consider $B$ as a divisor $(B)$ of
multiplicity one, then this implies that
\begin{align}\label{EQN:PULLBACK_MULT_BY_DLK}
(f^{\ell k})^*(B) = d^{\ell k} (B)
\end{align}
(see again \cite[Lem. A.5]{BLR1}).

Backwards invariance under $f^{\ell k}$ also implies $f^{\ell k}(B \setminus
I(f^{\ell k})) \subset B$.  Moreover, since we are supposing that none of the
irreducible components of $C$ is collapsed by $f^k$, we can use~(\ref{EQN:PULLBACK_MULT_BY_DLK}) and the Whitney Fold normal form given by
Lemma~\ref{LEM:NORMALFORM} to find that $\sigma(z,f^{\ell k}) \geq d^{\ell k}$ for
all $z \in B$.  Using Hypothesis (i), it follows from Lemma
\ref{LEM:MAXIMAL_INVARIANT_WHITNEYFOLD_CONTAINS_INDET_PT} that that $B$
contains a point of~$I(f^{\ell k})$.

For each $0 \leq j < \ell k$ we let
$B_j :=\overline{f^j(B \setminus I(f^j))}$, which is a non-trivial irreducible
algebraic curve that is not collapsed by $f$.  Proposition
\ref{PROP:CURVES_OF_VERY_HIGH_DEG_COLLAPSED} gives that at most finitely many
points of $B_j$ are in $D_{>d}$.  Thus, in order to have $\sigma(z,f^{\ell k})
= d^{\ell k}$ for generic points of $B$, we must have that $B_j \subset
D_{\geq d}$ for each $1 \leq j \leq \ell k$.  Hypothesis (ii) implies that
$B_j$ is disjoint from $I(f)$ for each $0 \leq j < \ell k$, contradicting
that $B$ contains a point of~$I(f^{\ell k})$. 

We can therefore let $C' \subset C$ denote the (non-trivial) union of all
irreducible components of $C$ that are collapsed by some iterate of $f^k$.
Since each of these components passes through the regular periodic point $z_0$,
each of them collapses to $z_0$.  Since $C$ is backward invariant under $f^k$,
$C'$ is also backward invariant under $f^k$.

Suppose for contradiction that $z_0$ is a smooth point of $C'$.  Since every
irreducible component of $C$ passes through $z_0$, this implies that $C'$ is
itself irreducible.  In particular, backward invariance of $C'$ under $f^k$
implies that
\begin{align}\label{EQN:PULLBACK_MULT_BY_DK}
(f^{k})^*(C') = d^{k} (C').
\end{align}
 Then we can
choose local coordinates $(x,y)$ centered at~$z_0$ such that $C' = \{x=0\}$.
Write $f^k$ in these local coordinates as 
\begin{align}\label{EQN:F_LOCAL_COORDS}
f^k(x,y) = (f^k_1(x,y),f^k_2(x,y)).
\end{align}
Equation (\ref{EQN:PULLBACK_MULT_BY_DK}) implies that $f^k_1(x,y) = x^{d^k} g(x,y)$ for
some non-vanishing holomorphic function $g$.  From this, it is immediate that
$f^k$ contracts the volume of a small bidisc centered along $C'$ with exponent
$d^k$, i.e. that for ever $z \in C'$ we have $\sigma(z,f^k) \geq d^k$.

Let $z_0,\ldots,z_{k-1}$ denote the periodic orbit of $z_0$.  Moreover, let $1
\leq j_0 \leq k$ be the smallest iterate for which $C'$ is collapsed by $f^j$.
For each $0 \leq j < j_0$ we let 
\begin{align*}
C'_j:=\overline{f^j(C' \setminus I(f^j))}.
\end{align*}
Then $C'_{j_0-1}$ is collapsed by $f$ and $z_{j_0} \in \BISET$, by definition.
By Hypothesis (iii), we have $\sigma(z_j,f) \leq d$ for each $0 \leq j \leq
k-1$.  Meanwhile, for each $0 \leq j \leq j_0-1$ the curve $C'_j \not \subset
D_{>d}$, by Proposition \ref{PROP:NO_CURVES_VERY_HIGH_SIGMA}.  Therefore, in
order to have $\sigma(z,f^k) \geq d^k$ for every $z \in C'$ we must have
$C'_{j_0-1} \subset D_{\geq d}$.  This contradicts Hypothesis~(ii) since
$C'_{j_0-1}$ contains a point of $I(f)$, by Lemma
\ref{LEM:COLLAPSED_CURVES_MEET_IF}.

Therefore, $z_0$ is a singular point of $C'$.  We now use this to show that $z_0$ is superattracting.
Let $(x,y)$ be local holomorphic coordinates centered at $z_0$ and suppose that $C'$ is given in these coordinates as $C' = \{q(x,y) = 0\}$
for some holomorphic function~$q$
(more specifically, we choose $q$ so that it defines the divisor $(C')$ with multiplicity one).  Since $z_0$ is a singular point of $C'$ we have 
\begin{align}\label{EQN:P_PARTIALS}
q(0,0) = \frac{\partial q}{\partial x} (0,0) = \frac{\partial q}{\partial y}(0,0) = 0.
\end{align}
Let $\ell \geq 1$ be chosen so that $f^{\ell k}(C' \setminus I(f^{\ell k})) = z_0$.  
Writing $f^{\ell k}(x,y) = (f^{\ell k}_1(x,y),f^{\ell k}_2(x,y))$ in these local coordinates we have that 
\begin{align}\label{EQN:F_FACTORS}
f^{\ell k}_1(x,y) = q(x,y) g_1(x,y) \qquad \mbox{and} \qquad  f^{\ell k}_2(x,y) = q(x,y) g_2(x,y),
\end{align}
for some holomorphic functions $g_1(x,y)$ and $g_2(x,y)$.  Equations (\ref{EQN:P_PARTIALS}) and (\ref{EQN:F_FACTORS}) imply that
$Df^{\ell k}(0,0) = {\bm 0}$, i.e.\ that $z_0$ is a superattracting periodic point.
We conclude that $z_0 \in \mathcal{E}(b)$.
\end{proof}

\subsection{Behavior of the volume exponent $\sigma$ under iteration.}

The exponents $\sigma(z,f)$ do not transform very well under iteration: if an iterate
of $z$ lands on $I(f)$, this leads to many different orbits of $z$,
all of which we need to control.   
Let $Y \subset \CP^2$.  
Given a sequence of open sets $N_0,\ldots,N_{n-1} \subset \CP^2$ let
\begin{align*}
f^{-n}_{N_0,N_1,\ldots,N_{n-1}}(Y) =  \{& z_0 \in N_0 \, : \, \mbox{there exists orbit} \, \bm z = (z_0,z_1,\ldots,z_{n}) \\  
&\mbox{with} \, z_i \in N_i \,  \, \mbox{for $1 \leq i \leq n-1$ and $z_n \in Y$}\}.
\end{align*}
If $\bm z = (z_0,z_1,\ldots,z_{n-1})$ is an orbit of $f$ we let $\sigma(\bm z,f^n) \equiv \sigma(z_0,z_1,\ldots,z_{n-1},f^n)$ be
the smallest positive number such that for any $\gamma > \sigma(\bm z,f^n)$ there is an $K > 0$ and neighborhoods $N_0,\ldots,N_{n-1}$ of $z_0,\ldots,z_{n-1}$ such that 
\begin{align}\label{EQN:DEF_ITERATED_SIGMA}
\vol(f^{-n}_{N_0,\cdots,N_{n-1}} Y) \leq K (\vol Y)^{1/\gamma}
\end{align}
for any measurable $Y \subset \CP^2$.

It is clear from the definition that 
for any orbit $z_0,\ldots,z_{n-1}$ and any  $1 \leq k \leq n-1$
\begin{align*}
\sigma(z_0,z_1,\ldots,z_{n-1},f^n) &\leq \sigma(z_0,\ldots,z_{k-1},f^k) \sigma(z_k,\ldots,z_{n-1},f^{n-k}) \quad \mbox{and}\\
\sigma(z_0,z_1,\ldots,z_{n-1},f^n) &\leq \sigma(z_0,f^n),
\end{align*}
where the exponent on the right hand side of the second inequality is from (\ref{EQN:DEF_SIGMA}).

\begin{lem}\label{LEM:VOL_ESTIMATES_UNDER_COMPOSITION}
For any $z_0 \in \CP^2$ we have
\begin{align}
\sigma(z_0,f^n) \leq \sup_{\bm z } \sigma(\bm z,f^n),
\end{align}
where the supremum is taken over all orbits $\bm z = (z_0,\ldots,z_{n-1})$.
\end{lem}

\begin{proof}

Since the indeterminate points of $f^i$ for each $1 \leq i \leq n$ are
isolated, we can find a neighborhood $V$ of $z_0$ containing no indeterminate
points for each $f^i$ other than (potentially)~$z_0$.  We do a finite
sequence of blow-ups over $z_0$ forming $\pi: \widetilde{V} \rightarrow \CP^2$
so that for each $1 \leq i \leq n$ the iterate $f^i$ lifts to a regular map
$\widetilde{f^i}$ making the diagram
\begin{eqnarray*}
\xymatrix{
\widetilde{V} \ar[d]^\pi \ar[dr]^{\widetilde{f^i}} & \\
V \ar[r]^{f^i} & \CP^2,
}
\end{eqnarray*}
commute wherever $f^i \circ \pi$ is defined.  

Let $\gamma$ be any exponent greater than $\sup_{\bm z } \sigma(\bm z,f^n)$.
Any point $\tilde{z} \in \pi^{-1}(z)$ determines the following orbit of length
$n$ of $z$:
\begin{align}\label{EQN:ORBIT_OF_Z}
\bm z \equiv (z_0,\ldots,z_{n-1}) := \left(z,\widetilde{f^1}(\tilde{z}),\ldots,\widetilde{f^{n-1}}(\tilde{z})\right).
\end{align}
By choice of $\gamma$, there is a is a sequence of neighborhoods $U_0(\bm z),
\ldots, U_{n-1}(\bm z)$ of $z_0,\ldots,z_{n-1}$, respectively, so that
(\ref{EQN:DEF_ITERATED_SIGMA}) holds.  Associated to
this sequence is a neighborhood $\widetilde{U}(\bm z)$ of $\tilde{z}$ such that
for any $Y \subset \CP^2$ we have
\begin{align}\label{GOODEQN}
f^{-n}_{U_0(\bm z), \ldots, U_{n-1}(\bm z)} Y = \pi\left(\left(\widetilde{f^n}\right)^{-1} Y \cap \widetilde{U}(\bm z)\right).
\end{align}
Since $\pi^{-1}\{z_0\}$ is compact, there exist finitely many length $n$ orbits $\bm z^1,\ldots, \bm z^\ell$ of $z_0$ 
whose neighborhoods $\widetilde{U}(\bm z^1),\ldots,\widetilde{U}(\bm z^1)$ cover $\pi^{-1}(z_0)$. 
If we let $U = \bigcup_{i=1}^\ell \pi\left(\tilde{U}(\bm z^i)\right)$ then 
(\ref{GOODEQN}) implies that for any $Y \subset \CP^2$ we have
\begin{align*}
f^{-n}(Y) \cap U  \subset  \bigcup_{i=1}^\ell f^{-n}_{U_0(\bm z^i), \ldots, U_{n-1}(\bm z^i)} Y.
\end{align*}
The result follows since each $f^{-n}_{U_0(\bm z^i), \ldots, U_{n-1}(\bm z^i)} Y$ satisfies (\ref{EQN:DEF_ITERATED_SIGMA}) with some suitable
multiplicative constant.
\end{proof}

We will also use a modified notation $\sigma(z,f,W)$ for the volume exponent in the case
that we require that (\ref{EQN:DEF_SIGMA}) only holds for measurable $Y \subset W$.
A straightforward adaptation of the proof of Lemma \ref{LEM:VOL_ESTIMATES_UNDER_COMPOSITION} yields:

\begin{lem}\label{LEM:VOL_ESTIMATES_UNDER_COMPOSITION_RESTRICTED_CODOMAIN}
Let $\Omega \subset \CP^2$ be a forward invariant open set.
For any $z_0 \in \CP^2 \setminus \Omega$ we have
\begin{align}
\sigma(z_0,f^n,\CP^2 \setminus \Omega) \leq \sup_{\bm z } \sigma(\bm z,f^n),
\end{align}
where the supremum is taken over all orbits $\bm z = (z_0,\ldots,z_{n-1}) \subset \CP^2 \setminus \Omega$.
\end{lem}

\subsection{Proof of Proposition~\ref{PROP:GOOD_ITERATE}.}
Let $\Omega$ be any forward invariant neighborhood of the exceptional set $\mathcal{E}$.
We will prove that there exist $B > 0$ and $1 < \alpha < d$ such that
for any $n \geq 0$ and any finite orbit $z_0,\ldots,z_{n-1}$ remaining in $\CP^2 \setminus \Omega$ we have
\begin{align}\label{EQN:GROWTH_CONDITION}
\sigma(z_0,\ldots,z_{n-1},f^n) \leq  B \alpha^n.
\end{align}
We can then let $n_0$ be sufficiently large that $\tau:= B \alpha^{n_0} < d^{n_0}$.
Then Lemma~\ref{LEM:VOL_ESTIMATES_UNDER_COMPOSITION_RESTRICTED_CODOMAIN} will imply that
$\sigma(z_0,f^{n_0},\CP^2 \setminus \Omega) < \tau$ for any $z_0 \in \CP^2
\setminus \Omega$.  Since $\CP^2 \setminus \Omega$ is compact, it is then covered by finitely many neighborhoods such that on each neighborhood $N$ we have
\begin{align*}
\vol(f^{-n_0} Y \cap N) \leq K_N (\vol Y)^{1/\tau},
\end{align*}
which will therefore prove the proposition.

\msk
To prove (\ref{EQN:GROWTH_CONDITION}), we consider three types of finite orbits. 

\msk
\noindent
{\bf Type I:} Periodic orbits that pass through $\VERYBADSET \setminus \mathcal{E}$.

\msk
\noindent
Let $p$ be any point from such an orbit.   We could\footnote{For example, $p$
could be a fixed point in $\VERYBADSET$.} have $\sigma(f^n(p),f) \geq d$ for
every $n \geq 0$, but the hypothesis that $p \not \in \mathcal{E}$ and the
results of \S \ref{SUBSEC:SUPERATTRACTING} and \S
\ref{SUBSEC:BACKWARD_INV_CURVES} will imply that $\sigma(p,f^n)$
decays sufficiently as we iterate and indeed grows at exponential rate slower than~$d$.

By Hypothesis (iii), $p$ is a regular
periodic point, i.e.\ the orbit of $p$ is disjoint from $I(f)$.  Since $p \not
\in \mathcal{E}(a)$ it is not maximally superattracting, $c_\infty(p,f) < d$.
Therefore, since $p \not \in \mathcal{E}(b)$, Corollary
\ref{COR:BOUND_RELATING_MULTIPLICITIES} and Proposition
\ref{PROP:CHARACTERIZATION_BACKWARD_INVARIANT_CURVES} imply that
$\mu_\infty(p,f) < d$.  Using the volume estimate from Lemma
\ref{LEM:FJ_ESTIMATE}, this implies that $\lim \sup (\sigma(p,f^n))^{1/n} < d$.
Hence, (\ref{EQN:GROWTH_CONDITION}) holds for such an orbit.

\msk
\noindent
{\bf Type II:} Periodic orbits contained in $\BADSETCORE \setminus \mathcal{E}$.

\msk
\noindent
Recall that $\BADSETCORE \subset \BADSET$ is finite and forward invariant by
Proposition \ref{PROP:NO_CURVES_MAX_DEGEN}.  It is disjoint from $I(f)$ by
Hypothesis (ii).  Inequality (\ref{EQN:GROWTH_CONDITION}) holds for these orbits
using exactly the same reasoning as for orbits of Type I.

\msk
\noindent
{\bf Type III:} Finite orbits ${\bm z} = (z_0,\ldots,z_{n-1}) \subset \CP^2
\setminus \Omega$ that are disjoint from orbits of Type~I and Type II.

\msk
\noindent
Let $k$ denote the number of elements of $\BADSETCORE$, which is finite and
forward invariant by Proposition \ref{PROP:NO_CURVES_MAX_DEGEN}.  Since
$\BADSETCORE$ is forward invariant and orbit ${\bm z}$ is disjoint from
periodic orbits of Type~II, we have  $z_0,\ldots,z_{n-k-1} \not \in
\BADSETCORE$.   Moreover, Proposition \ref{PROP:NO_CURVES_MAX_DEGEN} gives $N
\in \mathbb{N}$ such that at least one out of every $N$ points of
$z_0,\ldots,z_{n-k-1}$ is ``good'', landing in $\CP^2 \setminus \BADSET$.
Moreover, Lemma \ref{LEM:FINITELY_MANY_SIGMA_VALUES} gives a uniform $1 \leq
\sigma_0 < d$ such that $\sigma(p,f) \leq \sigma_0$ for these ``good'' points.

Let $\ell$ denote the number of elements of $\VERYBADSET$, which is finite by
Proposition \ref{PROP:NO_CURVES_VERY_HIGH_SIGMA}.  Since orbit ${\bm z}$ is
disjoint from periodic orbits of Type I, it can only meet $\VERYBADSET$ at most
$\ell$ times.  In other words, $\sigma(z_i,f) > d$ for at most $\ell$ values of
$0 \leq i \leq n-1$.  If we absorb the excess $\sigma(z_i,f) / d > 1$ from
these at most $k$ points into the multiplicative constant $B$, then the
previous paragraph implies that (\ref{EQN:GROWTH_CONDITION}) holds for an orbit
of Type III.

\msk
Let is now consider an arbitrary finite orbit ${\bm z} = (z_0,\ldots,z_{n-1})
\subset \CP^2 \setminus \Omega$.  If it is not of Type I, II, or III, then
there is some $0 < m < n-1$ such that $(z_0,\,\ldots,z_{m-1})$ is of Type III and
$(z_{m},\ldots,z_{n-1})$ is of either Type I or II.  Inequality (\ref{EQN:GROWTH_CONDITION}) holds for such an orbit using the submultiplicativity
\begin{align*}
\sigma(z_0,\ldots,z_{n-1},f^n) \leq \sigma(z_0,\ldots,z_{m-1},f^{m}) \sigma(z_m,\ldots,z_{n-1},f^{n-m}).
\end{align*}
This completes the proof of Proposition~\ref{PROP:GOOD_ITERATE} and hence of the Equidistribution Theorem.
\qed

\subsection{A useful proposition for determining $\mathcal{E}$.}
\label{SUBSEC:CHECKING_FOR_EXCEPTIONAL_POINTS}

It is easy to detect collapsed curves by checking each irreducible component of the critical locus, so it is usually
easy to determine $\mathcal{E}(b)$.  The following proposition makes it easier to determine $\mathcal{E}(a)$.

\begin{prop}\label{PROP:CHECKING_FOR_EXCEPTIONAL_POINTS}
Let $f: \CP^2 \rightarrow \CP^2$ satisfy the hypotheses of the Equidistribution Theorem.  Then
any maximally superattracting periodic orbit {\rm (}i.e.\ any point of $\mathcal{E}(a)${\rm )} has orbit passing through
$\BADSET$.  Moreover, it either
\begin{itemize}
\item[(i)] passes through the finite set $\VERYBADSET$, or
\item[(ii)] is contained entirely in the finite non-escaping set $\BADSETCORE \subset \BADSET$.
\end{itemize}
\end{prop}

\begin{proof}
Suppose $z_0$ is a maximally superattracting periodic point so that $d = c_\infty(z_0,f)$, by definition.  
Iterating a round ball of small radius centered at $z_0$, this implies that $\lim \inf (\sigma(p,f^n))^{1/n} = d$.  
Using the submultiplicativity of the volume exponent $\sigma$, the only way this can happen is if (i) or (ii) holds.
\end{proof}


\section{Applications of the Equidistribution Theorem to \\ other hierarchical lattices} \label{SEC:APPLICATIONS}

We now consider the Migdal-Kadanoff renormalization mappings associated to the five hierarchical lattices 
whose generating graphs are shown in Figure \ref{FIG:GEN_GRAPHS} (in the introduction).

\subsection{Linear Chain}
The hierarchical lattice generated by the double edge shown in Figure~\ref{FIG:GEN_GRAPHS} is the same as the classical $\mathbb{Z}^1$ lattice.
The corresponding renormalization mapping is
\begin{align*}
R[U:V:W] = [U^2+V^2:V(U+W):W^2+V^2].
\end{align*}
Its critical locus consists of the curves $C_1:=\{U+W=0\}$ and
$C_2:=\{UW-V^2 = 0\}$.  Generic points $z$ on each of these curves have volume
exponent $\sigma(z,R) = 2 = \deg(R)$.  Since these two curves meet at the
indeterminate points $a_{\pm} := [\pm i:1:\mp i]$, this implies that
$\sigma(a_\pm,R) = \deg(R)$, violating Hypothesis (ii) of the Equidistribution
Theorem.  

Despite this, the Global LYF Theorem holds for this hierarchical lattice.  Let
$\mathcal{F}$ denote the family of lines passing through the fixed point
$[1:0:1]$.  Each line $L \in \mathcal{F}$ is invariant under $R$ and all but
finitely many lines from $\mathcal{F}$ intersect $C_2$ at two distinct points,
neither of which is the intersection of the Principal LYF Locus $S_0 =
\{U+2V+W=0\}$ with $C_2$.  For such generic lines $L$ we have $R|_L(z) = z^2$
in a suitable choice of local coordinate $z$.  The exceptional points $z=0$ and
$z= \infty$ correspond to the intersections of $L$ with the conic $C_2$ and
$S_0 \cap L$ corresponds to a non-exceptional point.  Therefore, for all but
finitely many $L \in \mathcal{F}$ we have that the (normalized) iterated
preimages of $S_0 \cap L$ equidistribute to the measure of maximal entropy for
$R|_L$, see \cite{LYUBICH:NOTE,LYUBICH:MAX ENT,FLM}, which is equal to the
slice $S|_L$, where $S$ is the Green current for~$R$.  This is sufficient to
prove that $\frac{1}{2^n} (R^n)^* [S_0] \rightarrow S$.  Pulling everything
back under $\Psi$ one finds that the Global LYF Theorem holds for this
hierarchical lattice.

\subsection{$k$-fold DHL}
One can consider a generalization of the DHL with the generating graph having valence $k \geq 2$ at the marked vertices $a$ and $b$; see Figure~\ref{FIG:GEN_GRAPHS}.
The renormalization
mapping is 
\begin{align*}
R[U:V:W] = [(U^2+V^2)^k:V^k(U+W)^k:(W^2+V^2)^k],
\end{align*}
whose lift $\hat{R} : \mathbb{C}^3 \rightarrow  \mathbb{C}^3$ has Jacobian
\begin{align*}
   \det D\hat R = 2^2 k^3\, V^{k-1}\, (UW-V^2)\, (U+W)^k\, (U^2+V^2)^{k-1}\, (W^2+V^2)^{k-1}.
\end{align*}
In particular, the critical locus of $R$ is the same 
as for the DHL (see \S \ref{SUBSEC:CRITICAL_LOCUS}).  However,
when considered as a divisor, the curves have higher multiplicities, depending on~$k$.

This mapping has the same indeterminate points as for the DHL; $I(R) =
\{\INDmig_\pm\}$.  Since $I(R) \neq \emptyset$,  Hypothesis (i) is satisfied.
The mapping $R$ also has the same superattracting fixed points
$\{\CFIXmig,\CFIXmig'\}$.  Just as in \S \ref{SEC:DHL_VOLUME_ESTS}, one can
check that for any $z \in \CP^2 \setminus \{\CFIXmig,\CFIXmig'\}$ one has
$\sigma(z,R) < 2k = \deg(R)$.    Therefore, $\sigma(\INDmig_\pm)~<~\deg(R)$,
giving that Hypothesis (ii) is satisfied.  Meanwhile, the only curve collapsed
by $R$ is $\{U+W = 0\}$, which is collapsed to the fixed point $b_0 := [1:0:1]
\not \in \{\CFIXmig,\CFIXmig'\}$.  Since $\VERYBADSET \subset
\{\CFIXmig,\CFIXmig'\}$ it is invariant under $R$ and disjoint from $I(R) \cup
I^{-}(R) = \{a_\pm,b_0\}$, Hypothesis (iii) is also satisfied.

We now determine the exceptional set $\mathcal{E}$.  Since $\BADSET \subseteq
\{\CFIXmig,\CFIXmig'\}$ and since $\{\CFIXmig,\CFIXmig'\}$ is invariant,
Proposition \ref{PROP:CHECKING_FOR_EXCEPTIONAL_POINTS} gives that any point of
$\mathcal{E}(a)$ is from $\{\CFIXmig,\CFIXmig'\}$.  As for the DHL,
the curve \\ $\{UW-V^2=0\}$ is invariant and on it $R$ is conjugate to $z \mapsto
z^k$, with $z=0$ and $z=\infty$ corresponding to the superattracting fixed
points $\CFIXmig$ and $\CFIXmig'$, respectively.  Therefore,
$c_\infty(\CFIXmig,R) = c_\infty(\CFIXmig',R) = k < 2k = \deg R$, giving that
these points are not maximally superattracting.  Therefore, $\mathcal{E}(a) =
\emptyset$.  Meanwhile, the fixed point $b_0 \in I^{-}(R)$ is not superattracting,
so $\mathcal{E}(b) = \emptyset$.

Therefore, the Equidistribution Theorem gives that the (normalized) iterated
preimages of any algebraic curve $A \subset \CP^2$ equidistribute to the Green Current
for $R$.  In particular, the Global Lee-Yang-Fisher Current Theorem holds for the $k$-fold DHL.

\subsection{Triangles}
The renormalization mapping for the hierarchical lattice generated by the triangle shown in Figure~\ref{FIG:GEN_GRAPHS} is:
\begin{align*}
R[U:V:W] = [U^3+UV^2:UV^2+V^2W:V^2W+W^3].
\end{align*}
The lift $\hat{R} : \mathbb{C}^3 \rightarrow  \mathbb{C}^3$ has Jacobian
\begin{align*}
   \det D\hat R = 6\,V \left(U+W \right)  \left( {U}^{2}{V}^{2}+3\,{U}^{2}{W}^{2}-{V}^{
2}WU+{V}^{2}{W}^{2} \right),
\end{align*}
so that
the critical locus is the union of two lines and a quartic:
\begin{align*}
L_1 &:= \{V=0\}, \\
L_2 &:= \{U+W = 0\}, \mbox{ and } \\
Q&:= \{U^2V^2+3U^2W^2-UV^2W+V^2W^2\},
\end{align*}
each with multiplicity one.

Hypothesis (i) is satisfied because the indeterminacy locus $I(R)$ is
non-empty,  consisting of three points: $p := [0:1:0]$ and $q_\pm := [\pm
i:1:\mp i]$.  Point $p$ is resolved in one blow-up.  In a neighborhood of
$L_\ex(p)$ the lifted map $\tilde{R}$ has critical locus consisting of
$L_\ex(p)$ and the proper transforms $\tilde{L_2}$ and $\tilde{Q}$, each with
multiplicity one.  Moreover, $\tilde{L_2}$ and $\tilde{Q}$ intersect $L_\ex(p)$
at three distinct points, each of which is a normal crossing singularity.
Therefore, Lemma~\ref{LEM:SPECIAL_JACOBIAN} gives that any point $z \in
L_\ex(p)$ has $\sigma(z,\tilde{R}) = 2$.
Proposition~\ref{PROP:ESTIMATE_FOR_RATIONAL_MAP} then gives $\sigma(p,R) = 2 <
{\rm deg}(R)$.

Both of the indeterminate points $q_\pm$ are also resolved with one blow-up.
The lift $\tilde{R}$ does not have $L_\ex(q_\pm)$ as a critical curve.   The
proper transforms $\tilde{L_2}$ and $\tilde{Q}$ become disjoint and smooth in a
neighborhood of the exceptional divisor $L_\ex(q_\pm)$, so that $\det D
\tilde{R}$ has multiplicity $0$ or $1$ at every point of $L_\ex(q_\pm)$.
Lemma~\ref{LEM:FJ_ESTIMATE} and
Proposition~\ref{PROP:ESTIMATE_FOR_RATIONAL_MAP} then imply that $\sigma(q_\pm,R)
= 2 < {\rm deg}(R)$.  We conclude that Hypothesis (ii) is satisfied.

We now check Hypothesis (iii).  Since each of the critical curves has
multiplicity one, Lemma~\ref{LEM:FJ_ESTIMATE} gives that $\sigma(z,R) \leq 2$
at every smooth point of the critical locus.  Let \\ $r:= [1:0:-1] = L_1 \cap
L_2$.  Then the singular points of the critical locus are  $\{r,e,e'\} \cup
I(R)$.  Moreover, one can choose local coordinates $(x,y)$ in a neighborhood of
$r$ such that $\det DR~\asymp~xy$.  Hence Lemma~\ref{LEM:SPECIAL_JACOBIAN}
gives $\sigma(r,R) = 2$, so that $r \not \in \BADSET$.  Thus, $\BADSET \subset
\{e,e'\}$ is therefore invariant under $R$.   

Both $L_1$ and $Q$ go through two superattracting fixed points $e:=[1:0:0]$ and \\
$e':=[0:0:1]$, so they are not collapsed by $R$.  Meanwhile, a direct
calculation shows that $L_2$ is collapsed by $R$ to $s:=[1:0:-1] \in I^{-}(R)$,
which is a fixed point of $R$.  Since $\BADSET$ is
invariant under $R$ and disjoint from $I(R) \cup I^{-}(R) = \{p,q_\pm,s\}$,
Hypothesis (iii) is satisfied.

By Proposition \ref{PROP:CHECKING_FOR_EXCEPTIONAL_POINTS}, any point of
$\mathcal{E}(a)$ passes through $\BADSET \subset \{e,e'\}$, which is invariant
under $R$.  Moreover, the fixed points $e$ and $e'$ are not maximally
superattracting since the curve $\{UW-V^2~=~0\}$ is invariant for this
renormalization mapping as well, and on it $R$ is conjugate to $z \mapsto z^2$.
Therefore, $\mathcal{E}(a) = \emptyset$.  The only point of $I^{-}(R)$ is the
fixed point $s$, which is not superattracting.  Therefore, $\mathcal{E}(b) =
\emptyset$.

Thus, the Equidistribution Theorem gives
that the (normalized) iterated preimages of any algebraic curve $A$ equidistribute to the
Green current for $R$.  In particular, the Global Lee-Yang-Fisher Current
Theorem also holds for the hierarchical lattice generated by the triangle.

\subsection{Tripod}
The renormalization mapping for the hierarchical lattice generated by the tripod shown in Figure~\ref{FIG:GEN_GRAPHS} is:
\begin{align*}
& R[U:V:W] =\\
& [U^{3}+U^{2}V+V^{3}+V^{2}W : V \left( U^{2}+UV+VW + W^{2} \right):U{V}^{2}+{V}^{3}+V{W}^{2}+{W}^{3}],
\end{align*}
which has indeterminacy set $I(R) = \{[\omega:1:\overline{\omega}] \, : \, \omega^3 = -1\}$.

One can resolve the indeterminacy point $p = [-1:1:-1]$ by doing three
blow-ups, each one on done on the previously created exceptional divisor.  When
one applies Lemmas~\ref{LEM:FJ_ESTIMATE}, \ref{LEM:SPECIAL_JACOBIAN}, and
Proposition~\ref{PROP:ESTIMATE_FOR_RATIONAL_MAP} one finds $\sigma(p,R) \geq 3
= \deg(R)$.  Therefore, Hypothesis (ii) of the Equidistribution Theorem fails
for this mapping.  In particular, we do not know if the  Global Lee-Yang-Fisher
Current Theorem holds for the hierarchical lattice generated by the tripod.

\subsection{Split Diamond}
The renormalization mapping for the hierarchical lattice generated by the split diamond shown in Figure~\ref{FIG:GEN_GRAPHS} is:
{\small
\begin{align*}
R[U:V:W] = \left[U^{5}+2U^{2}V^{3}+V^{4}W:V^{2} \left(U^{3}+2UVW+W^{3} \right) :UV^{4}+2{V}^{3}{W}^{2}+{W}^{5}\right].
\end{align*}
}
The lift $\hat{R} : \mathbb{C}^3 \rightarrow  \mathbb{C}^3$ has Jacobian
\begin{align}\label{EQN:SPLIT_DIAMOND_JAC}
   \det D\hat R = &10V \left( UW-{V}^{2} \right)  \\
 & ( 4{V}^{3}{U}^{6}+5{W}^{3}{U}^{6}+5{V}^{2}{W}^{2}{U}^{5}+5{V}^{4}W{U}^{4}+15V{W}^{4}{U}^{4} \nonumber \\
&+ {V}^{6}{U}^{3}+5{V}^{3}{W}^{3}{U}^{3}+5{W}^{6}{U}^{3}+5{V}^{5}{W}^{2}{U}^{2}+5{V}^{2}{W}^{5}{U}^{2} \nonumber \\
&-{V}^{7}WU+5{V}^{4}{W}^{4}U+{V} ^{6}{W}^{3}+4{V}^{3}{W}^{6}). \nonumber
\end{align}

One can check that $I(R) = \{[0:1:0],[\omega:1:\overline{\omega}] \, : \,
\omega^3 = -1\}$.  Each of these indeterminate points requires three blow-ups
to resolve.  One can work through the
details in Maple and check that the lifted mappings have Jacobian of the form
$\det D \tilde{R} \asymp x^a y^b$ with ${\rm max}(a,b) \leq 3$ at each of the
points on each of the exceptional divisors.  Thus,
Lemma~\ref{LEM:SPECIAL_JACOBIAN} and
Proposition~\ref{PROP:ESTIMATE_FOR_RATIONAL_MAP} give $\sigma(p,R) =4 <
\deg(R)$ for each $p \in I(R)$.  Therefore, Hypotheses (i) and (ii) are satisfied.

Let us check Hypothesis (iii).  
Since each of the critical curves has multiplicity~$1$, we have $\sigma(p,R) \leq 2$ for 
any point outside the singular locus of the critical set.  Meanwhile, 
each of the intersections of any two of
the three critical curves and each of the singular points of $C_3$ lie in $I(R)
\cup \{\CFIXmig,\CFIXmig'\}$.    We conclude that $\BADSET \subset \{\CFIXmig,\CFIXmig'\}$ and is therefore invariant under $R$ and disjoint
from $I(R)$.

Each of the irreducible components of the critical locus (corresponding to the
three factors of (\ref{EQN:SPLIT_DIAMOND_JAC})) passes through the two
superattracting fixed points $\CFIXmig = [1:0:0]$ and $\CFIXmig' = [0:1:0]$ of
$R$ and therefore there are no curves collapsed by $R$.  Therefore, $I^{-}(R) =
\emptyset$.  Since $\BADSET$ is invariant under $R$ and disjoint from $I(R)
\cup I^{-}(R)$, Hypothesis~(iii) is satisfied.

Since $I^{-}(R) = \emptyset$ we have that $\mathcal{E}(b) = \emptyset$.  Meanwhile, the
critical curve $\{UW-{V}^{2} = 0 \}$ is invariant under $R$ and on it $R$ is
conjugate to $z \mapsto z^2$, with $z=0$ and $z=\infty$ corresponding to the
superattracting fixed points $\CFIXmig$ and~$\CFIXmig'$, respectively.
Therefore, $c_\infty(\CFIXmig,R) = c_\infty(\CFIXmig',R) = 2 < 5 = \deg R$,
giving that these points are not maximally superattracting.  Therefore, we also
have that $\mathcal{E}(a) = \emptyset$ and the Equidistribution Theorem gives
that the (normalized) iterated preimages of any algebraic curve $A \subset
\CP^2$ equidistributes to the Green Current for $R$.  In particular, the Global
Lee-Yang-Fisher Current Theorem also holds for the hierarchical lattice
generated by the split diamond.

\appendix


\section{Elements of Complex Geometry}\label{APP:COMPLEX GEOM}

This appendix presents background on the material from complex geometry that we
need beyond what is described in \cite[Appendix A]{BLR1}.  We refer the reader there for background
on rational maps, blow-ups, and divisors.

\subsection{Structure of rational maps $f: \CP^2 \rightarrow \CP^2$}\label{APPENDIX:STRUCTURE_RAT_MAP}

Throughout this subsection, we suppose $f: \CP^2 \rightarrow \CP^2$ is a dominant rational map of degree
$d \geq 2$.

There is a finite sequence of blow-ups
$\pi: \widetilde{\CP^2} \rightarrow \CP^2$ done over the points of the indeterminacy set $I(f)$
so that $f$ lifts to a regular map $\tilde{f}: \widetilde{\CP^2} \rightarrow
\CP^2$, making the following diagram commute
\begin{eqnarray}\label{RESOLUTION}
\xymatrix{
\widetilde{\CP^2} \ar[d]^\pi \ar[dr]^{\tilde{f}} & \\
\CP^2 \ar[r]^f & \CP^2,
}
\end{eqnarray}
wherever $f \circ \pi$ is defined (see \cite[Ch. IV, \S 3.3]{SHAF}).
We call such a commutative diagram a {\em resolution of the indeterminacy} of $f$.

For any $Y \subset \CP^2$ one can use the resolution of indeterminacy (\ref{RESOLUTION}) to define
\begin{eqnarray}\label{EQN:MAPPING_OF_SETS}
f(Y) := \tilde{f}(\pi^{-1}(Y)) \quad \mbox{and} \quad f^{-1}(Y):= \pi\left(\tilde{f}^{-1}(Y)\right).
\end{eqnarray}These definitions are independent of the choice of resolution.

The following two lemmas are well-known; see for example \cite{FS}:

\begin{lem}\label{LEM:COLLAPSED_CURVES_MEET_IF} Any curve that is collapsed by $f$ passes through $I(f)$.
\end{lem}

Let ${\rm Crit}(f)$ denote the critical locus of $f|_{\CP^2 \setminus I(f)}$
and let $V:=f(I(f) \cup {\rm Crit}(f))$.  Then the mapping
\begin{align*}
f|_{ \CP^2 \setminus f^{-1}(V)} : \CP^2 \setminus f^{-1}(V) \rightarrow \CP^2 \setminus V
\end{align*}
is a finite degree covering map.  The {\em topological degree} $d_{\rm top}(f)$ is the degree of this cover.

\begin{lem}\label{LEM:DTOP} If $I(f) \neq \emptyset$, then ${\rm deg}_{\rm top}(f) < d^2$.
\end{lem}


\begin{rem}\label{REM:DEG_JAC}
Let $\hat{f}: \C^3 \rightarrow \C^3$ denote a homogeneous lift of $f$.  Then
the complex Jacobian ${\rm Jac}(\hat{f}) := \det D \hat{f}$ is a homogeneous
polynomial of degree $3(d-1)$.  It induces a divisor ${\rm Jac}(f)$ on $\CP^2$
of the same degree.
\end{rem}

\begin{lem}[\bf Whitney Fold Normal Form\footnote{Here we are using a more general notion of ``Whitney Fold''
that was used in \cite[Appendix D.2]{BLR1}, where only exponent $\FOLDEXPONENT_C = 2$ was considered.}]\label{LEM:NORMALFORM} Suppose $C$ is an irreducible component of the critical locus
that is not collapsed by~$f$.  Then there is a finite set $S \subset C$ and
$\FOLDEXPONENT \equiv \FOLDEXPONENT_C \in \mathbb{Z}_+$ such that for any $p \in C \setminus S$
there exist local holomorphic coordinates $(x,y)$ in a neighborhood of $p$ and $(z,w)$ in a neighborhood of $f(p)$ in which
\begin{align*}
(z,w) = f(x,y) = (x,y^{\FOLDEXPONENT}).
\end{align*}
In particular, at every $p \in C \setminus S$ we have $\sigma(p,f) = e(p,f) =  \FOLDEXPONENT$ and for every $p \in C$
we have $\sigma(p,f) \geq \FOLDEXPONENT$.
\end{lem}
\noindent
(The volume exponent $\sigma(p,f)$ is defined in \S \ref{SUBSEC:EQUIDIST} and the local topological degree $e(p,f)$ in \S \ref{SUBSEC:NO_CURVES_MAX_DEG}.)

\begin{proof}
We include the following points from $C$ in $S$:
\begin{itemize}
\item[(i)] singular points of the critical locus of $f$,
\item[(ii)] points $p \in C$ for which $T_p C \subset {\rm Ker} Df(p)$, and
\item[(iii)] points of $I(f)$.
\end{itemize}
Since $C$ is not collapsed by $f$, the set $S$ is finite and, since $C$ is irreducible, $C \setminus S$
is connected.   Essentially the same proof as the characterization of Whitney Folds presented
in Lemma D.2 of \cite{BLR1} gives that for each point $p_0$
satisfying (i) - (iii) we can find local coordinates in some neighborhoods
of $p_0$ and $f(p_0)$ in which the map has the normal form $(z,w) = f(x,y) =
(x,y^\FOLDEXPONENT)$ for some integer $\FOLDEXPONENT \geq 2$.  Moreover, as $C \setminus S$ is connected, we see that
the exponent $r$ is constant on $C \setminus S$.  
From the normal form and Lemma \ref{LEM:FJ_ESTIMATE}, we see that at every
point of $C \setminus S$ we have $\sigma(p,f) = \mu(p,f) + 1 = e(p,f) = \FOLDEXPONENT$.  
The fact that $\sigma(p,f) \geq \FOLDEXPONENT$ for every $p \in \C$ is because $D_{\geq  \FOLDEXPONENT}$ is algebraic, by Lemma \ref{LEM:FINITELY_MANY_SIGMA_VALUES}.
\end{proof}

\subsection{Hyperplane bundle}\label{hyp bundle}
The hyperplane bundle and its tensor powers provide a convenient way to work with divisors on $\CP^k$.

The fibers of $\pi: \C^{k+1} \sm \{0\} \ra \CP^k$ are punctured complex lines $\C^*$. Compactifying each of
these lines at infinity, we add to $\C^{k+1}~\sm~\{0\}$ the line at infinity
$L_\infty\isom \CP^k$ obtaining the total space $$(\C^{k+1})^* \cup
L_\infty\isom  (\CP^{k+1})^*:=  \CP^{k+1}\sm \{ 0\}.$$ The projection naturally
extends to $\pi :  (\CP^{k+1})^* \ra \CP^k$, whose fibers are complex lines
$\C$.  The resulting line bundle is called the {\it hyperplane bundle} over $\CP^k$.  

In homogeneous coordinates $(z_0:\dots :z_k: t)$ on $\CP^{k+1}$,
this projection is just 
\begin{equation}\label{taut bundle pi}
   \pi: (z_0:\dots :z_k: t)\mapsto  (z_0:\dots :z_k),
\end{equation}
with $L_\infty= \{t=0\}$, $(\C^{k+1})^* = \{t=1\}$, and the map $(z:t)\mapsto t/\|z\|$
parameterizing the fibers (here $\|z\|$ stands for the Euclidean norm of $z\in \C^{k+1} \sm \{0\}$). 
This line bundle is endowed with the natural {\it Hermitian structure}:
$\|(z: t)\| = |t|/\|z\|$. 

Any non-vanishing linear form $Y$ on $\C^{k+1}$ determines a section of the hyperplane bundle: 
\begin{equation}
    s_Y: z\mapsto (z:Y(z)), \quad z\in \C^{k+1}.
\end{equation}
The divisor $D_Y$ (a projective line counted with multiplicity $1$) is precisely the zero divisor of $s_Y$.

The $d$th tensor power of the hyperplane bundle can be described as follows.
Its total space $X^d$ is the quotient of $(\C^{k+2})^*$ by the $\C^*$-action
$$
    (z_0,\dots,z_k,t)\mapsto (\la z_0, \dots ,\la z_k , \la^d t), \quad \la\in \C^*. 
$$
We denote the equivalence class of $(\hat z,t)$ using the ``homogeneous'' coordinates $(\hat z:t)$.
The projection $X^d \ra \CP^k$ is natural, as above (\ref{taut bundle pi}).  A
non-vanishing homogeneous polynomial $P$ on $\C^{k+1}$ of degree $d$ defines a
holomorphic section $s_P$ of this bundle given by \hbox{$s_P(z)~=~(\hat z: P(\hat z))$.}
This bundle is endowed with the Hermitian structure:
\begin{eqnarray}\label{EQN:HERMITIAN_STRUCTURE}
\|(z: t)\| = |t|/\|z\|^d.
\end{eqnarray}

More generally, any divisor $D = D_P - D_Q$ defines a section $s_D$ of the
$\deg(D)$-th tensor power of the hyperplane bundle, defined by
$s_D(z) = (\hat z: P(\hat z)/Q(\hat z))$.  One can recover $D$ from $s_D$ by taking its
zero divisor.

\subsection{Currents}\label{APPENDIX:CURRENTS}
We will now give a brief background on currents; for more details see
\cite{DERHAM,LELONG} and the appendices from \cite{DS_SURVEY,S_PANORAME}.  Currents are
naturally defined on general complex (or even smooth) manifolds, however to
continue our discussion of rational maps, divisors, etc, we restrict our
attention to projective manifolds.

A {\em $(1,1)$-current}
$\current$  on $V$ is a continuous linear functional on  $(k-1,k-1)$-forms with
compact support.  It can be also defined as a generalized differential
$(1,1)$-form $\sum \current_{ij} dz_i d\bar z_j$ with distributional coefficients.

A basic example is the current $[A]$ of integration over the regular points $A_{\rm reg}$ of an algebraic hypersurface $A$:
      $$\om\mapsto \int_{A_{\rm reg}} \om,$$
where $\om$ is a test $(k-1,k-1)$-form.  The current of integration over a divisor $D$ is defined by extending linearly.

The space of currents is given the distributional topology: $\current_n \ra \current$
if $\current_n(\om)~\ra~\current(\om)$ for every test form $\omega$.

A differential $(k-1,k-1)$-form $\om$ is called  positive if its integral over
any complex subvariety is non-negative. A (1,1)-current $\current$ is called {\it
positive} if $\current(\om)\geq 0$ for any positive (1,1)-form.
A current $\current$ is called {\it closed} if $d \current=0$, where the differential $d$
is understood in the distributional sense.

In this paper, we focus on closed, positive $(1,1)$-currents.  They have a 
simple description in terms of local potentials, rather 
analogous to the definition of divisors.

Recall that  $\di$ and $\dibar$ stand for the holomorphic and anti-holomorphic parts of the
external differential $d=\di+\dibar$. 
Their composition
 $\frac i{\pi} \di\dibar$ can be\footnote{Many authors introduce real operators $d = \di +\dibar$ and $d^c = \frac{i}{2\pi}(\dibar - \di)$ and write $dd^c$ instead of $\PLapl$.  We use $\PLapl$ to avoid confusion between the operator $d$
and the algebraic degree of a map.}
thought of as a kind of 
``pluri-Laplacian''
because, 
given a $C^2$-function $h$, the restriction of $\PLapl h$ to any
non-singular complex curve $X$ is equal to the form $\frac{1}{2\pi}\De (h|X) dz \wedge d{\bar z}$,
where $z$ is a local coordinate on $X$ and  $\De$ is the usual Laplacian in this coordinate.

If $U$ is an open subset of $\C^k$ and $h: U \ra [-\infty,\infty)$ is a
plurisubharmonic (PSH) function, then  $\PLapl h$ is a closed
(1,1)-current on $U$.  Conversely, the $\di \dibar$-Poincar\'e Lemma asserts
that every closed, positive $(1,1)$-current on $U$ is obtained this way.

Therefore, any closed positive $(1,1)$-current $\current$ on a manifold $V$ can be
described using an open cover $\{U_i\}$ of $V$ together with PSH functions
$h_i: U_i \ra [-\infty,\infty)$ that are  chosen so that $\current =\PLapl h_i$ in
each $U_i$.  The functions $h_i$ are called {\em local potentials} for $\current$
and they are required to  satisfy the compatibility condition that $h_i - h_j$
is pluriharmonic (PH) on any non-empty intersection $U_i \cap U_j \neq \emptyset$.  The {\em
support} of $\current$ is defined by:
\begin{eqnarray*}
\supp \current &:=& \{z \in V \, : \, \mbox{if $z \in U_j$ then $h_j$ is not pluriharmonic at $z$}\}.
\end{eqnarray*}
The compatibility condition assures that that above set is well-defined.  

The Poincar\'e-Lelong formula describes the current of integration over a divisor \hbox{$D = \{U_i,g_i\}$} by the system of local potentials $h_i:= \log|g_i|$.
I.e.,  on each $U_i$ we have $[D] = \PLapl \
\log|g_i|.$ The result is a closed $(1,1)$-current, which is
positive iff $D$ is effective (i.e. the multiplicities $k_i,\ldots,k_r$ are
non-negative).

Suppose $R: V \ra W$ is a dominant
rational map and $\current$ is a closed-positive $(1,1)$-current on $W$.  
The pullback $R^* \current$ is closed positive $(1,1)$-current on $V$, defined as follows.
First, one obtains a closed positive $(1,1)$-current $R^* \current$ defined on $V \sm I(R)$
by pulling-back the system of local potentials defining $\current$ under $R: V \sm I(R)~\ra~W$.
One then extends
$R^* \current$ trivially through $I(R)$, to obtain a closed, positive $(1,1)$-current defined on all of $V$. 
(By a result of Harvey and Polking \cite{HARVEY_POLKING}, this extension is closed.)
See \cite[Appendix A.7]{S_PANORAME} for further details.  Pullback is continuous with respect to the distributional topology.

\ssk
Similarly to divisors, there is a particularly convenient description of closed, positive $(1,1)$-currents on $\CP^k$.
Associated to any PSH function  $H: \C^{k+1}~\ra~[-\infty,\infty)$, having the homogeneity 
\begin{eqnarray}\label{EQN:LOG_HOMOG}
H(\la \hat z) = m \log|\la| + H(\hat z),
\end{eqnarray}
for some $m > 0$, is a closed, positive $(1,1)$-current, denoted by $T_H \equiv
\pi_* \left(\PLapl H\right)$ that is defined as follows.  For any open covering
$\{U_i\}$ of $\CP^k$ admitting local sections \hbox{$s_i: U_i \ra \C^{k+1} \sm \{0\}$} of the canonical projection $\pi$, $T_H$ is defined by the system of
local potentials $\{U_i, H \circ s_i\}$.  I.e.,  in each $U_i$, $T_H$ is
defined by $T_H = \PLapl  \ H \circ s_i$.  Moreover, every closed positive
$(1,1)$-current on $\CP^k$ is described in this way; See \cite[Thm.
A.5.1]{S_PANORAME}.  The function $H$ is called the {\em pluripotential} of~$T_H$.

The {\em mass} of the current $T$ is defined to be
\begin{align*}
\|T\| = \int_{\CP^k} T \wedge \omega_{\rm FS}^{k-1},
\end{align*}
where $\omega_{\rm FS}$ is the Fubini-Study $(1,1)$ form on $\CP^k$ and $T \wedge \omega_{\rm FS}^{k-1}$ is the positive measure defined by 
$\left<T \wedge \omega_{\rm FS}^{k-1},f\right> := \left<T,f \omega_{\rm FS}^{k-1} \right>$.  A calculation shows that the mass of $T_H$ can be computed from
the potential $H$ as
$\|T_H\| = m$, where $m$ is the constant from (\ref{EQN:LOG_HOMOG}).

If $R: \CP^k \ra \CP^k$ is a rational map, the action of pull-back is described by
\begin{eqnarray}
R^* T_H = T_{H \circ \hat R}.
\end{eqnarray}
If $R$ has algebraic degree $d$, then the lift $\hat R$ is a $k+1$-tuple of homogeneous polynomials of degree $d$.  It follows
that the mass satisfies  
\begin{align*}
\|R^* T_H\| = d \|T_H\|.
\end{align*}

\subsection{Kobayashi hyperbolicity and normal families}
\label{APP:KOB}

In \S \ref{SUBSEC:JULIA_AND_FATOU} we use the Kobayashi metric in order to
prove that the iterates $\Rmig^n$ form a normal family on certain subspaces of
$\CP^2$.  Here we recall the relevant definitions and some important results
that we use.  The reader can consult the books
\cite{KOB_HYPERBOLIC,LA_HYPERBOLIC} and the original papers by M.~Green
\cite{GREEN:PAMS,GREEN:AJM} for more details.  For more dynamical applications,
see e.g. \cite{DS_SURVEY,S_PANORAME}.

The Kobayashi pseudometric is a natural generalization of the Poincar\'e metric
on  Riemann surfaces. Let $\|\cdot\|$ stand for the Poincar\'e metric on the  unit disk $\D$.
Let $M$ by a complex manifold.
Pick a tangent vector $\xi \in T M$, and  let $\HH(\xi)$ be the family of holomorphic
curves $\gamma :\D \ra M$ tangent to the line  $\C\cdot \xi$ at $\gamma(0)$.
Then   $Df(v) = \xi$   for some $v\equiv v_\gamma \in T_0\D$,
and the Kobayashi pseudometric is defined to be:
\begin{eqnarray}
ds_M(\xi) = {\rm inf}_{\gamma \in \HH(\xi)} \|v_\gamma\|.
\end{eqnarray}
The Kobayashi pseudometric is designed so that holomorphic maps are distance
decreasing: if $f:U\ra M$ is holomorphic then
$ds_M (Df(\xi)) \leq ds_U (\xi)$.

The reason for  ``pseudo-'' is that for certain complex
manifolds $M$, $ds(\xi)$ can vanish for some non-vanishing tangent vectors  $\xi \neq 0$.
For example, $ds$ identically vanishes on  $\C^k$ or $\CP^k$.
A complex manifold $M$ is called {\em Kobayashi hyperbolic} if $ds$
is non-degenerate: $ds(\xi)>0$ for any non-vanishing $\xi\in TM$.
Then it induces a (Finsler) metric on $M$.

Let $N$ be a compact complex manifold. Endow it with some Hermitian metric
$|\cdot|_N$.  A complex submanifold $M\subset N$ is called {\em hyperbolically
embedded in $N$} if
the Kobayashi pseudometric on $M$  dominates the Hermitian metric on $N$, i.e.,
there exists $c>0$ such that  $ds_M(\xi) \geq  c |\xi|_N$ for all $\xi \in TM$.
Obviously, $M$ is Kobayashi hyperbolic in this case.

A complex manifold $M$ is called {\em Brody hyperbolic} if there are no non-constant
holomorphic mappings $f:\C~\ra~M$.  If $M$ is Kobayashi hyperbolic, it is also
Brody hyperbolic, but the converse is generally not true unless $M$ is compact.  

An open subset of $\CP^2$ that is Kobayashi hyperbolic, but not hyperbolically
embedded in $\CP^2$ is described in \cite[Example 3.3.11]{KOB_HYPERBOLIC}
and an open subset of $\C^2$ that is Brody hyperbolic but not Kobayashi
hyperbolic is described in \cite[Example 3.6.6]{KOB_HYPERBOLIC}.

\msk
A family $\FF$ of holomorphic mappings from a complex manifold $U$ to a
complex manifold $M$ is called {\em normal}  if every sequence in $\FF$ either
has a subsequence converging locally uniformly or a subsequence that diverges
locally uniformly to infinity in $M$.
In the case that $M$ is embedded into some compact manifold $Z$,
a stronger condition is that
$\FF$ is {\it precompact} in ${\rm Hol}(U,Z)$ (where $\Hol(U,Z)$ is the space of holomorphic mappings
$U\ra Z$ endowed with topology of uniform convergence on compact subsets of $U$).


\begin{prop}
\label{COR:NORMAL_FAMILIES}
Let $M$ be a hyperbolically embedded complex submanifold of a compact complex manifold $N$.
Then for any  complex manifold $U$,
the family $\Hol(U,M)$ is precompact in $\Hol (U, N)$.
\end{prop}

\noindent
See Theorem 5.1.11 from \cite{KOB_HYPERBOLIC}. 

The classical Montel's Theorem asserts that the family of holomorphic maps
\hbox{$\D~\ra~\C\sm \{0,1\}$} is normal (as $\C\sm \{0,1\}$ is a hyperbolic Riemann surface).
It is a foundation for the whole Fatou-Julia iteration theory.
Several higher dimensional versions of Montel's Theorem,
due to M.~Green \cite{GREEN:PAMS,GREEN:AJM}, are now available.
Though their role in dynamics is not yet so prominent,
they have found a number of interesting applications.
Below we will formulate two particular results  used in this paper (see \S \ref{SUBSEC:JULIA_AND_FATOU}).
%
%
%
The following is Theorem 2 from \cite{GREEN:PAMS}:
\begin{thm}
\label{THM:GREEN1}
Let $X$ be a union of (possibly singular) hypersurfaces $X_1,\ldots,X_m$ in a compact
complex manifold $N$.
Assume $N\sm X$ is Brody hyperbolic and
$$ X_{i_1} \cap \cdots \cap X_{i_k} \sm \left(X_{j_1} \cup \cdots X_{j_l}\right) \mbox{is Brody hyperbolic}$$
for any choice of distinct multi-indices $\{i_1,\ldots,i_k,j_1,\ldots,j_l\} = \{1,\ldots,m\}$.
Then $N \sm X$ is a complete hyperbolically  embedded submanifold of $N$.
\end{thm}

In the last section of \cite{GREEN:PAMS},  the following result is proved:

\begin{thm}\label{conic and 3 lines}
Let $M=\CP^2 \sm \left(Q \cup X_1 \cup X_2 \cup X_3\right)$, where
$Q$ is a non-singular conic and $X_1, X_2, X_3$ are lines.
Then  any non-constant holomorphic curve
$f:~\C~\ra~~M$ must lie in a line $L$ that is tangent to $Q$
at an intersection point with one of the lines, $X_i$,
and that contains the intersection point  $X_j \cap X_l$ of the other two lines.
\end{thm}

The configurations that appear in this theorem are related to amusing projective triangles:

\subsection{Self-dual triangles}\label{self-dual triangles}

Let $Q(z)=\sum q_{ij} z_i z_j$ be a non-degenerate quadratic form in $E\isom \C^3$,
and $X=\{Q=0\}$ be the corresponding conic in $\CP^2$.
The form $Q$ makes the space $E$ Euclidean,
inducing duality between points and lines in $\CP^2$.
Namely, to a point $z=(z_0:z_1:z_2)$ corresponds the line
$L_z= \{\zeta :\ Q(z, \zeta)=0\}$ called the {\it polar} of $z$ with respect to $X$
(here we use the same notation for the quadratic form and the corresponding inner product).
Geometrically, this duality looks as follows.
Given a point $z\in \CP^2$,
there are two tangent lines from $z$ to $X$. Then $L_z$ is the line passing through the
corresponding tangency points. (In case $z\in X$, the polar is tangent to $X$ at $z$).

Three points $z_i$ in $\CP^2$ in general position are called a ``triangle'' $\De$ with vertices $z_i$.
Equivalently, a triangle can be given by three lines $L_i$ in general position, its ``sides''.
Let us say that $\De$ is {\it self-dual} (with respect to the conic $X$)
if its vertices are  dual to the opposite sides.

\begin{lem}
  A triangle $\De$ with vertices $z_i$ is self-dual if and only if the corresponding vectors $\hat z_i\in E$
form an orthogonal basis with respect to the inner product~$Q$.
\end{lem}

All three sides of a self-dual triangle satisfy the condition of Theorem~\ref{conic and 3 lines},
so they can give us exceptional holomorphic curves $\C\ra \CP^2\sm (Q\cup X_1\cup X_2\cup X_3)$.    


\section {Green Current}\label{APP:COMPLEX_DYNAMICS}

\subsection{Green potential}\label{Green potential sec}

A rational mapping $R : \CP^m \rightarrow \CP^m $ is called {\em algebraically
stable} if there is no integer $n$ and no collapsing hypersurface $V \subset
\CP^{m}$ so that $R^n(V)$ is contained within the indeterminacy set of $R$,
\cite[p. 109]{S_PANORAME}.   (See also \cite[Appendix A4 and A5]{BLR1}.)

\begin{thm}[see \cite{S_PANORAME},Thm. 1.6.1]\label{Green potential}
  Let $R: \CP^m \ra \CP^m$ be an algebraically stable rational map of degree $d$.
Then the limit
$$
     G = \lim_{n\to \infty} \frac 1{d^n} \log \|\hat R^n \|
$$
exists in $L^1_\loc(\C^3)$ and determines a plurisubharmonic function. This function satisfies the following
equivariance properties:
\begin{equation}\label{la-equiv}
      G(\la z)= G(z)+\log |\la|, \quad \la\in \C^*,
\end{equation}
$$  G\circ \hat R = d G. $$
\end{thm}

\noindent
It is called the {\it Green potential} of $R$.

\subsection{Green current}\label{Green cur sec}

Applying $\PLapl$ to the Green potential, we obtain:

\begin{thm}[see \cite{S_PANORAME},Thm. 1.6.1]\label{Green current}
   Let $R: \CP^m \ra \CP^m$ be an algebraically stable rational map of degree $d$.
Then $\Green=\pi_*(\PLapl G)$ is a closed positive (1,1)-current on $\CP^m$
satisfying the equivariance relation: $R^* \Green = d \cdot \Green$.
\end{thm}


\noindent
The current $\Green$  is called the {\it Green current} of $R$.

The set of {\em normal}\footnote{Not to be confused with the notion of normal families.}
points for an algebraically stable rational map 
$R: \CP^m~\ra~\CP^m$ is:
\begin{eqnarray*}
\Regular := \left\{\begin{array}{cc} x \in \CP^m \, : & \, \mbox{there exits 
neighborhoods $U$ of $x$ and $V$ of $I(R)$} \\  & \mbox{ so that
$f^n(U) \cap V = \emptyset$ for every $n \in \N$.}\end{array}\right\}
\end{eqnarray*}
The normal points form an open subset of $\CP^m$.

One primary interest in the Green current $\Green$ is the following
connection between its support $\supp \Green$ and the Julia set $J_R$.
(See \S \ref{SUBSEC:JULIA_AND_FATOU} for the definitions of the Fatou and Julia
sets.) Note that $\supp \Green$ is closed and backwards invariant, $R^{-1}
\supp \Green \subset \supp \Green$, since $R^* \Green = d\cdot \Green$.

\begin{thm}[see \cite{S_PANORAME},Thm. 1.6.5]\label{Green current vs Fatou}
Let $f: \CP^m \ra \CP^m$ be an algebraically stable rational map.   Then:
\begin{eqnarray*}
J_R \cap \Regular \subset \supp \Green \subset J_R.
\end{eqnarray*}
\end{thm}

\noindent
An algebraically stable rational map $R: \CP^2 \ra \CP^2$ for which $\supp \Green
\subsetneq J_R$ is given in \cite[Example 2.1]{FS}.


\section{Open Problems}\label{APP:PROBLEMS}

\begin{problem}[\bf Existence of Fisher and Lee-Yang-Fisher distributions]\label{PROP:GLOBAL_LIMIT}
For which classical lattices {\rm (}$\mathbb{Z}^d$ for $d \geq 2$, etc{\rm )} does the limit
{\rm (\ref{EQN:C2_L1_LIMIT})} exist?   As explained in Prop. \ref{PROP:L1_GLOBAL_LIMIT}, this
would justify existence of the limiting distributions of Lee-Yang-Fisher
zeros for these lattices.

Similarly, for which classical three and higher-dimensional lattices {\rm (}$\mathbb{Z}^d$ for $d \geq 3$, etc{\rm )}
does the limit {\rm (\ref{EQN:C2_L1_LIMIT})} exist?  It would justify the existence of a limiting distribution of
Fisher zeros for these lattices.
\end{problem}

\begin{problem}[\bf Geometric properties of the Lee-Yang-Fisher current]\label{PROB:LAMINAR}

The theory of geometric currents has become increasingly useful in complex
dynamics, see \cite{RUELLE_SULLIVAN,BLS,Du,DINH,dT,DDG} as a sample.  

The Green current $\Green$ is strongly laminar in a neighborhood of
$\BOTTOMmig$.  The structure is given by the stable lamination of
$\BOTTOMmig$ {\rm (}see \S \ref{SUBSEC:FIXED POINTS}{\rm )} together with transverse
measure obtained under holomomy from the Lebesgue measure on $\BOTTOMmig$.
However, $\Green$ is not strongly laminar in a
neighborhood of the topless Lee-Yang ``cylinder'' $\Cmigtl$. 

 One can see this also follows:  a disc within the invariant line $\LINV$
centered at $\LINV \cap \BOTTOMmig$ is within the stable lamination of
$\BOTTOMmig$.  Therefore, an open neighborhood  within $\LINV$ of $\LINV \cap
\Cmigtl$ would have to be a leaf of the lamination.  However, $\Green$
restricts to $\LINV$ in a highly non-trivial way, coinciding
with the measure of maximal entropy for $\Rmig | \LINV$. {\rm (}It is supported on the Julia set
shown in Fig.~\ref{FIG:INVARIANT_LINE_JULIA}{\rm )}.

Does $\Green$ have a weaker geometric structure?  For example, is it non-uniformly laminar \cite{BLS,Du} or woven \cite{DINH,dT,DDG}? 
\end{problem}

\begin{problem}[\bf Support for the measure of maximal entropy]\label{PROB:MEASURE_CP2} 
What can be said about the support of the measure of maximal entropy $\nu$ that
was discussed in \S \ref{SUBSEC:MEAS_MAX_ENT}?  Is the critical fixed point
$\FIXmig_c \in \LINV$ within $\supp \nu$?  A positive answer to this question
is actually equivalent to $\Cmig \cap \supp \nu \neq \emptyset$ and also to
$\supp \nu \cap \LINV = J_{\Rmig|\LINV}$.

\end{problem}

\begin{problem}[\bf Fatou Set] \label{PROB:GLOBAL_BASINS} In Thm. \ref{THM:SOLID_CYLINDERS} we showed that
certain ``solid cylinders'' are in $\WW^s(\CFIXmig)$ and $\WW^s(\CFIXmig')$.  Computer experiments suggest a much stronger result: 
\begin{conj}
$\WW^s(\CFIXmig) \cup \WW^s(\CFIXmig')$ is the entire Fatou set for $\Rmig$.
\end{conj}
\noindent
\end{problem}

\begin{problem}[\bf Julia Set] \label{PROB:GLOBAL_JULIA}  
Prop. \ref{PROP:JULIA_NEAR_LZERO} gives that in a neighborhood of
$\BOTTOMmig$, $J_{\Rmig}$ is a $C^\infty$ $3$-manifold.   What can be said about the global
topology of $J_\Rmig$?
\end{problem}

\begin{rem} Note that each of the above problems \ref{PROB:LAMINAR} -- \ref{PROB:GLOBAL_JULIA} has a natural  counterpart for~$\Rphys$.
\end{rem}


\bsk

\renewcommand\refname{References from  dynamics and complex geometry}

\msk

\renewcommand\refname{References from  mathematical physics}


\begin{thebibliography}{10}




\bibitem[BLS]{BLS}
E. Bedford, M.~Yu. Lyubich, and J. Smillie.
\newblock Polynomial diffeomorphisms of {${\bf C}\sp 2$}. {IV}. {T}he measure
  of maximal entropy and laminar currents.
\newblock {\em Invent. Math.}, 112(1):77--125, 1993.


\bibitem[BS]{BS}
E. Bedford and J. Smillie.
\newblock Polynomial diffeomorphisms of {${\bf C}\sp 2$}: currents, equilibrium
  measure and hyperbolicity.
\newblock {\em Invent. Math.}, 103(1):69--99, 1991.

\bibitem[Be]{BERGER}
P. Berger.
\newblock Persistence of laminations.
\newblock {\em Bull. Braz. Math. Soc. (N.S.)}, 41(2):259--319, 2010.

\bibitem[BLR]{BLR1}
Pavel Bleher, Mikhail Lyubich, and Roland Roeder.
\newblock Lee--{Y}ang zeros for the {DHL} and 2{D} rational dynamics, {I}.
  {F}oliation of the physical cylinder.
\newblock {\em J. Math. Pures Appl. (9)}, 107(5):491--590, 2017.


\bibitem[BD]{BD}
Araceli~M. Bonifant and Marius Dabija.
\newblock Self-maps of {${\Bbb P}^2$} with invariant elliptic curves.
\newblock In {\em Complex manifolds and hyperbolic geometry ({G}uanajuato,
  2001)}, volume 311 of {\em Contemp. Math.}, pages 1--25. Amer. Math. Soc.,
  Providence, RI, 2002.



\bibitem[Br]{BROLIN}
H. Brolin.
\newblock Invariant sets under iteration of rational functions.
\newblock {\em Ark. Mat.}, 6:103--144 (1965), 1965.



\bibitem[C]{CEGRELL}
U. Cegrell.
\newblock Removable singularities for plurisubharmonic functions and related
  problems.
\newblock {\em Proc. London Math. Soc. (3)}, 36(2):310--336, 1978.







\bibitem[DeSMa]{DeSMa}
J. De Simoi and S. Marmi.
\newblock Potts models on hierarchical lattices and renormalization group dynamics.
\newblock {\em J. Phys. A}, 42(9):095001, 21 pp., 2009.

\bibitem[DeS]{DeS}
J. De Simoi
\newblock Potts models on hierarchical lattices and renormalization
              group dynamics. {II}. {E}xamples and numerical results.
\newblock {\em J. Phys. A}, 42(9):095002, 16 pp., 2009.


\bibitem[Dil]{DILLER}
J. Diller
\newblock Dynamics of birational maps of $\mathbb{CP}^2$
\newblock {\em Indiana Univ. Math. J.} 45 (1996), no. 3, 721–772.



\bibitem[DDG]{DDG}
J. Diller, R. Dujardin, and V. Guedj.
\newblock Dynamics of meromorphic maps with small topological degree {I}: from
  cohomology to currents.
\newblock {\em Indiana Univ. Math. J.}, 59(2):521--561, 2010.


\bibitem[Di]{DINH}
T.-C. Dinh.
\newblock Suites d'applications m\'eromorphes multivalu\'ees et courants
  laminaires.
\newblock {\em J. Geom. Anal.}, 15(2):207--227, 2005.


\bibitem[DS1]{DS_ENTROPY}
T.-C. Dinh and N. Sibony.
\newblock Une borne sup\'erieure pour l'entropie topologique d'une application
  rationnelle.
\newblock {\em Ann. of Math. (2)}, 161(3):1637--1644, 2005.


\bibitem[DS2]{DS}
T.-C. Dinh and N. Sibony.
\newblock Equidistribution towards the {G}reen current for holomorphic maps.
\newblock {\em Ann. Sci. \'Ec. Norm. Sup\'er. (4)}, 41(2):307--336, 2008.

\bibitem[DS3]{DS3}
T.-C. Dinh and N. Sibony.
\newblock Equidistribution problems in complex dynamics of higher dimension. 
\newblock Internat. J. Math. 28 (2017), no. 7, 1750057


\bibitem[DS3]{DS_SURVEY}
T.-C. Dinh and N. Sibony.
\newblock Dynamics in several complex variables: endomorphisms of projective
  spaces and polynomial-like mappings.
\newblock In {\em Holomorphic dynamical systems}, volume 1998 of {\em Lecture
  Notes in Math.}, pages 165--294. Springer, Berlin, 2010.


\bibitem[dT]{dT}
H. de~Th{\'e}lin.
\newblock Sur la construction de mesures selles.
\newblock {\em Ann. Inst. Fourier (Grenoble)}, 56(2):337--372, 2006.



\bibitem[Du]{Du}
R. Dujardin.
\newblock Laminar currents and birational dynamics.
\newblock {\em Duke Math. J.}, 131(2):219--247, 2006.

\bibitem[Fa]{FAVRE_THESIS}
C. Favre.
\newblock  Dynamique des applications rationnelles.
\newblock  Ph.D. Thesis, Université Paris-Sud XI, Orsay, 2000.




\bibitem[FaGu]{FAVRE_GUEDJ}
C. Favre and V. Guedj.
\newblock Dynamique des applications rationnelles des espaces multiprojectifs.
\newblock {\em Indiana Univ. Math. J.} 50 (2001), no. 2, 881–934.


\bibitem[FaJ]{FAVRE_JONSSON}
C.~Favre and M.~Jonsson.
\newblock Brolin's theorem for curves in two complex dimensions.
\newblock {\em Ann. Inst. Fourier (Grenoble)}, 53(5):1461--1501, 2003.



\bibitem[FS1]{FS}
J.~E. Fornaess and N. Sibony.
\newblock Complex dynamics in higher dimension. {II}.
\newblock In {\em Modern methods in complex analysis (Princeton, NJ, 1992)},
  v. 137 of {\em Ann. of Math. Stud.}, pages 135--182. Princeton Univ.
  Press, Princeton, NJ, 1995.

\bibitem[FS2]{FS2}
J.~E. Fornaess and N. Sibony.
\newblock Complex Hénon mappings in $\mathbb{C}^2$ and Fatou-Bieberbach domains.
\newblock {\em Duke Math. J.} 65 (1992), no. 2, 345–380.


\bibitem[FLM]{FLM}
A. Freire, A. Lopes, and R. Mañé.
\newblock An invariant measure for rational maps. 
{\em Bol. Soc. Brasil. Mat.} 14(1):45–62, 1983.



\bibitem[G1]{GREEN:PAMS}
M.~L. Green.
\newblock The hyperbolicity of the complement of {$2n+1$} hyperplanes in
  general position in {$P\sb{n}$} and related results.
\newblock {\em Proc. Amer. Math. Soc.}, 66(1):109--113, 1977.

\bibitem[G2]{GREEN:AJM}
M.~L. Green.
\newblock Some {P}icard theorems for holomorphic maps to algebraic varieties.
\newblock {\em Amer. J. Math.}, 97:43--75, 1975.




\bibitem[Gu1]{GUEDJ_VOL1}
Vincent Guedj.
\newblock Equidistribution towards the {G}reen current.
\newblock {\em Bull. Soc. Math. France}, 131(3):359--372, 2003.

\bibitem[Gu2]{GUEDJ_VOL2}
Vincent Guedj.
\newblock Decay of volumes under iteration of meromorphic mappings.
\newblock {\em Ann. Inst. Fourier (Grenoble)}, 54(7):2369--2386 (2005), 2004.


\bibitem[Gu3]{GUEDJ_LARGE_DEGREE}
V. Guedj.
\newblock Ergodic properties of rational mappings with large topological
  degree.
\newblock {\em Ann. of Math. (2)}, 161(3):1589--1607, 2005.


\bibitem[Gu4]{GUEDJ_ERGODIC}
V. Guedj.
\newblock Entropie topologique des applications m\'eromorphes.
\newblock {\em Ergodic Theory Dynam. Systems}, 25(6):1847--1855, 2005.



\bibitem[HaPo]{HARVEY_POLKING}
R. Harvey and J. Polking.
\newblock Extending analytic objects.
\newblock {\em Comm. Pure Appl. Math.}, 28(6):701--727, 1975.



\bibitem[HPS]{HPS}
M.~Hirsch, C.~Pugh, and M.~Shub.
\newblock {\em Invariant Manifolds}.
\newblock {Lecture notes in mathematics}, Volume 583. Springer-Verlag, Berlin-New York 1977.


\bibitem[HP]{HUBBARD_PAPADAPOL}
J.H.~Hubbard and P. Papadapol.
\newblock Superattractive fixed points in $\mathbb{C}^n$.
\newblock {\em Indiana Univ. Math. J.}, 43(1):321--365, 1994.

\bibitem[KR]{KR}
S. Kaschner and R.K.W. Roeder.
\newblock Superstable manifolds of invariant circles and codimension-one B\"ottcher functions.
\newblock {\em Ergodic Theory Dynam. Systems}, 35(1):152--175, 2015.


\bibitem[K]{KOB_HYPERBOLIC}
S. Kobayashi.
\newblock {\em Hyperbolic complex spaces}, volume 318 of {\em Grundlehren der
  Mathematischen Wissenschaften [Fundamental Principles of Mathematical
  Sciences]}.
\newblock Springer-Verlag, Berlin, 1998.



\bibitem[Kra]{KRANTZ}
S.~G. Krantz.
\newblock {\em Function theory of several complex variables}.
\newblock AMS Chelsea Publishing, Providence, RI, 2001.
\newblock Reprint of the 1992 edition.



\bibitem[La]{LA_HYPERBOLIC}
S. Lang.
\newblock {\em Introduction to complex hyperbolic spaces}.
\newblock Springer-Verlag, New York, 1987.


\bibitem[Le]{LELONG}
P.~Lelong.
\newblock {\em Fonctions plurisousharmoniques et formes diff\'erentielles
  positives}.
\newblock Gordon \& Breach, Paris, 1968.


\bibitem[Ly1]{LYUBICH:NOTE}
M.~Ju. Lyubich.
\newblock The measure of maximal entropy of a rational endomorphism of the Riemann sphere.
\newblock  {\em Funct. Anal. and Appl.}, 16:78 - 79, 1982.


\bibitem[Ly2]{LYUBICH:MAX ENT}
M.~Ju. Lyubich.
\newblock Entropy properties of rational endomorphisms of the {R}iemann sphere.
\newblock {\em Ergodic Theory Dynam. Systems}, 3(3):351--385, 1983.


\bibitem[PM]{PDM}
J. Palis, Jr. and W. de~Melo.
\newblock {\em Geometric theory of dynamical systems}.
\newblock Springer-Verlag, New York, 1982.
\newblock An introduction, Translated from the Portuguese by A. K. Manning.



\bibitem[Pr]{PROTIN}
F. Protin
\newblock Équidistribution vers le courant de Green.
\newblock Ann. Polon. Math. 115 (2015), no. 3, 201-218.


\bibitem[dR]{DERHAM}
G. de~Rham.
\newblock {\em Vari\'et\'es diff\'erentiables. {F}ormes, courants, formes
  harmoniques}.
\newblock Actualit\'es Sci. Ind., no. 1222 = Publ. Inst. Math. Univ. Nancago
  III. Hermann et Cie, Paris, 1955.


\bibitem[RS]{RUELLE_SULLIVAN}
D. Ruelle and D. Sullivan.
\newblock Currents, flows and diffeomorphisms.
\newblock {\em Topology}, 14(4):319--327, 1975.




\bibitem[RuSh]{RUSS_SHIFF}
A. Russakovskii and B. Shiffman.
\newblock Value distribution for sequences of rational mappings and complex
  dynamics.
\newblock {\em Indiana Univ. Math. J.}, 46(3):897--932, 1997.



\bibitem[Shaf]{SHAF}
I.~R. Shafarevich.
\newblock {\em Basic algebraic geometry. 1}.
\newblock Springer-Verlag, Berlin, second edition, 1994.
\newblock Varieties in projective space, Translated from the 1988 Russian
  edition and with notes by Miles Reid.



\bibitem[Si]{S_PANORAME}
N. Sibony.
\newblock Dynamique des applications rationnelles de {$\bold P\sp k$}.
\newblock In {\em Dynamique et g\'eom\'etrie complexes (Lyon, 1997)}, volume~8
  of {\em Panor. Synth\`eses}, pages ix--x, xi--xii, 97--185. Soc. Math.
  France, Paris, 1999.


\bibitem[T]{TAFLIN}
J. Taflin.
\newblock Equidistribution speed towards the Green current for endomorphisms of $\mathbb{P}^k$.
\newblock Adv. Math. 227 (2011), no. 5, 2059-2081.

\bibitem[U]{Ueda}
T. Ueda.
\newblock Fatou sets in complex dynamics on projective spaces.
\newblock {\em J. Math. Soc. Japan}, 46(3):545--555, 1994.


\end{thebibliography}

\begin{thebibliography}{10}

\bibitem[Ba]{Baxter} R.J. Baxter.
\newblock {\em Exactly solvable models in statistical mechanics.} 
\newblock Academic Press, London, 1982. 



\bibitem[BL]{BL}
P. Bleher and M. Lyubich, The Julia sets and complex singularities in hierarchical Ising
models, {\it Commun. Math. Phys.} {\bf 141} (1992), 453--474.

\bibitem[BZ1]{BZ1}
P. Bleher and E. \v Zalys, Existence of long-range order in the Migdal recursion equations,
{\it Commun. Math. Phys.} {\bf 67} (1979), 17--42.

\bibitem[BZ2]{BZ2}
P. Bleher and E. \v Zalys, Limit Gibbs distributions for the Ising model on hierarchical
lattices, {\it Lithuanian Math. J.} {\bf 28} (1989), 127-139.


\bibitem[BZ3]{BZ3}
P. Bleher and E. \v Zalys,
Asymptotics of the susceptibility for the Ising model on the
hierarchical lattices, {\it Commun. Math. Phys.} {\bf 120} (1989), 409--436.


\bibitem[BK]{Brascam_and_Kunz}
H.~J. Brascamp and H. Kunz,
\newblock Zeros of the Partition Function for the
Ising model in the complex temperature plane,
\newblock {\em J. Math. Phys.} {\bf 15} (1974), 65-66.


\bibitem[DDI]{DDI}
B. Derrida, L. De Seze, and C. Itzykson,
Fractal structure of zeros in hierarchical models,
{\it J. Statist. Phys.} {\bf 33} (1983), 559--569.


\bibitem[DIL]{DIL}
B. Derrida, C. Itzykson, and J.~M. Luck,
Oscillatory critical amplitudes in hierarchical models,
{\it Commun. Math. Phys.} {\bf 94} (1984), 115--132.


\bibitem[F]{Fis0}
M.~E. Fisher,
\newblock The Nature of Critical Points,
\newblock In {\em Lectures in Theoretical Physics}, Volume~7c, 
(W. Brittin editor) pages 1-157, University of
Colorado Press, Boulder, 1965.


\bibitem[vH]{VANHOVE}
L. Van-Hove,  Quelques propi\'et\'es g\'en\'erales de l'int\'egral de
configuration d'un syst\`em de particles avec interaction, 
{\it Physica} {\bf 15}, (1949) 951-961.

\bibitem[Ish]{Ish}
Y. Ishii,
Ising models, Julia sets, and similarity of the maximal entropy measures,
{\it J. Statist. Phys.} {\bf 78} (1995), 815--822.


\bibitem[K]{Kad}
L.~P. Kadanoff,
Notes on Migdal's recursion formulae, {\it Ann. Phys.} {\bf 100} (1976), 359--394.


\bibitem[KG1]{KG1}
M. Kaufman and R.~B. Griffiths,
Exactly soluble Ising models on hierarchical lattices,
{\it Phys. Rev. B} {\bf 24} (1981), 496--498.


\bibitem[LY]{LY}
T.~D. Lee and C.~N. Yang, Statistical theory of equations of state and phase transitions: II. Lattice gas
and Ising model. {\it Phys. Rev.} {\bf  87} (1952), 410-419.



\bibitem[MaSh1]{Matveev_Shrock1}
V.~Matveev and R.~Shrock, Complex-temperature singularities in the {$d=2$} {I}sing model: triangular and honeycomb lattices,
{\it  J. Phys. A: Math. Gen.} {\bf 29} (1996) 803-823.


\bibitem[MaSh2]{Matveev_Shrock2}
V.~Matveev and R.~Shrock, Complex-temperature properties of the {I}sing model on {$2$}{D} heteropolygonal lattices 
{\it  J. Phys. A: Math. Gen.} {\bf 28} (1995) 5235-5256.








\bibitem[R2]{Ruelle_book}
D. Ruelle.
\newblock {\em Statistical mechanics}. Rigorous results.
\newblock World Scientific Publishing Co. Inc., River Edge, NJ, 1999.
\newblock Reprint of the 1989 edition.


\bibitem[YL]{YL}
C.~N. Yang and T.~D. Lee,
Statistical theory of equations of state and phase transitions. I. Theory of condensation,
{\it Phys. Rev.} {\bf 87} (1952), 404--409.


\end{thebibliography}
\end{document}